\documentclass[11pt,reqno]{amsart}

\usepackage{fourier,fullpage}
\usepackage{amsmath,amsfonts,amssymb,mathrsfs}
\usepackage{graphicx,extpfeil}

\usepackage{indentfirst, latexsym, enumerate}
\usepackage{float}
\usepackage{epsfig,subfigure}
\usepackage{mathtools}

\usepackage{colortbl}

\usepackage{caption}
\usepackage{bm}
\usepackage{enumitem}

\title[Snowflakes of Carnot groups into Euclidean space]{Embedding snowflakes of Carnot groups into bounded dimensional Euclidean spaces with optimal distortion}

\author[Seung-Yeon Ryoo]{Seung-Yeon Ryoo}
\address[Seung-Yeon Ryoo]{\newline Mathematics Department, Princeton University, 
\newline Princeton, New Jersey 08544-1000, United States} \email{sryoo@math.princeton.edu}

\newtheorem{theorem}{Theorem}[section]
\newtheorem{lemma}[theorem]{Lemma}
\newtheorem{corollary}[theorem]{Corollary}
\newtheorem{proposition}[theorem]{Proposition}

\newtheorem{remark}[theorem]{Remark}

\newtheorem{question}[theorem]{Question}

\newcommand{\bbr}{\mathbb R}

\newcommand{\bbh}{\mathbb H}
\newcommand{\bbs}{\mathbb S}
\newcommand{\bbz}{\mathbb Z}

\newcommand{\bbn} {\mathbb N}

\begin{document}

\date{\today}


\thanks{Acknowledgment: This work will be part of a doctoral dissertation under the supervision of Professor Assaf Naor at Princeton University. I thank him for suggesting the problem of establishing the results of \cite{tao2021embedding} in the setting of general Carnot groups, and for helpful discussions. This work was partially supported by the Korea Foundation for Advanced Studies.}

\begin{abstract}
We show that for any Carnot group $G$ there exists a natural number $D_G$ such that for any $0<\varepsilon<1/2$ the metric space $(G,d_G^{1-\varepsilon})$ admits a bi-Lipschitz embedding into $\mathbb{R}^{D_G}$ with distortion $O_G(\varepsilon^{-1/2})$. We do this by building on the approach of T. Tao (2021), who established the above assertion when $G$ is the Heisenberg group using a new variant of the Nash--Moser iteration scheme combined with a new extension theorem for orthonormal vector fields. Beyond the need to overcome several technical issues that arise in the more general setting of Carnot groups, a key point where our proof departs from that of Tao is in the proof of the orthonormal vector field extension theorem, where we incorporate the Lov\'{a}sz local lemma and the concentration of measure phenomenon on the sphere in place of Tao's use of a quantitative homotopy argument.
\end{abstract}
\maketitle \centerline{\date}


\section{Introduction}
\setcounter{equation}{0}
A map $f:(X,d_X)\to(Y,d_Y)$ between two metric spaces $(X,d_X)$ and $(Y,d_Y)$ is said to have distortion at most $D$ if there exists a constant $C>0$ such that
\[
Cd_X(x_1,x_2)\le d_Y(f(x_1),f(x_2))\le CDd_X(x_1,x_2),\quad \mathrm{for~all~} x_1,x_2\in X.
\]
For $0<\varepsilon<1$, the $(1-\varepsilon)$-snowflake of a metric space $(X,d_X)$ is defined to be the metric space $(X,d_X^{1-\varepsilon})$ (it is clear that $d_X^{1-\varepsilon}$ also defines a metric on $X$).

The metric spaces that we will focus on in this paper are Carnot groups. Following \cite{le2017primer}, a Carnot group is a 5-tuple $(G,\delta_\lambda,\Delta,\|\cdot\|,d_G)$, where:
\begin{itemize}[leftmargin=*]
    \item The Lie group $G$ is a stratified group, i.e., $G$ is a simply connected Lie group whose Lie algebra $\mathfrak{g}$ admits the direct sum decomposition
    \[
    \mathfrak{g}=V_1\oplus V_2\oplus\cdots\oplus V_s,
    \]
    where $V_{s+1}=0$ and $V_{r+1}=[V_1,V_r]$ for $r=1,\cdots,s$.
    \item For each $\lambda\in \bbr^+$, the linear map $\delta_\lambda:\mathfrak{g}\to \mathfrak{g}$ is defined by
    \[
    \left.\delta_\lambda\right|_{V_r}=\lambda^i \mathrm{id}_{V_r},\quad r=1,\cdots,s.
    \]
    \item The bundle $\Delta$ over $G$ is the extension of $V_1$ to a left-invariant subbundle:
    \[
    \Delta_p\coloneqq (dL_p)_e V_1,\quad p\in G.
    \]
    \item The norm $\|\cdot\|$ is initially defined on $V_1$, and is then extended to $\Delta$ as a left-invariant norm:
    \[
    \|(dL_p)_e(v)\|\coloneqq \|v\|,\quad p\in G,~v\in V_1.
    \]
    \item The metric $d_G$ on $G$ is the \textit{Carnot-Carath\'eodory distance} associated to $\Delta$ and $\|\cdot\|$, i.e.,
    \[
    d_G(p,q)\coloneqq \inf\left\{\int_0^1 \|\dot{\gamma}(t)\|dt : \gamma\in C_{\mathrm{pw}}^\infty ([0,1];G),\gamma(0)=p,\gamma(1)=q,\dot{\gamma}\in \Delta \right\},\quad p,q\in G,
    \]
    where $C_{\mathrm{pw}}^\infty ([0,1];G)$ consists of the piecewise smooth functions from $[0,1]$ to $G$.
\end{itemize}

One of the simplest examples of a noncommutative Carnot group is the Heisenberg group $\bbh^3$. It was shown in \cite{tao2021embedding} that, for $0<\varepsilon<1/2$, one can embed the snowflake $(\bbh^{3},d_{\bbh^{3}}^{1-\varepsilon})$ into a bounded dimensional Euclidean space with optimal distortion $O(\varepsilon^{-1/2})$ (more precisely, Theorem \ref{mainthm} below for the case $G=\bbh^3$ was proven in \cite{tao2021embedding}).\footnote{Following widespread convention, $A\lesssim B$, $A=O(B)$, and $B=\Omega(A)$ mean $A\le C B$ for a universal constant $C$, and $A\asymp B$ means $(A\lesssim B) \wedge (B\lesssim A)$. If the constant $C$ depends on other parameters, this is denoted using subscripts, e.g., $A\lesssim_{C_0,N_0}B$, $A=O_{C_0,N_0}(B)$, and $B=\Omega_{C_0,N_0}(A)$ mean $A\le C(C_0,N_0)B$ where $C(C_0,N_0)$ depends only on $C_0$ and $N_0$, and $A\asymp_{C_0}B$ means $A\lesssim_{C_0}B$ and $B\lesssim_{C_0}A$.} The goal of this paper is to show that the methods of \cite{tao2021embedding} extend\footnote{When comparing \cite{tao2021embedding} and this paper, one should keep in mind that the vectorfields and metrics of \cite{tao2021embedding} are right-invariant, whereas this paper uses left-invariant vectorfields and metrics.} to the more general setting of Carnot groups.

\begin{theorem}\label{mainthm}
For each Carnot group $G$, there exists a natural number $D_G$ such that for every $0<\varepsilon<1/2$ there exists an embedding of $(G,d_G^{1-\varepsilon})$ into $\bbr^{D_G}$ with distortion $O_G(\varepsilon^{-1/2})$.
\end{theorem}

We have not attempted to optimize $D_G$, but from the analysis in this paper it will be clear that $D_G$ is at least $\Omega(23^{n_h})$, where $n_h$ is the Hausdorff dimension of $G$, i.e., $n_h=\sum_{r=1}^s  r \dim V_r$.

We will not work in the large epsilon regime $\frac{1}{2}\le \varepsilon <1$ because Theorem \ref{mainthm} is simply not true in this case; see \cite{assouad1983plongements} for the notion of metric dimension which tells us why such a result is impossible. Furthermore, we will assume $0<\varepsilon<\frac 1A$, where $A$ is a very large number. For $\frac 1A\le \varepsilon<\frac{1}{2}$, a construction of \cite{assouad1983plongements} gives such an embedding; we are thus only interested in the small epsilon regime.

One standard consequence of the above theorem, which was also observed in \cite{tao2021embedding}, is the following corollary.
\begin{corollary}
Let $G$ be a Carnot group, and suppose $\Gamma\subset G$ is a discrete subgroup of $G$, where for any two distinct points $\gamma_1,\gamma_2\in \Gamma$ one has $d_G(\gamma_1,\gamma_2)\ge 1$. Let $D_G$ be as in Theorem \ref{mainthm}, and for $R\ge 2$ define the discrete ball $B_{\Gamma}(0,R)\coloneqq\{\gamma\in \Gamma:d_G(0,\gamma)<R\}$. Then there exists an embedding of the discrete ball $B_{\Gamma}(0,R)$ with the induced metric $d_G$ into $\bbr^{D_G}$ of distortion $O_G(\log ^{1/2}R)$.
\end{corollary}
This follows from Theorem \ref{mainthm} because on $B_\Gamma(0,R)$ with $R\ge 2$, the metric $d_G$ is comparable to $d_G^{1-1/\log R}$.

Some history behind Theorem \ref{mainthm}: Carnot groups are special cases of doubling metric spaces, where a metric space $(X,d_X)$ is said to be $K$-doubling for some natural number $K$ if for any $x\in X$ and $R>0$ there exist points $y_1,\cdots,y_K\in X$ such that
\[
B_R(x)\subset \bigcup_{j=1}^K B_{R/2}(y_j),
\]
where $B_R(x)\coloneqq\{z\in X:d_X(x,z)<R\}$ is the ball of radius $R$ centered at $x$; we say that $(X,d_X)$ is doubling if $(X,d_X)$ is $K$-doubling for some natural number $K$. In a seminal paper, Assouad \cite{assouad1983plongements} showed that for $0<\varepsilon<\frac 12$, the $(1-\varepsilon)$-snowflake of a $K$-doubling metric space admits an embedding into $\ell_2^{O_K(\varepsilon^{-O(1)})}$ with distortion $O_K(\varepsilon^{-1/2})$. Here, the distortion $O(\varepsilon^{-1/2})$ is sharp for the Heisenberg group $\bbh^3$; see \cite[Section 4]{naor2010assouad} for a proof of this fact. There is a lot of relevant literature on how to sharpen Assouad's theorem in different ways, with or without various assumptions; see, for example, Abraham, Bartal, and Neiman \cite{abraham2008embedding}, Bartal, Recht, and Schulman \cite{bartal2011dimensionality}, Gottlieb and Krauthgamer \cite{gottlieb2015nonlinear}, and Gupta, Krauthgamer, and Lee \cite{gupta2003bounded}, but we do not mean this list to be extensive. The direction that is quite relevant to this paper is that of \cite{naor2010assouad}, which showed that one can construct snowflake embeddings of any doubling metric space into bounded dimensional Euclidean spaces at the cost of slightly worsening the distortion: more precisely, one can embed the $(1-\varepsilon)$-snowflake of a $K$-doubling space into $\ell_2^{O(\log K)}$ with distortion $O(\frac{\log^2 K}{\varepsilon^2})$, or more generally, for any $0<\delta\le 1$, into $\ell_2^{O(\frac{\log K}{\delta})}$ with distortion $O((\frac{\log K}{\varepsilon})^{1+\delta})$. One may then further inquire whether we can take the target dimension to be uniformly bounded in $0<\varepsilon<\frac 12$ while simultaneously attaining the best possible distortion $O(\varepsilon^{-1/2})$. We state this separately as a question.
\begin{question}[Assouad's theorem with optimal distortion and bounded dimension]\label{weakconj}
For any natural number $K\ge 2$, does there exist a natural number $D(K)$ such that if $(X,d_X)$ is a $K$-doubling metric space and $0<\varepsilon<\frac 12$, then there exists an embedding of the snowflake $(X,d_X^{1-\varepsilon})$ into $\bbr^{D(K)}$ with distortion $O_K(\varepsilon^{-1/2})$?
\end{question}
For completeness, we state a stronger version of the above question, which is motivated by the fact that we must have the lower bound $D(K)\gtrsim \log K$.
\begin{question}[Assouad's theorem with optimal distortion and optimal dimension]\label{strongconj}
For any natural number $K\ge 2$, does there exist a natural number $D(K)=O(\log K)$ such that if $(X,d_X)$ is a $K$-doubling metric space and $0<\varepsilon<\frac 12$, then there exists an embedding of the snowflake $(X,d_X^{1-\varepsilon})$ into $\bbr^{D(K)}$ with distortion $O_K(\varepsilon^{-1/2})$?
\end{question}
These questions are relevant to finding counterexamples to a question raised in \cite{lang2001bilipschitz}, which, in our notation, can be stated as follows.
\begin{question}[Lang--Plaut problem \cite{lang2001bilipschitz}]\label{Lang-Plaut}
For any natural number $K\ge 2$, does there exist a natural number $D(K)$ such that if $X$ is a subspace of $\ell_2$ that is $K$-doubling under the metric it inherits from $\ell_2$, then there exists an embedding of $(X,\|\cdot\|_2)$ into $\bbr^{D(K)}$ with $O_K(1)$ distortion?
\end{question}
As observed in \cite{naor2010assouad}, if an embedding as in Question \ref{weakconj} fails to exist for the Heisenberg group $\bbh^3$, then the Lang--Plaut problem would be answered in the negative, since it is known that $\bbh^3$ admits $(1-\varepsilon)$-snowflake embeddings into $\ell_2$ with distortion $O(\varepsilon^{-1/2})$ with the additional property that the doubling constant of the image is uniformly bounded. Nevertheless, it was shown in \cite{tao2021embedding} that an embedding as in Question \ref{weakconj} for $\bbh^3$ exists, i.e., Theorem \ref{mainthm} for $G=\bbh^3$ is true. Thus, the Heisenberg group (or more precisely, the doubling $\ell_2$ images of the snowflakes thereof) fails to serve as a counterexample to the Lang--Plaut problem. 

The following question was posed in \cite{naor2010assouad} to highlight this connection.
\begin{question}\label{snowflake-with-doubling-image}
For any natural number $K\ge 2$, does there exist a natural number $K'(K)$ such that if $(X,d_X)$ is a $K$-doubling metric space and $0<\varepsilon<\frac 12$, then there exists an embedding of the snowflake $(X,d_X^{1-\varepsilon})$ into $\ell_2$ with distortion $O_K(\varepsilon^{-1/2})$ with the additional property that image of $X$ is $K'(K)$-doubling?
\end{question}
Of course, the connection is that if Question \ref{snowflake-with-doubling-image} has a positive answer while Question \ref{weakconj} has a negative answer, then the Lang--Plaut problem (Question \ref{Lang-Plaut}) must have a negative answer.

The main purpose of this paper is to expand upon our partial knowledge of Question \ref{weakconj}, by answering it in the affirmative in the setting of Carnot groups. We begin by carrying the methods of \cite{tao2021embedding} into the setting of Carnot groups, and whenever a tool of that paper becomes inadequate in the setting of Carnot groups, we replace it with a tool more suitable in the general language of doubling metric spaces. More specifically, the key tools of \cite{tao2021embedding} are a new variant of the Nash--Moser iteration scheme and a certain extension theorem for orthonormal vector fields using a quantitative homotopy argument. When generalizing to the case of Carnot groups, one potential source of trouble is that an arbitrary Carnot group might have an arbitrarily large step size $s$, which could complicate its geometric properties and make the tools of \cite{tao2021embedding} fail. It will be shown in Section \ref{sec:NM-perturb} that the Nash--Moser iteration scheme directly generalizes to the setting of Carnot groups, and we prove a orthogonality statement slightly stronger than that of \cite{tao2021embedding}. However, in Section \ref{sec:ON-ext}, it will be clear that there are some obstructions to the quantitative homotopy argument. Nevertheless, we will prove the orthonormal vector field extension theorem even for general doubling metric spaces, using the concentration of measure phenomenon for the sphere and the Lov\'{a}sz Local Lemma. 

Ultimately, we would like to answer Question \ref{weakconj} for general doubling metric spaces; this paper is an intermediate step in such an investigation, showing that it is at least true for Carnot groups. We plan to address the case of general doubling metric spaces in future work, by possibly adapting some of the proof methods of this paper. In particular, the orthonormal vector field extension theorem for doubling metric spaces (Theorem \ref{Lip-lifting}) seems to be a good starting point for future work, and we would have to find either a metric analog or a replacement for the Nash--Moser iteration scheme.

We now briefly overview some of the results and strategies of \cite{assouad1983plongements,naor2010assouad,tao2021embedding} for constructing snowflake embeddings, and describe how these ideas connect to the proof strategy of this paper.

The starting point of constructing snowflake embeddings is the classical fact that if $X$ is a metric space, $A>1$, $0<\varepsilon<\frac 12$, and $\{\phi_m:X\to \ell_2\}_{m\in \bbz}$ is a collection of maps such that $|\phi_m|\le A^m$ and $\phi_m$ is 1-Lipschitz, then the Weierstrass sum
\begin{equation}\label{Weierstrass}
\Phi=\sum_{m\in\bbz}A^{-m\varepsilon}\phi_m
\end{equation}
(which is easily seen to be absolutely convergent, say by requiring $\phi_m(x_0)=0$ for all $m\in \mathbb{Z}$ for some fixed $x_0\in X$) is $(1-\varepsilon)$-H\"older, with the H\"older norm bounded by $O(\varepsilon^{-1})$. If one assumes in addition that the maps $\phi_m$ take values in mutually orthogonal subspaces of $\ell_2$, then we can bound the $(1-\varepsilon)$-H\"older norm by $O(\varepsilon^{-1/2})$. In the case where $X$ is a doubling metric space, Assouad \cite{assouad1983plongements} constructed functions $\phi_m$ with the above properties and with the additional property that if $d(p,q)\asymp A^m$ then $|\phi_m(p)-\phi_m(q)|\gtrsim_{K} A^m$. Then $\Phi$ satisfies the H\"older lower bound $|\Phi(p)-\Phi(q)|\gtrsim_{K} d(p,q)^{1-\varepsilon}$ for any $p,q\in X$, and thus is an embedding of the $(1-\varepsilon)$-snowflake of $X$ into $\ell_2$ with distortion $O_K(\varepsilon^{-1/2})$ (we do not stress dependence on $A$ for now, as Assouad chooses $A$ to be a constant).

When improving upon Assouad's theorem to reduce the target dimension (especially when trying to keep the target dimension independent of $\varepsilon$), one usually keeps the idea of using the Weierstrass sum \eqref{Weierstrass} to guarantee the upper bound but needs to be more clever to enforce the lower bound. Note that for $A^{m-1}\le d(p,q)\le A^m$, the sum $\sum_{n<m}A^{-n\varepsilon}(\phi_n(p)-\phi_n(q))$ is negligible compared to $d(p,q)^{1-\varepsilon}$,\footnote{More precisely, we would have to take $A^{m-1+\delta}\le d(p,q)\le A^{m+\delta}$ for a fixed small $\delta>0$, while taking the length scale $A$ sufficiently large, but for clarity of the introduction we will not address this issue for now. See the proof of Proposition \ref{AlmostLipschitz} for a precise argument. We also remark that the length scale $A$ will be chosen last in our ``hierarchy of constants'' (see subsection \ref{sec:hierarchy}), so $A$ will dominate every other parameter used in this paper.} so in order to enforce the lower bound it is enough to have $|\sum_{n\ge m}A^{-n\varepsilon}(\phi_n(p)-\phi_n(q))|\gtrsim d(p,q)^{1-\varepsilon}$. Thus, it is natural to devise an iterative construction: having constructed maps $\{\phi_n:X\to\bbr^D\}_{n>m}$ such that the partial sum $\sum_{n>m}A^{-n\varepsilon}\phi_n$ is a map that is able to distinguish points which are at least distance $A^{m+1}$ apart, by being a map that ``oscillates'' at scale $A^{m+1}$ and above, we need to devise a component $A^{-m\varepsilon}\phi_m:X\to\bbr^D$ which ``oscillates'' at scale $A^{m}$ and which, when added to $\psi:=\sum_{n>m}A^{-n\varepsilon}\phi_n$, makes the sum $A^{-m\varepsilon}\phi_m+\psi$ able to distinguish points which are at least distance $A^{m}$ apart. Thus, the challenge is that given the map $\psi:X\to\bbr^D$, we need to make use of a limited amount of coordinates in constructing a map $\phi_m$, so that although $\psi$ and $\phi_m$ ``share coordinates'', the sum $A^{-m\varepsilon}\phi_m+\psi$ satisfies the required H\"older lower bound.

In \cite{naor2010assouad}, the maps $\phi_m$ are defined via a probabilistic construction: after they construct random partitions arising from nets (defined in Section \ref{sec:MG}) with good probabilistic padding properties, they define the maps $\phi_m$ using the distance to the boundaries of these partitions, and by a nested use of the Lov\'{a}sz Local Lemma conditioned on the partial sum $\sum_{n>m}A^{-n\varepsilon}\phi_n$, they show that their maps $\phi_m$ have the desired property. This construction gives distortion $O(\varepsilon^{-1-\delta})$, where the $\varepsilon^{-1}$ factor comes from the H\"older constant of \eqref{Weierstrass}, and the additional factor $\varepsilon^{-\delta}$ comes from some technicalities of the construction, namely that one needs to introduce a slightly finer length scale than $A$ when constructing the nets to ensure that the Lov\'{a}sz Local Lemma is applicable and thus the H\"older lower bound is achieved.

The result of \cite{naor2010assouad} was surprising at the time since it was the first to achieve a target dimension that is uniformly bounded in the amount of snowflaking. It is also surprising that they have achieved the theoretically best possible target dimension $\log K$, up to constant factors. However, \cite{naor2010assouad} fell short of giving a final answer to the problem of snowflake embeddings, since the distortion they have achieved was $O(\varepsilon^{-1-\delta})$, not $O(\varepsilon^{-1/2})$. They have asked whether it is even possible, citing that if Question \ref{weakconj} is answered in the negative for $\bbh^3$, then Question \ref{Lang-Plaut} would be answered in the negative. Surprisingly, Tao \cite{tao2021embedding} subsequently answered Question \ref{weakconj} in the affirmative for $\bbh^3$, and thus overturned the above method of answering Question \ref{Lang-Plaut}.

In order to see how one can achieve the optimal distortion $O(\varepsilon^{-1/2})$, let us revisit \eqref{Weierstrass}. When working with a bounded number of dimensions, one needs to work with a notion of orthogonality weaker than requiring the $\phi_m$ to have pairwise orthogonal ranges. In fact, one can see that we do not need to require every pair of the $\phi_m$ to take values in mutually orthogonal spaces; instead, we need only that each $\phi_m$ be orthogonal in some sense to the partial sum $\sum_{n>m}A^{-n\varepsilon}\phi_n$ (recall that we don't need to consider the partial sum $\sum_{n<m}A^{-n\varepsilon}\phi_n$ when considering points $p,q\in X$ with $d(p,q)\asymp A^m$, as the sum $\sum_{n<m}A^{-n\varepsilon}(\phi_n(p)-\phi_n(q))$ is negligible compared to $A^{m(1-\varepsilon)}\asymp d(p,q)^{1-\varepsilon}$ in this case). Following \cite{tao2021embedding}, we will choose the notion of orthogonality to be that the horizontal derivatives are perpendicular: if $\nabla$ denotes the horizontal gradient in the Carnot group $G$ (see subsection \ref{sec:carnotgeom} for the definition of horizontal gradient), then we require the orthogonality $\nabla \phi_m\cdot \nabla \sum_{n>m}A^{-n\varepsilon}\phi_n=0$ (actually we will derive a slightly stronger orthogonality condition; see Section \ref{sec:NM-perturb}). Overall, the $\phi_m$ are constructed in an iterative manner, and this orthogonality is added to one of the conditions that such an iterative construction must achieve.

Note that this orthogonality easily guarantees the H\"older upper bound of $\Phi$, because an iterated use of the Pythagorean theorem gives
\[
\left|\nabla \sum_{n\ge m}A^{-n\varepsilon}\phi_n\right|^2 = \sum_{n\ge m}A^{-2n\varepsilon}|\nabla \phi_n|^2 \lesssim_A \varepsilon^{-1}A^{-2n\varepsilon},
\]
(we will have $|\nabla \phi_n|=O(1)$); to get the H\"older lower bound, one needs to guarantee that the partial sum $\sum_{n\ge m}A^{-n\varepsilon}\phi_n$ ``oscillates'' at scale $A^m$. In general, proving the H\"older lower bound is more complicated than proving the H\"older upper bound, as the geometry of the particular space $X$ under consideration plays a more significant role.

For Carnot groups $G$, we can use Taylor expansions to exploit the local geometry of $G$ in the following way. Fixing a basis $X_{r,1},\cdots, X_{r,k_r}$ of each stratum $V_r$, each point $p\in G$ can be expressed in the coordinates $p=\exp\left(\sum_{r=1}^s\sum_{i=1}^{k_r}x_{r,i}X_{r,i}\right)$, while its distance to the identity is roughly $d_G(p,0)\asymp_G\sum_{r=1}^s\sum_{j=1}^{k_r}|x_{r,j}|^{1/r}$ (the notation will be explained in full detail in subsection \ref{sec:carnotgeom}). Under this coordinate system, any $C^{s(s+1)}$-function $\Psi:G\to\mathbb{R}^D$ can be approximated by its Taylor expansion as
\[
\Psi(p)=\Psi(0)+\sum_{m=1}^s \frac{1}{m!}\left(\sum_{r=1}^s\sum_{j=1}^{k_r}x_{r,j}X_{r,j}\right)^m\Psi(0)+(\mbox{Taylor approx. error})
\]
(the smallness of the Taylor approximation error can come from a bound on the $C^{s(s+1)}$-norm of $\Psi$ and on a smallness assumption on $d_G(p,0)$).
It will be seen in the proof of Proposition \ref{AlmostLipschitz} that we can rearrange the above into the nicer form
\begin{align*}
    \Psi(p)=&\Psi(0)+\sum_{r=1}^s\sum_{j=1}^{k_r}(x_{r,j}+\mbox{coordinate errors})X_{r,j}\Psi(0)\\
        &+ (\mbox{cross derivative terms})+(\mbox{Taylor approx. error}),
\end{align*}
where the ``coordinate errors" and the ``cross derivative terms" can be made small by a smallness assumption on $d_G(p,0)$ and on control of the second and higher order derivatives of $\Psi$. Thus, if we also have quantitative linear independence of the derivatives $X_{r,j}\Psi(0)$ (see subsection \ref{sec:linalg} for the notion of quantitative linear independence; we will later refer to this as a ``freeness property''), then we can deduce
\begin{align*}
    |\Psi(p)-\Psi(0)|\gtrsim &\max_{1\le r\le s, 1\le j\le k_r}|(x_{r,j}+\mbox{coordinate errors})X_{r,j}\Psi(0)|-(\mbox{errors})\\
    \gtrsim &d_G(p,0)^s \min_{1\le r\le s, 1\le j\le k_r}|X_{r,j}\Psi(0)| -(\mbox{errors}).
\end{align*}
We remark that this is one point where our proof becomes more complicated due to the more general Carnot group structure, compared to the setting of $G=\mathbb{H}^3$ in \cite{tao2021embedding}. But this, so far, does not necessitate a fundamental change of techniques.

To deduce the H\"older lower bound for our embedding $\Phi$ at scale $A^0=1$, we will apply the above approximation to $\Psi = \sum_{n\ge 0}A^{-n\varepsilon}\phi_n$, and consider points at a slightly smaller scale, more precisely $d_G(p,0)\asymp A^{-1/(s+1)}$. (This uses the additional fact that we effectively only need a H\"older lower bound for part of the scales, as long as we can cover the entire set of scales by a finite set of rescalings. See \eqref{fullmapping}.) Thus, to enforce the H\"older lower bound, we will have to justify the above approximation and give quantitative bounds on the various errors, namely we will need estimates on the $C^{s(s+1)}$-norm of $\Psi$ and quantitative linear independence of the derivatives $X_{r,j}\Psi(0)$.

Thus, all in all, the iterative step can be stated as follows: upon rescaling so that $m=0$, and defining $\psi=\sum_{n> 0}A^{-n\varepsilon}\phi_n$, the main problem is,  given a function $\psi:G\to \bbr^D$, which ``oscillates'' at scale $A$ (this will be measured by the rescaled $C^m$ norms which will be introduced in subsection \ref{sec:funcspace}), which is a $C^{s(s+1)}$-function, and whose derivatives are quantitatively linearly independent, to construct a function $\phi:G\to \bbr^D$, which ``oscillates'' at scale 1, which is a $C^{s(s+1)}$-function, and  with the orthogonality property $\nabla \phi\cdot \nabla \psi=0$ along with a quantitative linear independence for the derivatives of $\psi+\phi$. The precise statement and proof of the iterative step will be presented in Section \ref{sec:iterationlemma}.

Once we have the iterative step as above, the rest of the inductive construction is fairly easy and is given in Section \ref{sec:applyiteration}. This inductive construction mostly follows the line of \cite{tao2021embedding}, but we have spelled out the details for completeness. In order to begin the iteration, we need a single function $\phi_m$ to begin with, which oscillates at a fixed scale $A^m$, is of class $C^{s(s+1)}$, and whose derivatives are quantitatively linearly independent (we will call this a ``freeness'' property). This is done in Proposition \ref{FirstStep}. It is easy to construct a smooth function with the first two properties, and then we can guarantee the latter freeness property by a Veronese-type embedding often employed in the Nash embedding literature (and also used in \cite{tao2021embedding}). Once we have this single function, we employ the above iterative step a finite number of times to obtain a finite family of mappings $\{\phi_m:G\to\mathbb{R}^D\}_{M_1\le m\le M_2}$ (Proposition \ref{FiniteIteration}), and then use the Arzel\`a-Ascoli theorem to pass to a full family of mappings $\{\phi_m:G\to\mathbb{R}^D\}_{m\in \mathbb{Z}}$ (Theorem \ref{lacunary}). There is a small issue of losing one degree of regularity when using the Arzel\`a-Ascoli theorem, but this can easily be fixed by requiring $C^{s^2+s+1}$-regularity in the first place.

We devote the rest of the introduction to explaining the proof ideas behind the above iterative step.

A na\"ive approach to solving the equation\footnote{In Section \ref{sec:NM-perturb} instead of $\nabla \phi\cdot \nabla \psi=0$ we will solve the stronger equation $X_i\phi\cdot X_j\psi+X_j\phi\cdot X_i\psi=0$ for $i,j=1,\cdots,k$. We are stating the simpler version to keep the Introduction simple.} would be to use the Leibniz rule directly: denoting by $X_1,\cdots, X_k$ a left-invariant vectorfield basis for $V_1$ so that $\nabla = (X_1,\cdots,X_k)$, the Leibniz rule tells us that $X_i\phi \cdot X_i\psi = X_i(\phi\cdot X_i\psi)-\phi\cdot X_iX_i\psi$, so in order to solve
\[
X_i\phi\cdot X_i\psi=0, \quad i=1,\cdots,k,
\]
we could simply solve
\[
\phi\cdot X_i\psi=\phi\cdot X_iX_i\psi=0,\quad  i=1,\cdots,k
\]
(Note that the trivial solution $\phi$=0 is not acceptable, as we need a freeness property of $\psi+\phi$ at scale 1.) This latter equation is at least algebraically well-posed in the sense that the vectors $X_i\psi$ and $X_iX_i\psi$, $i=1,\cdots,k$, are (quantitatively) linearly independent; however it seems inherent that $\phi$ must be in the same regularity class as $X_i\psi$ and $X_iX_i\psi$; i.e., $\phi$ must have two fewer degrees of regularity compared to $\psi$. This ``loss of derivatives problem'' would make the iteration fail.

The solution proposed by \cite{tao2021embedding} to get around this loss of derivatives problem was to introduce Littlewood--Paley projections $P_{(\le N_0)}$ on $\bbh^3$ (see subsection \ref{sec:LP} for Littlewood--Paley projections), and first solve the approximate equation 
\[
X_i\tilde{\phi}\cdot X_iP_{(\le N_0)}\psi=0,\quad i=1,\cdots,k,
\]
by solving
\[
\tilde{\phi}\cdot X_iP_{(\le N_0)}\psi=\tilde{\phi}\cdot X_iX_iP_{(\le N_0)}\psi=0,\quad i=1,\cdots,k.
\]
Because lower order derivatives of $\psi$ can control higher order derivatives of $P_{(\le N_0)}\psi$ (up to losses growing on the order of the Littlewood--Paley frequency $N_0$; see Theorem \ref{LP} for the precise quantitative statement), and because we are looking at a smaller scale $1$ for $\phi$ compared to the larger scale $A$ for $\psi$, it is plausible that one can achieve as much control on $\tilde{\phi}$ as $\psi$, as long as we take $A$ very large compared to $N_0$. More precisely, because $\psi$ oscillates at scale $A$, its homogeneous $\dot{C}^m$-norm will behave like $A^{-m}$, so $\nabla^2 P_{(\le N_0)}\psi$ will have $\dot{C}^m$-norm roughly $A^{-m-2}$ for $m\le s^2+s-2$ and $A^{-s^2-s}N_0^{m- s^2-s+2}$ for $m> s^2+s-2$, which are all $O(1)$ if $A$ is chosen sufficiently larger than $N_0$. This will ensure that any reasonable construction based on $X_iP_{(\le N_0)}\psi$ and $X_iX_iP_{(\le N_0)}\psi$ will produce a function that ``oscillates'' at scale $1$ and is of class $C^{s^2+s+1}$, and thus has the same level of regularity as the function $\tilde{\phi}$ that we want to produce.

We will show in subsection \ref{sec:LP} that Carnot groups admit Littlewood--Paley projections that share the same properties as those used in \cite{tao2021embedding}.

Thus, the main idea of \cite{tao2021embedding} to solve the equation $\nabla \phi\cdot \nabla \psi =0$ is to decompose it into two steps. In the first step, we need to solve the low-frequency version of the equation $\nabla \tilde{\phi}\cdot \nabla P_{(\le N_0)}\psi =0$ using the Leibniz rule, i.e., we solve the equation 
\[
\tilde{\phi}\cdot X_iP_{(\le N_0)}\psi=0, ~\tilde{\phi}\cdot X_iX_iP_{(\le N_0)}\psi=0,\quad i=1,\cdots,k,
\]
while guaranteeing that $\tilde{\phi}$ has the same regularity as $XP_{(\le N_0)}\psi$ and $XXP_{(\le N_0)}\psi$.
In the second step, we would assume the low-frequency solution $\tilde{\phi}$ as given, and then we would need to correct $\tilde{\phi}$ into a ``true'' solution $\phi$ while guaranteeing that $\phi$ has the same regularity as $\tilde{\phi}$.

The first step, namely solving the approximate equation $\tilde{\phi}\cdot X_iP_{(\le N_0)}\psi=\tilde{\phi}\cdot X_iX_iP_{(\le N_0)}\psi=0$, with $\tilde{\phi}$ having the same regularity as $X_iP_{(\le N_0)}\psi$ and $X_iX_iP_{(\le N_0)}\psi$, is the part where our proof in the setting of general Carnot groups departs the most from the setting of $\mathbb{H}^3$ in \cite{tao2021embedding}, and is the part where the main novelty of this paper lies in. For $\bbh^3$, \cite{tao2021embedding} resolves this issue by first observing that $\bbh^3$ has a cocompact lattice and a CW structure that is compatible with it, and then provides the extension using quantitative null homotopy on each of the cells of that CW structure (note that, although the null homotopy is used on an infinite number of cells, the final construction is globally controlled, because up to left-translation there is only a finite number of distinct cells). This argument may not work for a general Carnot group $G$, as $G$ might not admit a cocompact lattice (say if the structure constants for every basis of $G$ were irrational). Instead, we give a proof, using the concentration of measure phenomenon on the sphere and the Lov\'{a}sz Local Lemma, that is independent of the topology of the space under consideration, by using only the fact that $G$ is a doubling metric space (see Section \ref{sec:ON-ext} for details). We essentially do not need the differential structure of $G$ because the uniform continuity of $\tilde{\phi}$ is the main hurdle here; we can automatically gain higher regularity by convolving with a mollifier and using the Gram-Schmidt process.

To put our solution to the approximate equation into context, we will pose, in Section \ref{sec:ON-ext}, a more general question (Question \ref{general-lift}), which asks whether one can extend a given orthonormal system of vectorfields within the same regularity class. It will become clear that our solution provides a partial positive answer to Question \ref{general-lift}, at least when the base space is a doubling metric space, the regularity class is contained in the Lipschitz class, the regularity class is closed under simple algebraic operations, and if smoothening a Lipschitz vectorfield by convolving it with some scalar mollifier provides the resulting vectorfield with the desired regularity (see Theorem \ref{partialpositiveanswer} for details).

For the second step, once we have the approximate solution $\tilde{\phi}$, we need to correct it into a true solution $\phi$ while preserving the regularity. Tao \cite{tao2021embedding} solved this by developing a perturbative theory for the bilinear form $\nabla \phi\cdot\nabla \psi$. More precisely, \cite{tao2021embedding} develops a version of the Nash--Moser iteration scheme to show how one can correct $\tilde{\phi}$ by small amounts into a true solution $\phi$ to $\nabla \phi\cdot\nabla \psi=0$, without losing any regularity (technically, the Nash--Moser iteration scheme necessitates that we work in the H\"older class $C^{m,\alpha}$, $m\ge 3$, $\alpha\in (\frac 12,1)$ instead of the usual $C^m$ class, and we will thus have to accommodate for this for the rest of this paper, but this does not affect any of the arguments made so far). This method of solving the orthogonality equation $\nabla \phi\cdot \nabla\psi=0$ only requires us to look at first and second derivatives. We will show in Section \ref{sec:NM-perturb} that generalizing this Nash--Moser iteration argument of \cite{tao2021embedding} from the case of $\bbh^3$ to the broader setting of Carnot groups does not incur serious difficulties even if the step size of $G$ is greater than 2. We will also show that by modifying the techniques of \cite{tao2021embedding}, one can solve the stronger orthogonality equation
\[
X_i\phi\cdot X_j\psi+X_j\phi\cdot X_i \psi =0,\quad i,j=1,\cdots,k,
\]
but we will not be able to obtain the stronger orthogonality equation
\[
X_i\phi\cdot X_j\psi=0,\quad i,j=1,\cdots,k.
\]
See Proposition \ref{Perturbation} and Corollary \ref{perturbation-cor} for the precise statement. One can find a detailed description and motivation of this Nash--Moser iteration scheme in the introduction to \cite{tao2021embedding}.

In light of the proof method for Theorem \ref{mainthm} described above, it is natural to ask whether these methods can be generalized further. Many of these methods depend on the fact that Carnot groups are the tangent spaces of themselves (recall that Carnot groups arise as tangent spaces of sub-Riemannian manifolds). It would be necessary to revamp many of the ideas here, especially the Nash--Moser iteration scheme, to be applicable to the setting of doubling metric spaces. Even the orthonormal vectorfield extension theorem (Theorem \ref{Lip-lifting}) would require some reformulation since it is unlikely that we will have vectorfields in the setting of doubling metric spaces; at the least, the vectorfields should be replaced by elements of a Grassmannian. One realistic hope is that one could transform Theorem \ref{Lip-lifting} into a higher-dimensional block basis variant of the construction of \cite{naor2010assouad} and thus construct an embedding with distortion $O(\varepsilon^{-1/2-\delta})$ using the random net construction of that work. One could also improve upon the results of \cite{naor2010assouad} by using hierarchical nets since these may be easier to control and describe compared to ordinary nets (see \cite{har2006fast} for the construction and applications of hierarchical nets).

The rest of this paper is organized as follows. We first introduce some elementary background in Section \ref{sec:prelim}, and develop the Nash--Moser perturbation theorem in Section \ref{sec:NM-perturb} and the orthonormal extension theorem in Section \ref{sec:ON-ext}. It is in Section \ref{sec:ON-ext} that the proof idea differs the most from \cite{tao2021embedding}: whereas that work proves the orthonormal extension theorem in the spirit of quantitative topology, we will prove it using the concentration of measure phenomenon on the sphere and the Lov\'{a}sz local lemma. We then develop the main iteration lemma in Section \ref{sec:iterationlemma} and show how it gives us our desired embedding in Section \ref{sec:applyiteration}.


\section{Preliminaries}\label{sec:prelim}
\setcounter{equation}{0}

\subsection{Hierarchy of constants}\label{sec:hierarchy}

We will select absolute constants in the following order:
\begin{itemize}[leftmargin=*]
    \item A H\"older exponent $\alpha\in (\frac 12,1)$ and a level of regularity $m^*$ depending on $G$. For simplicity, one can fix $\alpha=\frac 23$ and $m^*=s^2+s+1$, where $s$ is the step size of $G$.
    \item A sufficiently large natural number $C_0>1$ depending on $G$, $\alpha$ and $m^*$. Specifically, this choice will occur in \eqref{C_0_hierarchy-1}, the third inequality of \eqref{C_0_hierarchy-2}, the second inequality of \eqref{C_0_hierarchy-3}, \eqref{C_0_hierarchy-4}, right after \eqref{C_0_A_hierarchy-1}, right before \eqref{C_0_A_hierarchy-2}, \eqref{C_0-hierarchy-5}, and right after \eqref{C_0_A_hierarchy-3}.
    \item A sufficiently large dyadic number $N_0$ depending on $G$ and $C_0$. This choice will occur in \eqref{N_0_hierarchy-6}, in the derivation of \eqref{G-9-again} of Corollary \ref{perturbation-cor}, \eqref{N_0_hierarchy-1}, \eqref{N_0_hierarchy-2}, \eqref{N_0_hierarchy-3}, \eqref{N_0_hierarchy-4}, \eqref{N_0_hierarchy-5}, and \eqref{N_0_hierarchy-7}.
    \item A sufficiently large dyadic number $A$ depending on $G$, $C_0$ and $N_0$. This choice will occur in \eqref{A_hierarchy-2}, \eqref{A_hierarchy_1}, \eqref{A_hierarchy-3}, \eqref{A_hierarchy-5}, right after \eqref{C_0_A_hierarchy-1}, right before \eqref{C_0_A_hierarchy-2}, right after \eqref{C_0_A_hierarchy-3}, right after \eqref{A_hierarchy-6}, \eqref{A_hierarchy-7}, \eqref{A_hierarchy-8}, and \eqref{A_hierarchy-9}.
\end{itemize}

\subsection{Basic linear algebra}\label{sec:linalg}
Denote by $|\cdot |$ the Euclidean metric and by $\langle \cdot ,\cdot \rangle $ the Euclidean inner product on Euclidean spaces $\mathbb{R}^D$.

If $T:\bbr^{D_1}\to \bbr^{D_2}$ is a linear map, we also denote by $|T|$ the Frobenius norm of $T$. Also, for $1\le n\le D$, the exterior power $\bigwedge^n \bbr^D$ can be identified with $\bbr^{ \binom{D}{n}}$, and so we can also define a Euclidean norm $|\cdot |$ and a Euclidean inner product $\langle \cdot ,\cdot \rangle $ on $\bigwedge^n \bbr^D$. With this norm on $\bigwedge^n \bbr^D$, the \textit{Cauchy--Binet formula} tells us that  for $v_1,\cdots,v_n\in \bbr^D$,
\[
|v_1\wedge \cdots\wedge v_n|^2=\det\left(v_i\cdot v_j\right)_{1\le i,j\le n}=\det (TT^*),
\]
where $T:\bbr^D\to\bbr^n$ is the linear map
\[
T(u)\coloneqq (u\cdot v_1,\cdots,u\cdot v_n).
\]
More generally, the \textit{polarized Cauchy--Binet formula} tells us that for $u_1,\cdots, u_n,v_1,\cdots,v_n\in \bbr^D$,
\[
\bigg<u_1\wedge \cdots\wedge u_n,v_1\wedge \cdots\wedge v_n\bigg>=\det\left(u_i\cdot v_j\right)_{1\le i,j\le n}.
\]
It is not difficult to see that we have a Cauchy--Schwarz-like inequality: for every $v_1,\cdots,v_n\in\bbr^D$ and $1\le i<n\le D$, we have
\[
|v_1\wedge \cdots\wedge v_n|\le |v_1\wedge \cdots\wedge v_i||v_{i+1}\wedge \cdots\wedge v_n|.
\]
We will simply refer to this as the Cauchy--Schwarz inequality in the rest of this paper.

\subsection{Some metric space geometry}\label{sec:MG}
Let $(X,d)$ be a metric space. For any $f:X\to\bbr^D$, we define the norms
\[
\|f\|_{C^0}\coloneqq \sup_{x\in X}|f(x)|,\quad \|f\|_{\mathrm{Lip}}\coloneqq \sup_{x,y\in X,~x\neq y}\frac{|f(x)-f(y)|}{d(x,y)}.
\]
These norms satisfy certain algebraic properties. If $f,g:X\to\bbr^D$ then
\begin{equation}\label{first-alg}
    \|f\cdot g\|_{\mathrm{Lip}}\le \|f\|_{C^0}\|g\|_{\mathrm{Lip}}+\|f\|_{\mathrm{Lip}}\|g\|_{C^0}.
\end{equation}
Also, if $f:X\to\bbr$ and $f(x)\ge c>0$ for all $x\in X$ then
\begin{align}
    \|1/f\|_{\mathrm{Lip}}\le c^{-2}\|f\|_{\mathrm{Lip}},\label{reciprocal}\\
    \|\sqrt{f}\|_{\mathrm{Lip}}\le \frac{1}{2\sqrt{c}}\|f\|_{\mathrm{Lip}}.\label{squareroot}
\end{align}
One can use these properties to see that for $f:X\to\bbr^D$ with $|f(x)|\ge c>0$ for all $x\in X$ we have
\begin{equation}\label{unit-alg}
    \left\|\frac{f}{|f|}\right\|_{\mathrm{Lip}}\le \left\|\frac{1}{|f|}\right\|_{C^0}\left\|f\right\|_{\mathrm{Lip}}+\left\|\frac{1}{|f|}\right\|_{\mathrm{Lip}}\left\|f \right\|_{C^0}\le (c^{-1}+c^{-2}\left\|f \right\|_{C^0})\left\|f\right\|_{\mathrm{Lip}}.
\end{equation}

For $\delta>0$, a subset $\mathcal{N}_\delta \subset X$ is a \textit{$\delta$-net} if for any distinct $x,y\in \mathcal{N}_\delta$ one has $d(x,y)\ge \delta $. By Zorn's lemma, $\delta$-nets which are maximal with respect to inclusion exist, and if $\mathcal{N}_\delta$ is a maximal $\delta $-net then we have the covering
\[
X=\bigcup_{x\in \mathcal{N}_\delta}B_\delta(x).
\]
An immediate consequence of the doubling property is that if $X$ is a $K$-doubling metric space and $m\ge 0$, then for any $\delta$-net $\mathcal{N}_\delta$ we have
\begin{equation}\label{doubling-net}
    |\mathcal{N}_\delta\cap B_{2^m \delta}(x)|\le K^{m+1}\quad \mathrm{for~all~} x\in X.
\end{equation}
\subsection{Function spaces on Carnot groups}\label{sec:funcspace}
We will assume that the norm $\|\cdot\|$ on $V_1$ is an inner product on $V_1$ (in other words, we may assume $G$ is a sub-Riemannian Carnot group, as opposed to being a sub-Finsler Carnot group). This is possible by John's ellipsoid theorem, which allows us to replace $\|\cdot\|$ by an inner product norm while introducing distortion at most $\sqrt{k}$, which is acceptable since this is independent of the amount $\varepsilon$ of snowflaking.

We fix a left-invariant orthonormal basis $X_1,\cdots,X_k$ of $V_1$ with respect to $\|\cdot\|$. If $\phi:G\to \bbr^D$ is a differentiable function, we let $\nabla \phi:G\to \bbr^{kD}$ denote the horizontal gradient
\[
\nabla \phi\coloneqq (X_1\phi,\cdots, X_k\phi).
\]
By iteration, we have $\nabla^m\phi:G\to \bbr^{k^mD}$ for any $m\ge 1$, if $\phi$ is $m$ times differentiable. We recall the $C^0$ norm
\[
\|\phi\|_{C^0}= \sup_{p\in G}|\phi(p)|,
\]
and define the higher $C^m$ norms
\[
\|\phi\|_{C^m}\coloneqq \sum_{0\le j\le m} \|\nabla^j \phi\|_{C^0}.
\]
For a fixed spatial scale $R>0$, we define the $C^m_R$ norm to be the rescaled norm
\[
\|\phi\|_{C^m_R}\coloneqq \sum_{0\le j\le m} R^j\|\nabla^j \phi\|_{C^0}.
\]

Given a H\"{o}lder exponent $0<\alpha <1$ we may also define the homogeneous H\"{o}lder norm
\[
\|\phi\|_{\dot{C}^{0,\alpha}}\coloneqq \sup_{p,q\in G,p\neq q}\frac{|\phi(p)-\phi(q)|}{d(p,q)^\alpha}
\]
and the higher H\"{o}lder norms
\[
\|\phi\|_{C^{m,\alpha}}\coloneqq \|\phi\|_{C^m}+\|\nabla^m \phi\|_{\dot{C}^{0,\alpha}}
\]
and more generally, the rescaled H\"{o}lder norm
\[
\|\phi\|_{C^{m,\alpha}_R}\coloneqq \|\phi\|_{C^m_R}+R^{m+\alpha}\|\nabla^m \phi\|_{\dot{C}^{0,\alpha}}.
\]
One may easily verify
\[
\|\phi\|_{C_R^{m,\alpha}}\lesssim \|\phi\|_{C_R^{m+1}}.
\]

By an iterated application of the product rule, one can verify the algebra properties
\begin{align*}
\|\phi\psi\|_{C^m_R}&\lesssim_m \|\phi\|_{C^m_R}\|\psi\|_{C^m_R},\\
\|\phi\psi\|_{C^{m,\alpha}_R}&\lesssim_m \|\phi\|_{C^{m,\alpha}_R}\|\psi\|_{C^{m,\alpha}_R}.
\end{align*}
These inequalities continue to hold when $\phi$ and $\psi$ are vector-valued and we take the wedge product or the dot product, where the constants do not depend on the dimension of the codomain of $\phi$ and $\psi$:
\begin{alignat*}{3}
\|\phi\cdot \psi\|_{C^m_R}\lesssim_m \|\phi\|_{C^m_R}\|\psi\|_{C^m_R}&,\quad \|\phi\cdot \psi\|_{C^{m,\alpha}_R}&\lesssim_m \|\phi\|_{C^{m,\alpha}_R}\|\psi\|_{C^{m,\alpha}_R},\\
\|\phi\wedge \psi\|_{C^m_R}\lesssim_m \|\phi\|_{C^m_R}\|\psi\|_{C^m_R}&,\quad \|\phi\wedge \psi\|_{C^{m,\alpha}_R}&\lesssim_m \|\phi\|_{C^{m,\alpha}_R}\|\psi\|_{C^{m,\alpha}_R}.
\end{alignat*}

More generally, one can observe that these algebra properties continue to hold when we replace the above norms with norms of the form 
\[
\|\phi\|_{C^m_{\{R_j\}_{i=0}^m}}\coloneqq \sum_{0\le j\le m} R_j\|\nabla^j \phi\|_{C^0},
\]
where $\{R_j\}_{i=0}^m$ is a ``sequence of spatial scales'', i.e., a sequence of positive real numbers, that is log-concave: $R_{i}R_j\ge R_{i+j}$. Examples of such norms include
\begin{equation}\label{logconcave}
\|\phi\|+R\|\nabla \phi\|_{C^m},~\mbox{or }\|\phi\|_{C^0}+\|\nabla \phi\|_{C^m_{1/R}},\quad R\ge 1.
\end{equation}

\subsection{Some Carnot group geometry}\label{sec:carnotgeom}

Recall the decomposition $\mathfrak{g}=V_1\oplus V_2\oplus\cdots \oplus V_s$. We will define $\dim G=n$, $\dim V_r=k_r$, $k=k_1$, and the Hausdorff dimension $n_h\coloneqq \sum_{r=1}^s r k_r$. We will assume $s\ge 2$, since if $s=1$, then $G$ is just a finite-dimensional Euclidean space, and near-optimal snowflake embeddings of Euclidean spaces were constructed in \cite{assouad1983plongements}. This will give $n_h\ge 4$, as we must have $k_1\ge 2$ and $k_2\ge 1$.

For $2\le r\le s$, we fix a basis $X_{r,1},\cdots, X_{r,k_r}$ of $V_r$ and extend them to left-invariant vectorfields over $G$. For $r=1$, we simply write $X_{1,i}=X_i$.

As $G$ is nilpotent and simply connected, the exponential map $\exp:\mathfrak{g}\to G$ is a diffeomorphism. Recall that we have defined the scaling maps $\delta_\lambda:\mathfrak{g}\to \mathfrak{g}$ for $\lambda>0$ by
\[
\left.\delta_\lambda\right|_{V_r}=\lambda^i \mathrm{id}_{V_r},\quad r=1,\cdots,s.
\]
One may then define the dilation $\delta_\lambda:G\to G$ so that it commutes with $\exp$:
\[
\delta_\lambda \circ \exp=\exp\circ \delta_\lambda.
\]
One can compute that $\delta_\lambda$ is the unique Lie group automorphism $\delta_\lambda:G\to G$ such that $(\delta_\lambda)_*=\delta_\lambda$.
Moreover, $\delta_\lambda$ interacts with the left-invariant vector fields as follows:
\[
X_{r,i}(\phi \circ \delta_\lambda)=\lambda^r(X_{r,i} \phi)\circ \delta_\lambda,\quad r=1,\cdots, s,~i=1,\cdots, k_r.
\]
The special case $r=1$ tells us that $\delta_\lambda$ is a scaling in the Carnot-Carath\'eodory metric:
\[
d_{G}(\delta_\lambda(p),\delta_\lambda(p'))=\lambda d_{G}(p,p'),\quad p,p'\in G.
\]
By iteration, we can also deduce
\[
\nabla^m (\phi\circ \delta_\lambda)=\lambda^m (\nabla^m \phi)\circ \delta_\lambda.
\]

One can parametrize $G$ by $\bbr^n$, by first identifying $G$ with $\mathfrak{g}$ via the exponential map $\exp$, and then identifying $\mathfrak{g}$ with $\bbr^n$ via the basis $\{X_{r,i}\}_{1\le r\le s,1\le i\le k_r}$. We will denote the corresponding canonical basis as $\{f_{r,i}\}_{1\le r\le s,1\le i\le k_r}$.

We may define a weighted degree for polynomials in $x_{r,i}$ by assigning degree $r$ to $x_{r,i}$. It is clear that $\delta_\lambda$ acting upon a homogeneous polynomial of degree $m$ is just multiplication by $\lambda^m$, so the differential operator $X_{r,i}$ acts on polynomials by reducing the weighted degree by $r$ in each term. One can also see, using the scaling $\delta_\lambda$, that
\[
d_G\left(\exp(\sum_{r=1}^s\sum_{i=1}^{k_r}x_{r,i} X_{r,i}),e_G\right)\asymp_G \sum_{r=1}^s\sum_{i=1}^{k_r}|x_{r,i}|^{1/r}.
\]

One can express the group law in this coordinate system using the Baker-Campbell-Hausdorff formula
\[
gh =
\sum_{m = 1}^s\frac {(-1)^{m-1}}{m}
\sum_{\begin{smallmatrix} r_i + s_i > 0 \\ i=1,\ldots,m \end{smallmatrix}}
\frac{[ g^{r_1} h^{s_1} g^{r_2} h^{s_2} \dotsm g^{r_m} h^{s_m} ]}{(\sum_{j = 1}^m (r_j + s_j)) \cdot \prod_{i = 1}^m r_i! s_i!},
\]
where the sum is finite since $G$ is of step $s$, and we have used the notation
\[
[ g^{r_1} h^{s_1} \dotsm g^{r_m} h^{s_m} ] = [ \underbrace{g,[g,\dotsm[g}_{r_1} ,[ \underbrace{h,[h,\dotsm[h}_{s_1} ,\,\dotsm\, [ \underbrace{g,[g,\dotsm[g}_{r_m} ,[ \underbrace{h,[h,\dotsm h}_{s_m} ]]\dotsm]]].
\]
Thus, we can see that
\[
\left(\sum_{r=1}^s\sum_{i=1}^{k_r}x^0_{r,i} X_{r,i}\right)\left(\sum_{r=1}^s\sum_{i=1}^{k_r}x^1_{r,i} X_{r,i}\right)=\left(\sum_{r=1}^s\sum_{i=1}^{k_r}x^2_{r,i} X_{r,i}\right)
\]
where
\[
x^2_{r,i}=x^0_{r,i}+x^1_{r,i}+(\mbox{homogeneous polynomial of }\{x^0_{r',i'}\}_{r'<r},\{x^1_{r',i'}\}_{r'<r}\mbox{ of degree }r).
\]
It follows that the Lebesgue measure on $\bbr^n$ is a Haar measure on $G$ because the Jacobian of left-multiplication is a unit upper triangular matrix. Also, the left-invariant vectors  $\{X_{r,i}\}_{1\le r\le s,1\le i\le k_r}$ in this coordinate system take the form
\begin{equation}\label{vf-form}
    X_{r,i}=\frac{\partial}{\partial x_{r,i}}+\sum_{r'>r}^s\sum_{j=1}^{k_{r'}}(\mbox{homogeneous polynomial of }\{x_{r'',i'}\}_{r''<r'}\mbox{ of degree }r'-r)\frac{\partial}{\partial x_{r',j}}.
\end{equation}

For $r,r'=1,\cdots,s$, $i=1,\cdots, k_{r}$, $i'=1,\cdots,k_{r'}$, we define the lexicographic ordering
\[
(r,i)\preceq (r',i')\quad\Leftrightarrow \quad r<r'~\mbox{or}~r=r'\mbox{ and }i\le i'.
\]

For $r>0$, we define the open ball
\[
B_r\coloneqq \{h\in G:d_G(h,e_G)<r\},
\]
and for $g\in G$ and $r>0$ we define the open ball
\[
B_r(g)\coloneqq \{h\in G:d_G(h,g)<r\}=gB_r,
\]
where the last equality follows from left-invariance of the metric.

A simple volumetric argument gives the following bound for any $\delta$-net $\mathcal{N}_\delta$:
\begin{equation}\label{locbd}
    |\mathcal{N}_\delta\cap B_R(p)|\le \left(\frac{2R}{\delta}+1\right)^{n_h},\quad p\in G, R>0.
\end{equation}

\subsection{Littlewood--Paley theory on Carnot groups}\label{sec:LP}

A basic tool used in \cite{tao2021embedding} was a Littlewood--Paley theory for functions defined on the Heisenberg group. One can easily modify the argument in that paper to show the following. For a positive number $N$ and a $C^0$ function $\phi:G\to \bbr^D$, one can construct the Littlewood--Paley projection $P_{(\le N)}\phi:G\to \bbr^D$, which is a $C^\infty$ function, and the variants
\[
{P_{(<N)}}\phi\coloneqq P_{(\le N/2)}\phi,~ {P_{(N)}}\phi\coloneqq P_{(\le N)}\phi-{P_{(<N)}}\phi,~ P_{(>N)}\phi\coloneqq \phi-P_{(\le N)}\phi,~ P_{(\ge N)}\phi\coloneqq \phi-P_{(< N)}\phi
\]
which satisfy the following regularity properties.

\begin{theorem}[Littlewood--Paley Theory {\cite[Theorem 6.1]{tao2021embedding}}]\label{LP}
Let $\phi:G\to\bbr^D$ be bounded and continuous.
\begin{enumerate}[leftmargin=*]
    \item (scaling) For any $\lambda>0$ and $N>0$, we have
    \[
    P_{(\le N)}(\phi\circ \delta_\lambda)=(P_{(\le N/\lambda)}\phi)\circ \delta_\lambda,
    \]
    and similarly for ${P_{(<N)}}$, ${P_{(N)}}$, $P_{(\ge N)}$, and $P_{(>N)}$.
    \item (Littlewood--Paley decomposition) For any dyadic number $N_0$, we have
    \[
    \phi=P_{(\le N_0)} \phi+\sum_{N>N_0 ~dyadic} {P_{(N)}}\phi.
    \]
    \item (regularity) If $N,M>0$, $j,l\ge 0$, and $\phi$ is of class $C^l$, one has the estimates
    \begin{align}
        \|\nabla^l P_{(\le N)} \phi \|_{C^j_{1/N}}&\lesssim_{G,j,l} \|\nabla^l \phi\|_{C^0},\label{lp-1}\\
        \| P_{(N)} \phi \|_{C^j_{1/N}},\| P_{(>N)} \phi \|_{C^j_{1/N}}&\lesssim_{G,j,l,\alpha} N^{-l}\|\nabla^l \phi\|_{C^0},N^{-l-\alpha}\|\nabla^l \phi\|_{\dot{C}^{0,\alpha}},\label{lp-5}\\
        \| P_{(\le N)} \phi \|_{C^l_{1/M}},\| P_{(N)} \phi \|_{C^l_{1/M}},\| P_{(>N)} \phi \|_{C^l_{1/M}}&\lesssim_{G,l} \|\phi\|_{C^l_{1/M}}.\nonumber
    \end{align}
\end{enumerate}
\end{theorem}

The construction of the Littlewood--Paley projection in \cite{tao2021embedding} is as follows. As the Laplace-Kohn operator
\[
L\coloneqq -\sum_{i=1}^k X_i^2
\]
is self-adjoint on $L^2(G)$, where $G$ is given its Haar measure (which is the Lebesgue measure on $\bbr^{n}$), for any $m\in L^\infty(\bbr)$ one can define bounded operators $m(L)$ on $L^2(G)$ which commute with each other and with $L$. But by a result of \cite{hormander1967hypoelliptic}, $L$ is a hypoelliptic operator, and thus one may apply a result of \cite{hulanicki1984functional} to develop a more refined bounded functional calculus for $L$: if $m\in C_c^\infty(\bbr)$, then this operator is given by convolution with a Schwartz function $K:G\to \bbr$:
\[
    m(L)f=f*K,\quad \mathrm{for~all~} f\in  L^2(G),
\]
where a function on $G$ is said to be a Schwarz function if it is a Schwarz function on $\bbr^n$, and $*$ denotes the convolution operator:
\[
f*K(p)=\int_G f(g)K(g^{-1}p)dg.
\]
For such $m$, the operator $m(L)$ can be extended to functions in $C^0$ using the above convolution formula.

Now, if we choose any smooth function $\varphi:\bbr\to\bbr$ supported on $[-1,1]$ that equals $1$ on \linebreak$[-1/2,1/2]$, we can now define, for $N>0$, the Littlewood--Paley projection $P_{(\le N)}$ using the above functional calculus by the formula
\begin{align*}
    P_{(\le N)} &\coloneqq \varphi(L/N^2).
\end{align*}

The proof of the properties listed in Theorem \ref{LP} is mostly the same as presented in Theorem 6.1 of \cite{tao2021embedding}. When following the proof there, the only part that requires modification in the setting of Carnot groups is the following. For $i=1,\cdots,k$, given any Schwarz function $K:G\to\mathbb{R}$, we need to show that there are Schwarz functions $K_j:G\to\mathbb{R}$, $j=1,\cdots,k$, so that for any $C^1$-function $\phi:G\to\mathbb{R}^D$,
\begin{equation}\label{integration_by_parts}
    \phi*X_iK = \sum_{j=1}^k X_j\phi*K_j.
\end{equation}
This is a consequence of integration by parts:
\[
X\phi*K=-\phi*\tilde{X}K,
\]
where $X$ is any left-invariant vector field and $\tilde{X}$ its right-invariant counterpart. First, we have
\[
\phi*X_iK = -X_i\phi*K+\phi*(X_i-\tilde{X}_i)K.
\]
Now recalling \eqref{vf-form} and its variant for the right-invariant counterparts, we see that
\[
X_i-\tilde{X}_i=\sum_{r=2}^s\sum_{j=1}^{k_r}q_{r,j}\tilde{X}_{r,j},
\]
where $q_{r,j}$ is a polynomial that commutes with $\tilde{X}_{r,j}$. Thus
\[
(X_i-\tilde{X}_i)K=\sum_{r=2}^s\sum_{j=1}^{k_r} \tilde{X}_{r,j}(q_{r,j}K),
\]
and since each $\tilde{X}_{r,j}$ may be written as a linear combination of nested brackets of $\tilde{X}_1,\cdots,\tilde{X}_k$, we have
\[
(X_i-\tilde{X}_i)K=\sum_{j=1}^k \tilde{X}_j K_j
\]
for some Schwarz functions $K_j$, since Schwarz functions are closed under multiplication by polynomials and actions by right-invariant vector fields. Hence
\[
\phi*(X_i-\tilde{X}_i)K= -\sum_{j=1}^k X_j\phi * K_j,
\]
and we have the form \eqref{integration_by_parts}.


\section{Nash--Moser Perturbation for a bilinear form}\label{sec:NM-perturb}
\setcounter{equation}{0}
For two given $C^1$ functions $\phi,\psi:G\to \bbr^D$, we define the bilinear form $B(\phi,\psi):G\to\operatorname{Sym}^2(\bbr^{k})\subset \mathbb{R}^k\otimes \mathbb{R}^k$ as
\begin{equation}\label{Bdef}
B(\phi,\psi)\coloneqq \operatorname{Sym}\left((X_i\phi\cdot X_j\psi)_{i,j=1,\cdots,k}\right).
\end{equation}
Later, when constructing good embeddings of the Carnot group $G$, we will encounter the following situation.
Given $\psi:G\to \bbr^D$ with certain regularity properties, so that $\psi$ ``represents'' the geometry of $G$ at scale $A$ and above, we will need to find a ``nontrivial'' solution $\phi:G\to \bbr^D$ to
\begin{equation}\label{perp}
    B(\phi,\psi)=0,
\end{equation}
so that $X\phi\cdot X\psi=0$ for any horizontal left-invariant vectorfield $X$ of $G$. This way, the Pythagorean theorem will tell us that $|\nabla(\phi+\psi)|^2=|\nabla\phi|^2+|\nabla\psi|^2$, which, coupled with an Assouad-type summation technique, will give us optimal control on the growth of $|\nabla \psi|^2$ and hence provide us with the optimal distortion rate $O(\varepsilon^{-1/2})$ (note that Assouad \cite{assouad1983plongements} achieved this orthogonality and hence the optimal distortion by allowing the $\phi$ and $\psi$ to take values in different direct sum components of the target space, but thereby losing control on the dimension of the target space). Here, when we say that $\phi$ is ``nontrivial'', we mean that $\psi+\phi$ also has the regularity properties of $\psi$ but at scale 1 instead of $A$. Attempts to solve this system \eqref{perp} directly using the Leibniz rule and linear algebra gives less control on the smoothness on $\phi$ than that on $\psi$, which is unsuitable for iteration. The solution proposed by \cite{tao2021embedding} was to first find a nontrivial and approximate solution $\Tilde{\phi}$ to \eqref{perp}, or more precisely a solution to the low-frequency equation
\begin{equation}\label{G-5}
    B(\Tilde{\phi},P_{(\le N_0)}\psi)=0.
\end{equation}
This way, we have control on all levels of smoothness of $P_{(\le N_0)}\psi$ (by Theorem \ref{LP} (3)), and hence also on $\tilde{\phi}$. This $\tilde{\phi}$ will be constructed in later sections. Once we have this approximate solution $\tilde{\phi}$, Tao \cite{tao2021embedding} then proposed to use a variant of the Nash--Moser iteration scheme to find small perturbations of $\tilde{\phi}$, which are small enough to preserve the non-triviality of $\Tilde{\phi}$, and which allows us to solve the original equation \eqref{perp}.

Our goal in this section is mainly to show that the Nash--Moser iteration scheme of \cite{tao2021embedding} carries on to the general setting of Carnot groups without obstruction, while proving a slightly stronger orthogonality statement \eqref{perp} than that of \cite{tao2021embedding}. The rest of this section follows the argument of Section 7 of \cite{tao2021embedding}; we have reproduced the entire argument here to keep track of certain calculations that arise from higher-dimensional matrix operations, as well as to state and verify various estimates in the setting of Carnot groups.

Because the Nash--Moser iteration process produces error terms, we will need to consider a slightly more general setting. 
Given $\psi:G\to \bbr^D$ and $F=(F_{ij})_{i,j=1,\cdots,k}:G\to \operatorname{Sym}^2(\bbr^{k})$, we consider the problem of finding a solution $\phi:G\to \bbr^D$ to
\begin{equation}\label{G-1}
    B(\phi,\psi)=F.
\end{equation}

One easy way to solve \eqref{G-1} is to solve the system
\begin{align*}
\begin{cases}
    \phi\cdot X_i\psi =0,\\
    \phi\cdot X_iX_j\psi =-F_{ij},\\
    \phi\cdot X_{2,i'}\psi =0,
\end{cases}
\quad 1\le i,j\le k,~1\le i'\le k_2,
\end{align*}
for then, since $X_iX_j-X_jX_i\in \mathrm{span}\{X_{2,1},\cdots,X_{2,k_2}\}$, $1\le i,j\le k$, we have
\[
\begin{cases}
    \phi\cdot X_i\psi =0,\\
    \phi\cdot X_iX_j\psi =-F_{ij},
\end{cases}
\quad i,j=1,\cdots, k,
\]
and so
\[
X_i\phi\cdot X_j \psi=X_i(\phi\cdot X_j\psi)-\phi\cdot X_iX_j\psi=F_{ij},\quad i,j=1,\cdots,k.
\]
This system is solvable if $\{X_i\psi\}_{i=1,\cdots,k}\cup\{X_iX_j\psi\}_{1\le i\le j\le k}\cup \{X_{2,i'}\psi\}_{i'=1,\cdots,k_2}$ are pointwise independent.

More precisely, for each $p\in G$ define the linear map $T_\psi(p):\bbr^{D}\to\bbr^{k+\frac{k(k+1)}{2}+k_2}$ by
\begin{align*}
    T_\psi(p)(v)\coloneqq \Big((v\cdot X_i\psi(p))_{1\le i\le k},~(v\cdot X_iX_j\psi(p))_{1\le i\le j\le k},~(v\cdot X_{2,i'}\psi(p))_{1\le i'\le k_2}\Big),\quad v\in \bbr^D.
\end{align*}
If $\{X_i\psi\}_{i=1,\cdots,k}\cup\{X_iX_j\psi\}_{1\le i\le j\le k}\cup \{X_{2,i'}\psi\}_{i'=1,\cdots,k_2}$ are pointwise independent, i.e., if $T_\psi(p)$ has full rank, or equivalently (by the Cauchy--Binet formula) if 
\begin{equation*}
\left|\bigwedge_{i=1}^k X_i\psi(p)\wedge \bigwedge_{1\le i\le j\le k} X_iX_j\psi(p)\wedge \bigwedge_{i'=1}^{k_2} X_{2,i'}\psi(p) \right|>0,
\end{equation*}
then we can define the pseudoinverse $T_\psi(p)^{-1}:\bbr^{k+\frac{k(k+1)}{2}+k_2}\to\bbr^{D}$ of $T_\psi(p)$ by the formula
\[
T_\psi(p)^{-1}\coloneqq T_\psi(p)^*(T_\psi(p)T_\psi(p)^*)^{-1}.
\]
(Note that the linear independence condition necessitates that $D\ge k+\frac{k(k+1)}{2}+k_2$, so the pseudoinverse is well-defined.)
Then for any continuous functions $a_i:G\to\bbr$, $i=1,\cdots,k$, $b_{ij}:G\to\bbr$, $1\le i\le j\le k$, $c_{i'}:G\to\bbr$, $i'=1,\cdots,k_2$, we have the pointwise identities
\begin{align}
\begin{aligned}\label{G-4}
    T_\psi(p)^{-1}(a_1(p),\cdots,a_k(p),b_{11}(p),\cdots,b_{kk}(p),\qquad\qquad\qquad&&&& \\
	c_1(p),\cdots,c_{k_2}(p))\cdot X_i\psi(p) &=a_i(p),&\quad &i=1,\cdots, k,&\\
    T_\psi(p)^{-1}(a_1(p),\cdots,a_k(p),b_{11}(p),\cdots,b_{kk}(p),\qquad\qquad\qquad&&&& \\
	c_1(p),\cdots,c_{k_2}(p))\cdot X_iX_j\psi(p) &=b_{ij}(p),&\quad &1\le i\le j\le k,&\\
    T_\psi(p)^{-1}(a_1(p),\cdots,a_k(p),b_{11}(p),\cdots,b_{kk}(p),\qquad\qquad\qquad&&&&\\
	 c_1(p),\cdots,c_{k_2}(p))\cdot X_{2,i'}\psi(p) &=c_{i'}(p),&\quad &i'=1,\cdots,k_2.&\\
\end{aligned}
\end{align}
If we define
\begin{equation}\label{flip-b}
    b_{ji}=b_{ij}+\sum_{i'=1}^{k_2}\alpha_{i,j,i'}c_{i'},\quad 1\le i< j\le k,
\end{equation}
where $\alpha_{i,j,i'}$ are the structure constants so that
\begin{equation}\label{struc_const}
    X_jX_i-X_iX_j=\sum_{i'=1}^{k_2}\alpha_{i,j,i'}X_{2,i'},\quad 1\le i\le j\le k,
\end{equation}
then we can extend the second equation of \eqref{G-4} to 
\begin{equation}\label{G-4-ext}
    T_\psi(p)^{-1}(a(p),b(p),c(p))\cdot X_iX_j\psi(p) =b_{ij}(p),\quad i, j=1,\cdots, k.
\end{equation}
So, by using the Leibniz rule as above, we have
\begin{equation}\label{G-4-summary}
X_i\Big( T_\psi^{-1}(a,b,c)\Big)\cdot X_j\psi = X_i(a_j)-b_{ij},\quad i, j=1,\cdots, k.
\end{equation}
As a consequence, one has the explicit solution
\begin{equation}\label{explicit}
    \phi_{\mathrm{explicit}}(p)\coloneqq T_\psi(p)^{-1}(0,-F(p),0)
\end{equation}
to \eqref{G-1} (when we are plugging in $F(p)$ to the above expression, we are using a standard identification $\operatorname{Sym}^2(\mathbb{R}^k)\simeq \mathbb{R}^{\frac{k(k+1)}{2}}$).

The problem with this solution to \eqref{G-1} is that the solution $\phi_{\mathrm{explicit}}$ constructed in this manner will have two fewer degrees of regularity than $\psi$, which will be unsuitable for iteration purposes. We will overcome this issue by applying the above procedure to the Littlewood--Paley components of $\psi$ and $F$.

\begin{proposition}[Perturbation theorem, analog of {\cite[Proposition 7.1]{tao2021embedding}}]\label{Perturbation}
Let $M$ be a real number with
\[
M\ge C_0^{-1}.
\]
Let ${m^*}\ge 2$ and $\frac 12<\alpha<1$.
Suppose we are given a $C^{{m^*},\alpha}$-map $\psi:G\to \bbr^{D}$ with the following regularity properties:
\begin{enumerate}[leftmargin=*]
    \item (H\"older regularity at scale $A$) We have
    \begin{equation}\label{psi-4}
        \|\nabla^2 \psi\|_{C_A^{{m^*}-2,\alpha}}\le C_0A^{-1}.
    \end{equation}
    \item (nondegenerate first derivatives) For any $p\in G$, we have
    \begin{equation}\label{psi-1}
        C_0^{-1}M\le |X_i\psi(p)|\le C_0 M,\quad i=1,\cdots,k,
    \end{equation}
    \item (locally free embedding) For any $p\in G$, we have
    \begin{equation}\label{psi-3}
        \left|\bigwedge_{i=1}^k X_i\psi(p)\wedge \bigwedge_{1\le i\le j\le k} X_iX_j\psi(p)\wedge \bigwedge_{i'=1}^{k_2} X_{2,i'}\psi(p) \right|\gtrsim_{C_0} A^{-\frac{k(k+1)}{2}-k_2}M^{k}.
    \end{equation}
\end{enumerate}

Let $F:G\to\operatorname{Sym}^2(\bbr^{k})$ be a function with bounded  $C^{2{m^*}-1}$-norm: $\|F\|_{C^{2{m^*}-1}}<\infty$. Let $\Tilde{\phi}:G\to\bbr^{D}$ be a solution to the low-frequency equation \eqref{G-5} with bounded $C^{{m^*},\alpha}$-norm: $\|\Tilde{\phi}\|_{C^{{m^*},\alpha}}<\infty$. Then there exists a $C^{{m^*},\alpha}$-solution $\phi$ to \eqref{G-1} which is a small perturbation of $\tilde{\phi}$:
\begin{equation}\label{G-8}
    \|\phi-\Tilde{\phi}\|_{C^{{m^*},\alpha}}\lesssim_{C_0} A\|F\|_{C^{2{m^*}-1}}+A^{2-{m^*}-\alpha}\|\Tilde{\phi}\|_{C^{{m^*},\alpha}},
\end{equation}
and which makes the following cross terms small:
\begin{align}\label{G-9}
\begin{aligned}
    &\|X_i\phi\cdot X_j\psi-X_i\Tilde{\phi}\cdot X_jP_{(\le N_0)}\psi\|_{C^0}\\
&\qquad\qquad\qquad\lesssim_{C_0} \|F\|_{C^{2{m^*}-1}}+N_0^{1-{m^*}-\alpha}A^{1-{m^*}-\alpha}\|\Tilde{\phi}\|_{C^{{m^*},\alpha}},\quad i,j=1,\cdots,k.
\end{aligned}
\end{align}
Here, we treat $\alpha$ as a universal constant (we may take $\alpha=\frac 23$), and we allow the constant $C_0$ to depend on ${m^*}$. This will not contradict our hierarchy of constants, as ${m^*}$ will be later chosen to depend on $G$ (more precisely, it will equal $s^2+s+1$).
\end{proposition}
\begin{remark}\label{perturb-rem}
\begin{enumerate}[leftmargin=*]
    \item The bilinear form used in \cite{tao2021embedding} was the simpler version
    \[
    \tilde B(\phi,\psi)=(X_1\phi\cdot X_1\psi,\cdots,X_k\phi\cdot X_k\psi),
    \]
    and the corresponding Nash-Moser iteration scheme only established the weaker orthogonality $\tilde B (\phi,\psi)=0$, while still being able to establish the estimate \eqref{G-9}. We may create a version of Proposition \ref{Perturbation} for $\tilde B$ by assuming the weaker freeness property for \eqref{psi-3}:
    \[
    \left|\bigwedge_{i=1}^k X_i\psi(p)\wedge \bigwedge_{i=1}^k X_iX_i\psi(p) \right|\gtrsim_{C_0} A^{-k}M^{k}.
    \]
    The proof of this weaker Proposition would be not so different from the proof of Proposition \ref{Perturbation} given below, where we would modify the pseudoinverse $T_\psi^{-1}$ and its usage in the obvious way. The statement \eqref{G-9} becomes slightly harder to prove, but one may directly implement the methods of \cite{tao2021embedding}.
    \item One may imagine strengthening the theorem to obtain the stronger full orthogonality statement
\[
X_i\phi\cdot X_j\psi=0,\quad i,j=1\cdots,k,
\]
but attempts to modify the Nash--Moser iteration scheme to accommodate this difference cause the resulting infinite series to diverge (more specifically, we are then forced to place derivatives on ${P_{(N)}}\psi$ in \eqref{new-ab}, which we must avoid in order to make the defining series converge). The best one can achieve with the tools outlined in this section is \eqref{G-9}. Nevertheless, we will be able to choose $\tilde{\phi}$ with
\[
X_i\tilde{\phi}\cdot X_jP_{(\le N_0)}\psi=0,\quad i,j=1\cdots,k,
\]
which, along with \eqref{G-9}, establishes that $X_i\phi\cdot X_j\psi$ is sufficiently small.
\end{enumerate}
\end{remark}

\begin{proof}  It will suffice to find a function $\phi$ with the stated bounds solving the approximate equation
\begin{equation}\label{phi-bounds}
\| B(\phi,\psi) - F \|_{C^{2{m^*}-1}} \le A^{2-{m^*}}\|F\|_{C^{2{m^*}-1}}+A^{3-2{m^*}}\|\tilde{\phi}\|_{C^{{m^*},\alpha}}
\end{equation}
rather than the precise equation \eqref{G-1}, while satisfying \eqref{G-8} and \eqref{G-9}:
\begin{equation}\label{G-8-prime}
    \|\phi-\Tilde{\phi}\|_{C^{{m^*},\alpha}}\lesssim_{C_0} A\|F\|_{C^{2{m^*}-1}}+A^{2-{m^*}-\alpha}\|\Tilde{\phi}\|_{C^{{m^*},\alpha}},
\end{equation}
and
\begin{align}\label{G-9-prime}
\begin{aligned}
    &\|X_i\phi\cdot X_j\psi-X_i\Tilde{\phi}\cdot X_jP_{(\le N_0)}\psi\|_{C^0}\\
	&\qquad\qquad\qquad\lesssim_{C_0} \|F\|_{C^{2{m^*}-1}}+N_0^{1-{m^*}-\alpha}A^{1-{m^*}-\alpha}\|\Tilde{\phi}\|_{C^{{m^*},\alpha}},\quad i,j=1,\cdots,k.
\end{aligned}
\end{align}
One can then iteratively replace $(\tilde \phi, F)$ by the error term $(0, F - B(\phi,\psi))$  and sum the resulting solutions to obtain an exact solution to \eqref{G-1}. This is possible due to the linearity of this equation in $\phi$.

We will first construct a low-frequency solution $\phi_{(\le N_0)}$ to the low-frequency equation
\[
B(\phi_{(\le N_0)}, P_{(\le N_0)} \psi) = P_{(\le N_0)} F
\]
and then, given $\phi_{(\le N/2)}$ for dyadic $N > N_0$ by induction, we will add higher frequency components $\phi_{(N)}$ to obtain $\phi_{(\le N)}:=\phi_{(\le N/2)}+\phi_{(N)}$, which approximately solves the higher-frequency equation
\[
B(\phi_{(\le N)}, P_{(\le N)} \psi) \approx P_{(\le N)} F.
\]

More specifically, the construction goes as follows. We first construct the low-frequency component as
\begin{equation}\label{phi-start}
 \phi_{(\le N_0)} \coloneqq \tilde \phi + T_{P_{(\le N_0)} \psi}^{-1}(0, -P_{(\le N_0)} F, 0).
\end{equation}
This form was chosen so that, from \eqref{G-4-summary}, one has
\begin{equation}\label{low-freq-cross-term}
    X_i(\phi_{(\le N_0)}-\tilde \phi)\cdot X_jP_{(\le N_0)} \psi = P_{(\le N_0)} F_{ij},\quad i,j=1,\cdots,k,
\end{equation}
and from \eqref{G-5} and the bilinearity of $B$, one has
\begin{equation}\label{b-ident-1}
 B(\phi_{(\le N_0)}, P_{(\le N_0)} \psi) = P_{(\le N_0)} F.
\end{equation}

Next, for every dyadic $N > N_0$ we recursively define the higher-frequency component $\phi_{(N)}$ by the formula
\begin{equation}\label{phi-add}
 \phi_{(N)} \coloneqq T_{P_{(\le N)} \psi}^{-1}((a^i_N)_{1\le i\le k},(b^{ij}_N)_{1\le i\le j\le k}, (c^{i'}_N)_{1\le i'\le k_2})
\end{equation}
where
\begin{align}\label{new-ab}
\begin{cases}
\begin{aligned}
a^i_N &\coloneqq - (X_i P_{(\le N)} \phi_{(<N)}) \cdot {P_{(N)}} \psi, \\
b^{ij}_N &\coloneqq - (X_iX_j P_{(\le N)} \phi_{(<N)}) \cdot {P_{(N)}} \psi - {P_{(N)}} F_{ij},\\
c^{i'}_N &\coloneqq - (X_{2,i'}P_{(\le N)}\phi_{(<N)}) \cdot {P_{(N)}} \psi,
\end{aligned}
\end{cases}
\quad 1\le i\le j\le k,~1\le i'\le k_2,
\end{align}
and
\[
\phi_{(<N)} \coloneqq \phi_{(\le N_0)} + \sum_{\substack{N_0 < M < N\\ M ~\mathrm{dyadic}}} \phi_{(M)}.
\]
We will also define
\[
\phi_{(\le N)} \coloneqq \phi_{(\le N_0)} + \sum_{\substack{N_0 < M \le N\\ M~\mathrm{dyadic}}} \phi_{(M)}.
\]
Note that in the definition of $\phi_{(N)}$, no derivatives are placed on ${P_{(N)}} \psi$ and there is some mollification of the $\phi_{(<N)}$ term, in order to avoid the loss of derivatives problem.
This form of $\phi_{(N)}$ was chosen so that
\begin{equation}\label{b-ident-2}
 B(\phi_{(N)}, P_{(\le N)} \psi) + B(P_{(\le N)} \phi_{(<N)}, {P_{(N)}} \psi) = {P_{(N)}} F.
\end{equation}
Indeed, by \eqref{flip-b}, \eqref{struc_const}, and \eqref{G-4-ext}, we have
\[
b^{ij}_N = - (X_iX_j P_{(\le N)} \phi_{(<N)}) \cdot {P_{(N)}} \psi - {P_{(N)}} F_{ij},\quad  i,j=1,\cdots, k,
\]
so we can compute, for $i,j=1,\cdots,k$,
\begin{align}\label{high-freq-cross-term}
\begin{aligned}
&X_i \phi_{(N)}\cdot X_j P_{(\le N)} \psi+X_iP_{(\le N)} \phi_{(<N)}\cdot X_j {P_{(N)}} \psi\\
&\qquad\qquad\qquad\qquad\qquad\qquad\stackrel{\mathclap{\eqref{G-4-summary}}}{=}~\big(X_ia_N^j-b_N^{ij}\big)+X_iP_{(\le N)} \phi_{(<N)}\cdot X_j {P_{(N)}} \psi\\
&\qquad\qquad\qquad\qquad\qquad\qquad=~{P_{(N)}} F_{ij}-X_jP_{(\le N)} \phi_{(<N)}\cdot X_i {P_{(N)}} \psi+X_iP_{(\le N)} \phi_{(<N)}\cdot X_j {P_{(N)}} \psi.
\end{aligned}
\end{align}
Symmetrizing in $i$ and $j$ now gives \eqref{b-ident-2}.

We have the following control on the regularity of $T_{P_{(\le N)} \psi}^{-1}$:

\begin{lemma}[regularity of the pseudoinverse; analog of {\cite[Lemma 7.2]{tao2021embedding}}]\label{tpsi}  For any $N \geq N_0$, one has
\begin{align*}
\| T_{P_{(\le N)} \psi}^{-1} \|_{C^{{m^*}-2}_A} &\lesssim_{C_0} A,\\
\| \nabla^{{m^*}-2} T_{P_{(\le N)} \psi}^{-1} \|_{C^{{m^*}+1}_{1/N}} &\lesssim_{C_0} A^{3-{m^*}}.
\end{align*}
\end{lemma}

\begin{proof}  Abbreviating $T=T_{P_{(\le N)} \psi}$, recall the definition of the pseudoinverse
\[
T^{-1} = \frac{1}{\det(TT^*)} T^* \mathrm{adj}(T T^*)
\]
where $\mathrm{adj}(A)$ denotes the adjugate matrix of $A$.
We then need to show the bounds
\[
\left\| \nabla^l\left(\frac{1}{\det(TT^*)} T^* \mathrm{adj}(T T^*) \right) \right\|_{C^0} \lesssim_{C_0} A B_l
\]
for $0 \le l \le 2{m^*}-1$, where $B_l$ is the log-convex sequence
\[
B_l \coloneqq 
\begin{cases}
A^{-l}, & 0\le l\le {m^*}-2,\\
N^{l-{m^*}+2} A^{2-{m^*}},& l>{m^*}-2.
\end{cases}
\]

From \eqref{psi-4}, \eqref{psi-1}, and Theorem \ref{LP}(3), we have
\begin{alignat*}{4}
|\nabla^l X_i P_{(\le N)} \psi| &\lesssim_{C_0} M B_l,\quad &i=1,\cdots, k,\quad &0 \le l \le 2{m^*}-1,\\
|\nabla^l X_iX_j P_{(\le N)} \psi| &\lesssim_{C_0} A^{-1} B_l,\quad &i,j=1,\cdots, k,\quad &0 \le l \le 2{m^*}-1,\\
|\nabla^l X_{2,i'} P_{(\le N)} \psi| &\lesssim_{C_0} A^{-1} B_l,\quad &i'=1,\cdots, k_2,\quad &0 \le l \le 2{m^*}-1.
\end{alignat*}
Thus, viewing $T$ as a $\left(k+\frac{k(k+1)}{2}+k_2\right) \times D$ matrix, for any $l$-th order horizontal differential operator $W_l$ the first $k$ rows of $W_l T$ have norm $O_{C_0}(M B_l)$, and the bottom $\frac{k(k+1)}{2}+k_2$ have norm $O_{C_0}(A^{-1} B_l)$.  By the product rule, and noting that $B_l B_{l'} \le B_{l+l'}$ for all $l,l' \ge 0$ we conclude that the $\left(k+\frac{k(k+1)}{2}+k_2\right) \times \left(k+\frac{k(k+1)}{2}+k_2\right)$ matrix $W_l(TT^*)$ has top left $k \times k$ block consisting of elements of magnitude $O_{C_0}(M^2 B_l)$, top right $k\times \left(\frac{k(k+1)}{2}+k_2\right)$ and bottom left $\left(\frac{k(k+1)}{2}+k_2\right)\times k$ blocks consisting of elements of magnitude  $O_{C_0}(A^{-1} M B_l)$, and bottom right $\left(\frac{k(k+1)}{2}+k_2\right)\times \left(\frac{k(k+1)}{2}+k_2\right)$ block consisting of elements of magnitude $O_{C_0}(A^{-2} B_l)$.  By the product rule and cofactor expansion, $W_l \mathrm{adj}(TT^*)$ then has top left block of norm $O_{C_0}(M^{2k-2}A^{-k(k+1)-2k_2} B_l)$, top right and bottom left blocks of norm $O_{C_0}(M^{2k-1}A^{-k(k+1)-2k_2+1} B_l)$, and bottom right block of norm\linebreak $O_{C_0}(M^{2k}A^{-k(k+1)-2k_2+2} B_l)$. Again, by the product rule, every row of the $ D\times \left(k+\frac{k(k+1)}{2}+k_2\right)$ matrix $W_l(T^* \mathrm{adj}(TT^*))$ is of norm $O_{C_0}(M^{2k} A^{-k(k+1)-2k_2+1} B_l)$ (many are lower order than this).  

Similarly, $\nabla^l(\det(TT^*))$ has magnitude $O_{C_0}(M^{2k} A^{-k(k+1)-2k_2} B_l)$.  
Meanwhile, from \eqref{psi-1}, \linebreak\eqref{psi-3}, \eqref{psi-4}, and using \eqref{lp-5} to approximate $P_{(\le N)} \psi$ by $\psi$ up to negligible error (note that $N\ge N_0$ and our hierarchy of constants, namely that $N_0$ is chosen after $C_0$), we easily obtain the wedge product lower bound
\begin{equation}\label{N_0_hierarchy-6}
    \left|\bigwedge_{i=1}^k X_iP_{(\le N)}\psi(p)\wedge \bigwedge_{1\le i\le j\le k} X_iX_jP_{(\le N)}\psi(p)\wedge \bigwedge_{i'=1}^{k_2} X_{2,i'}P_{(\le N)}\psi(p) \right|\gtrsim_{C_0} A^{-\frac{k(k+1)}{2}-k_2}M^{k}.
\end{equation}
From this and the Cauchy--Binet formula we have the matching lower bound
\[
\det(TT^*) \gtrsim_{C_0} M^{2k} A^{-k(k+1)-2k_2}
\]
for the determinant.  Hence, by the quotient rule, the derivatives $\nabla^l(\det(TT^*)^{-1})$ have magnitude $O_{C_0}(M^{-2k} A^{k(k+1)+2k_2} B_l)$.
The claim now follows from the product rule.
\end{proof}
\begin{remark}
Of course, Lemma \ref{tpsi} can be strengthened to guarantee not only $C^{2{m^*}-1}$-regularity of $T^{-1}$ but also all higher levels of regularity. We stopped at $C^{2{m^*}-1}$ because this is only what we will need later.
\end{remark}

The rest of the proof of Proposition \ref{Perturbation} follows mostly as in \cite{tao2021embedding}, except for the proof of \eqref{G-9}. We have reproduced the argument for completeness.

First, we would need the smoothness of our proposed solution. We begin with the smoothness of the low-frequency component $\phi_{(\le N_0)}$. From the above Lemma we have the estimate
\begin{equation}\label{A_hierarchy-2}
    \| T_{P_{(\le N_0)} \psi}^{-1} \|_{C^{2{m^*}-1}} \lesssim_{C_0} A,
\end{equation}
(this uses our hierarchy of constants, namely that $A$ is chosen after $N_0$), while from \eqref{lp-1} we have
\[
\| P_{(\le N_0)} F \|_{C^{2{m^*}-1}} \lesssim_G \|F\|_{C^{2{m^*}-1}}
\]
and thus from \eqref{phi-start},
\begin{equation}\label{sim}
 \| \phi_{(\le N_0)} - \tilde \phi \|_{C^{2{m^*}-1}} \lesssim \|T^{-1}_{P_{(\le N_0)}\psi}\|_{C^{2{m^*}-1}}\|P_{(\le N_0)}F\|_{C^{2{m^*}-1}}\lesssim_{C_0} A\|F\|_{C^{2{m^*}-1}}.
 \end{equation}

Next, we establish the smoothness of the higher-frequency components $\phi_{(N)}$. From \eqref{lp-5} we have, for $N\ge N_0$ dyadic,
\begin{equation*}
 \| \nabla^m P_{(\le N)} \phi_{(<N)} \|_{C^{{m^*}+1}_{1/N}} \lesssim_G \| \nabla^m \phi_{(<N)} \|_{C^0} \le \| \phi_{(<N)} \|_{C^2},\quad m=1,2.
\end{equation*}
This implies in particular that
\[
\| X_i P_{(\le N)} \phi_{(<N)} \|_{C^{{m^*}+1}_{1/N}},\| X_iX_j P_{(\le N)} \phi_{(<N)} \|_{C^{{m^*}+1}_{1/N}},\| X_{2,i'} P_{(\le N)} \phi_{(<N)} \|_{C^{{m^*}+1}_{1/N}}  \lesssim_G \| \phi_{(<N)} \|_{C^2}.
\]
Again, from \eqref{lp-5} and \eqref{psi-4}, we also have the estimates
\begin{align}\label{P_Npsi}
\begin{aligned}
&\| {P_{(N)}} \psi \|_{C^{{m^*}+1}_{1/N}} \\
&\qquad\lesssim_G N^{-{m^*}-\alpha} \| \nabla^{{m^*}} \psi \|_{\dot C^{0,\alpha}} \lesssim  N^{-{m^*}-\alpha} A^{2-{m^*}-\alpha} \| \nabla^2 \psi \|_{C^{{m^*}-2,\alpha}_A} \lesssim_{C_0} N^{-{m^*}-\alpha} A^{1-{m^*}-\alpha}     
\end{aligned}
\end{align}
and
\begin{align*} 
\| {P_{(N)}} F \|_{C^{{m^*}+1}_{1/N}} &\lesssim_G N^{-2{m^*}+1} \| \nabla^{2{m^*}-1} F \|_{C^0} \lesssim N^{-2{m^*}+1} \| F \|_{C^{2{m^*}-1}}.
\end{align*}
Finally, from Lemma \ref{tpsi} one has
\[
\| T_{P_{(\le N)} \psi}^{-1} \|_{C^{2{m^*}-1}_{1/N}} \lesssim_{C_0} A
\]
since $A^{1-j} \lesssim A N^j$ for $0 \le j \le {m^*}-2$ and $A^{3-{m^*}} N^{j-{m^*}+2} \lesssim A N^j$ for ${m^*}-2 \le j \le 2{m^*}-1$. Inserting the above estimates into \eqref{phi-add}, we conclude that
\begin{align}\label{phi-add-estimate}
\begin{aligned}
 &\| \phi_{(N)} \|_{C^{{m^*}+1}_{1/N}} \\
&\qquad \lesssim_{\mathstrut} \|T^{-1}_{P_{(\le N)}\psi}\|_{C^{{m^*}+1}_{1/N}}\Big(\sum_{i=1}^k\|X_iP_{(\le N)}\phi_{(<N)}\|_{C^{{m^*}+1}_{1/N}}\|{P_{(N)}}\psi\|_{C^{{m^*}+1}_{1/N}}\\
&\qquad \qquad\qquad\qquad\qquad\qquad+\sum_{i,j=1}^k\|X_iX_iP_{(\le N)}\phi_{(<N)}\|_{C^{{m^*}+1}_{1/N}}\|{P_{(N)}}\psi\|_{C^{{m^*}+1}_{1/N}}\\
&\qquad \qquad\qquad\qquad\qquad\qquad\qquad+\sum_{i'=1}^{k_2}\|X_{2,i'}P_{(\le N)}\phi_{(<N)}\|_{C^{{m^*}+1}_{1/N}}\|{P_{(N)}}\psi\|_{C^{{m^*}+1}_{1/N}}  +\|{P_{(N)}} F\|_{C^{{m^*}+1}_{1/N}}\Big)\\
&\qquad \lesssim_{C_0} A \Big(A^{1-{m^*}-\alpha}N^{-{m^*}-\alpha}\|\phi_{(<N)}\|_{C^2}+N^{-2{m^*}+1}\|F\|_{C^{2{m^*}-1}}\Big)\\
&\qquad \lesssim_{C_0} A^{2-{m^*}-\alpha} N^{-{m^*}-\alpha}\|\phi_{(<N)}\|_{C^2}  + A N^{-2{m^*}+1} \| F \|_{C^{2{m^*}-1}}
\end{aligned}
\end{align}
and so
\begin{equation}\label{phi-add-estimate-2}
\| \phi_{(N)} \|_{C^{{m^*}}} \lesssim_{C_0} A^{2-{m^*}-\alpha} N^{-\alpha}\|\phi_{(<N)}\|_{C^2}  + A N^{-{m^*}+1} \| F \|_{C^{2{m^*}-1}}.    
\end{equation}
By the triangle inequality we thus have
\begin{align*}
 \| \phi_{(\le N)} - \tilde \phi \|_{C^{{m^*}}} &\le  \| \phi_{(< N)} - \tilde \phi \|_{C^{{m^*}}} + \| \phi_{(N)} \|_{C^{{m^*}}} \\ 
&\le (1 + O_{C_0}(A^{2-{m^*}-\alpha} N^{-\alpha})) \| \phi_{(<N)} - \tilde \phi \|_{C^{{m^*}}}\\
&\quad + O_{C_0}(A^{2-{m^*}-\alpha} N^{-\alpha}) \| \tilde \phi \|_{C^{{m^*},\alpha}} + O_{C_0}(A N^{-{m^*}+1}) \| F \|_{C^{2{m^*}-1}}.
\end{align*}
One can easily see by induction, with base case \eqref{sim}, that
\[
\| \phi_{(\le N)} - \tilde \phi \|_{C^{{m^*}}} \lesssim_{C_0} A\|F\|_{C^{2{m^*}-1}}+A^{2-{m^*}-\alpha} N_0^{-\alpha}\|\tilde{\phi}\|_{C^{{m^*},\alpha}} ,\quad N\ge N_0,
\]
and so by the triangle inequality, and noting that ${m^*}\ge 2$, we have
\begin{equation} \label{sim-3b}
\| \phi_{(\le N)} \|_{C^{{m^*}}} \lesssim_{C_0} A\|F\|_{C^{2{m^*}-1}}+\|\tilde{\phi}\|_{C^{{m^*},\alpha}} ,\quad N\ge N_0.
\end{equation}
Inserting this back into \eqref{phi-add-estimate} and \eqref{phi-add-estimate-2}, and noting again that ${m^*}\ge 2$, we obtain
\begin{align}\label{sim-3a}
\begin{aligned}
 \| \phi_{(N)} \|_{C^{{m^*}+1}_{1/N}} &\lesssim_{C_0} A^{2-{m^*}-\alpha} N^{-{m^*}-\alpha}\|\tilde{\phi}\|_{C^{{m^*},\alpha}}  + (A^{3-{m^*}-\alpha} N^{-{m^*}-\alpha}+AN^{-2{m^*}+1}) \| F \|_{C^{2{m^*}-1}} \\
 &\le A^{2-{m^*}-\alpha} N^{-{m^*}-\alpha}\|\tilde{\phi}\|_{C^{{m^*},\alpha}}  +  AN^{-{m^*}-\alpha} \| F \|_{C^{2{m^*}-1}}
\end{aligned}
\end{align}
and
\begin{align}\label{sim-3}
\begin{aligned}
 \| \phi_{(N)} \|_{C^{{m^*}}} \le N^{{m^*}} \| \phi_{(N)} \|_{C^{{m^*}+1}_{1/N}}\lesssim_{C_0} A^{2-{m^*}-\alpha} N^{-\alpha}\|\tilde{\phi}\|_{C^{{m^*},\alpha}}  +  A N^{-\alpha} \| F \|_{C^{2{m^*}-1}}.
\end{aligned}
\end{align}

We thus conclude that the sum
\[
\phi \coloneqq \phi_{(\le N_0)} + \sum_{N > N_0} \phi_{(N)}
\]
converges in the $C^{{m^*}}$ norm (and consequently also in the $C^2$ norm, as ${m^*}\ge 2$).

We now prove \eqref{G-8-prime}. From \eqref{sim}, it is enough to show
\[
\left\| \sum_{N > N_0} \phi_{(N)} \right\|_{C^{{m^*},\alpha}} \lesssim_{C_0} A \|F\|_{C^{2{m^*}-1}}+A^{2-{m^*}-\alpha}\|\Tilde{\phi}\|_{C^{{m^*},\alpha}}.
\]
As noted above, from \eqref{sim-3} we have
\[
\left\| \sum_{N > N_0} \phi_{(N)} \right\|_{C^{{m^*}}}\le  \sum_{N > N_0} \left\|\phi_{(N)} \right\|_{C^{{m^*}}} \lesssim_{C_0} AN_0^{-\alpha} \|F\|_{C^{2{m^*}-1}}+A^{2-{m^*}-\alpha}N_0^{-\alpha}\|\Tilde{\phi}\|_{C^{{m^*},\alpha}},
\]
so it remains to show H\"older regularity:
\begin{equation}\label{hold}
\left|\nabla^{{m^*}} \sum_{N > N_0} \phi_{(N)}(p) - \nabla^{{m^*}} \sum_{N > N_0} \phi_{(N)}(q)\right| \lesssim_{C_0} (A \|F\|_{C^{2{m^*}-1}}+A^{2-{m^*}-\alpha}\|\Tilde{\phi}\|_{C^{{m^*},\alpha}}) d(p,q)^\alpha
\end{equation}
for any $p,q \in G$.  By the triangle inequality, it is enough to show
\begin{equation*}
 \sum_{N > N_0}\left|\nabla^{{m^*}} \phi_{(N)}(p) - \nabla^{{m^*}}  \phi_{(N)}(q)\right| \lesssim_{C_0} (A \|F\|_{C^{2{m^*}-1}}+A^{2-{m^*}-\alpha}\|\Tilde{\phi}\|_{C^{{m^*},\alpha}}) d(p,q)^\alpha
\end{equation*}
On one hand, we may bound 
\begin{align*}
|\nabla^{{m^*}} \phi_{(N)}(p) - \nabla^{{m^*}} \phi_{(N)}(q)| &\lesssim \| \nabla^{{m^*}} \phi_{(N)} \|_{C^0} \lesssim \| \phi_{(N)} \|_{C^{{m^*}}} \\
&\stackrel{\mathclap{\eqref{sim-3}}}{\lesssim}_{C_0} A^{2-{m^*}-\alpha}N^{-\alpha}\|\tilde{\phi}\|_{C^{{m^*},\alpha}} + A N^{-\alpha}\|F\|_{C^{2{m^*}-1}}\\
&= (A^{2-{m^*}-\alpha}\|\tilde{\phi}\|_{C^{{m^*},\alpha}} + A\|F\|_{C^{2{m^*}-1}})N^{-\alpha}.
\end{align*}
On the other hand, one has
\begin{align*}
|\nabla^{{m^*}} \phi_{(N)}(p) - \nabla^{{m^*}} \phi_{(N)}(q)| &\lesssim \| \nabla^{{m^*}+1} \phi_{(N)} \|_{C^0} d(p,q)\lesssim N^{{m^*}+1} \| \phi_{(N)} \|_{C^{{m^*}+1}_{1/N}} d(p,q)\\
&\stackrel{\mathclap{\eqref{sim-3a}}}{\lesssim }_{C_0} (A^{2-{m^*}-\alpha}N^{-\alpha}\|\tilde{\phi}\|_{C^{{m^*},\alpha}} + A N^{-\alpha}\|F\|_{C^{2{m^*}-1}})(N d(p,q))\\
&\lesssim_{C_0}  (A^{2-{m^*}-\alpha}\|\tilde{\phi}\|_{C^{{m^*},\alpha}} + A\|F\|_{C^{2{m^*}-1}}) N^{-\alpha}(N d(p,q)).
\end{align*}
Thus, the left-hand side of \eqref{hold} is bounded by
\[
\lesssim_{C_0} (A^{2-{m^*}-\alpha}\|\tilde{\phi}\|_{C^{{m^*},\alpha}} + A\|F\|_{C^{2{m^*}-1}}) \sum_N N^{-\alpha} \min(1, N d(p,q))
\]
and the claim \eqref{G-8-prime} follows by summing the double-ended geometric series
\[
\sum_N N^{-\alpha} \min(1, N d(p,q))
\]
using the hypothesis $0 < \alpha < 1$.

Now we prove \eqref{phi-bounds}.  As $\phi_{(\le N)}$ converges in $C^2$ to $\phi$ as $N \to \infty$, and $P_{(\le N)} \psi$ converges in $C^2$ to $\psi$, we may write $B(\phi,\psi)$ as the uniform limit of $B(\phi_{(\le N)}, P_{(\le N)} \psi)$. Using \eqref{b-ident-1} and \eqref{b-ident-2}, we have the telescoping sum
\begin{align*}
B(\phi,\psi) &= B(\phi_{(\le N_0)}, P_{(\le N_0)} \psi) + \sum_{N > N_0} (B(\phi_{(\le N)}, P_{(\le N)} \psi) - B(\phi_{(< N)}, P_{(< N)} \psi)) \\
&=   P_{(\le N_0)} F + \sum_{N > N_0} (B(\phi_{(N)}, P_{(\le N)} \psi) + B(\phi_{(<N)}, {P_{(N)}} \psi))\\
&=   P_{(\le N_0)} F + \sum_{N > N_0} (P_{(N)} F + B(P_{(>N)}\phi_{(<N)}, {P_{(N)}} \psi))\\
&= F+\sum_{N > N_0}  B(P_{(>N)}\phi_{(<N)}, {P_{(N)}} \psi),
\end{align*}
or
\[
B(\phi,\psi) - F = \sum_{N>N_0} B(P_{(>N)}\phi_{(<N)}, {P_{(N)}} \psi).
\]
Each of the terms in the right-hand side, being a ``high-high paraproduct'' of $\nabla \psi$ and $\nabla \phi$, has much higher regularity ($C^{2{m^*}-1}$) than either $\nabla \psi$ or $\nabla \phi$ ($C^{{m^*}-1}$). Indeed, by the triangle inequality and product rule, we have
\begin{align*}
\| B(\phi,\psi) - F\|_{C^{2{m^*}-1}} &\le \sum_{N>N_0} \| B(P_{(\ge N)}\phi_{(<N)}, {P_{(N)}} \psi) \|_{C^{2{m^*}-1}}\\
&\lesssim \sum_{N>N_0} \sum_{j_1+j_2 = 2{m^*}-1} \| \nabla P_{(>N)} \phi_{(<N)} \|_{C^{j_1}}\| \nabla {P_{(N)}} \psi \|_{C^{j_2}}.    
\end{align*}
For any $0 \le j_1 \le 2{m^*}-1$, one has from \eqref{lp-5} and \eqref{sim-3} that
\begin{align*}
\| \nabla P_{(>N)} \phi_{(<N)} \|_{C^{j_1}} &\lesssim N^{j_1+1} \|  P_{(>N)} \phi_{(<N)} \|_{C^{j_1+1}_{1/N}} \\
&\lesssim_G N^{j_1 +1-{m^*}} \| \nabla^{{m^*}} \phi_{(<N)} \|_{C^0} \\
&\lesssim_{C_0} N^{j_1 +1-{m^*}} (A^{2-{m^*}-\alpha} N^{-\alpha}\|\tilde{\phi}\|_{C^{{m^*},\alpha}}  +  AN^{-\alpha} \| F \|_{C^{2{m^*}-1}})\\
&= N^{j_1 +1-{m^*}-\alpha} (A^{2-{m^*}-\alpha} \|\tilde{\phi}\|_{C^{{m^*},\alpha}}  +  A\| F \|_{C^{2{m^*}-1}}).
\end{align*}
Also, for any $0 \le j_2 \le 2{m^*}-1$, we have from \eqref{lp-5} and \eqref{psi-4} that 
\begin{align*}
\| \nabla {P_{(N)}} \psi \|_{C^{j_2}} &\lesssim N^{j_2+1} \| {P_{(N)}} \psi \|_{C^{j_2+1}_{1/N}} \lesssim_{G} N^{j_2+1} N^{-{m^*}-\alpha} \| \nabla^{{m^*}} \psi \|_{\dot C^{0,\alpha}}\lesssim_{C_0} N^{j_2+1-{m^*}-\alpha} A^{1-{m^*}-\alpha},
\end{align*}
and thus
\begin{align*}
&\| B(\phi,\psi) - F\|_{C^{2{m^*}-1}} \\
&\qquad\lesssim_{C_0}\sum_{N>N_0}\sum_{j_1+j_2=2{m^*}-1} N^{j_1 +1-{m^*}-\alpha}N^{j_2+1-{m^*}-\alpha} A^{1-{m^*}-\alpha} (A^{2-{m^*}-\alpha} \|\tilde{\phi}\|_{C^{{m^*},\alpha}}  +  A \| F \|_{C^{2{m^*}-1}})\\
&\qquad\lesssim_{C_0}\sum_{N>N_0} N^{1-2\alpha} A^{1-{m^*}-\alpha} (A^{2-{m^*}-\alpha} \|\tilde{\phi}\|_{C^{{m^*},\alpha}}  +  A \| F \|_{C^{2{m^*}-1}})\\
&\qquad\lesssim_{C_0} N_0^{1-2\alpha} A^{1-{m^*}-\alpha} (A^{2-{m^*}-\alpha} \|\tilde{\phi}\|_{C^{{m^*},\alpha}}  +  A \| F \|_{C^{2{m^*}-1}})
\end{align*}
(we used $\frac 12<\alpha<1$) which gives \eqref{phi-bounds}.

Finally, we prove \eqref{G-9-prime}. Fix $i,j=1,\cdots,k$. We need to establish
\[
\|X_i\phi\cdot X_j\psi-X_i\Tilde{\phi}\cdot X_jP_{(\le N_0)}\psi\|_{C^0}\lesssim_{C_0} \|F\|_{C^{2{m^*}-1}}+N_0^{1-{m^*}-\alpha}A^{1-{m^*}-\alpha}\|\Tilde{\phi}\|_{C^{{m^*},\alpha}}.
\]
Using telescoping sums,
\begin{align*}
    &X_i\phi\cdot X_j\psi-X_i\Tilde{\phi}\cdot X_jP_{(\le N_0)}\psi\\
    &\qquad\qquad\qquad=~ X_i(\phi_{(\le N_0)}-\tilde \phi)\cdot X_jP_{(\le N_0)} \psi+\sum_{N>N_0}\left( X_i\phi_{(\le N)}\cdot X_jP_{(\le N)}\psi-X_i\phi_{(<N)}\cdot X_jP_{(< N)}\psi\right)\\
    &\qquad\qquad\qquad\stackrel{\mathclap{\eqref{low-freq-cross-term}}}{=}~ P_{(\le N_0)} F_{ij}+\sum_{N>N_0}\left(X_i \phi_{(N)}\cdot X_j P_{(\le N)} \psi+X_i \phi_{(<N)}\cdot X_j {P_{(N)}} \psi\right)\\
    &\qquad\qquad\qquad\stackrel{\mathclap{\eqref{high-freq-cross-term}}}{=}~ P_{(\le N_0)} F_{ij}+\sum_{N>N_0}{P_{(N)}} F_{ij}+\sum_{N>N_0}\left(-X_jP_{(\le N)} \phi_{(<N)}\cdot X_i {P_{(N)}} \psi+X_i \phi_{(<N)}\cdot X_j {P_{(N)}} \psi\right)\\
    &\qquad\qquad\qquad=~F+\sum_{N>N_0}\left(-X_jP_{(\le N)} \phi_{(<N)}\cdot X_i {P_{(N)}} \psi+X_i \phi_{(<N)}\cdot X_j {P_{(N)}} \psi\right).
\end{align*}
Thus
\begin{align*}
    \|X_i\phi\cdot X_j\psi-X_i\Tilde{\phi}\cdot X_jP_{(\le N_0)}\psi\|_{C^0}&\le \|F\|_{C^0}+\sum_{N>N_0}\left(\|\nabla P_{(\le N)} \phi_{(<N)}\|_{C^0}+\|\nabla \phi_{(<N)}\|_{C^0} \right)\|\nabla {P_{(N)}} \psi\|_{C^0}\\
    &\stackrel{\mathclap{\eqref{lp-1}}}{\lesssim}_{G}\|F\|_{C^0}+\sum_{N>N_0}\|\nabla \phi_{(<N)}\|_{C^0}\|\nabla {P_{(N)}} \psi\|_{C^0}\\
    &\stackrel{\mathclap{\eqref{sim-3b}, \eqref{P_Npsi}}}{\lesssim}_{C_0}\|F\|_{C^0}+\sum_{N>N_0}\left(A\|F\|_{C^{2{m^*}-1}}+\|\tilde{\phi}\|_{C^{{m^*},\alpha}}\right)N^{1-{m^*}-\alpha} A^{1-{m^*}-\alpha}\\
    &\lesssim \|F\|_{C^{2{m^*}-1}}+N_0^{1-{m^*}-\alpha} A^{1-{m^*}-\alpha}\|\tilde{\phi}\|_{C^{{m^*},\alpha}},
\end{align*}
as desired.
\end{proof}

We state Proposition \ref{Perturbation} for the case $F=0$, which is the form that will be used later. (We remark that when converting \eqref{G-9} into \eqref{G-9-again} below, we have used our hierarchy of choosing $N_0$ after $C_0$.)
\begin{corollary}\label{perturbation-cor}
Let $M$ be a real number with
\[
M\ge C_0^{-1}.
\]
Let ${m^*}\ge 2$ and $\frac 12<\alpha<1$. Suppose we are given a $C^{{m^*},\alpha}$-map $\psi:G\to \bbr^{D}$ with the following regularity properties:
\begin{enumerate}[leftmargin=*]
    \item (H\"older regularity at scale $A$) We have
    \begin{equation*}
        \|\nabla^2 \psi\|_{C_A^{{m^*}-2,\alpha}}\le C_0A^{-1}.
    \end{equation*}
    \item (nondegenerate first derivatives) For any $p\in G$, we have
    \begin{equation*}
        C_0^{-1}M\le |X_i\psi(p)|\le C_0 M,\quad i=1,\cdots,k,
    \end{equation*}
    \item (locally free embedding) For any $p\in G$, we have
    \begin{equation*}
        \left|\bigwedge_{i=1}^k X_i\psi(p)\wedge \bigwedge_{1\le i\le j\le k} X_iX_j\psi(p)\wedge \bigwedge_{i'=1}^{k_2} X_{2,i'}\psi(p) \right|\gtrsim_{C_0} A^{-\frac{k(k+1)}{2}-k_2}M^{k}.
    \end{equation*}
\end{enumerate}

Let $\Tilde{\phi}:G\to\bbr^{D}$ be a $C^{{m^*},\alpha}$-solution to the low-frequency equation \eqref{G-5} with $\|\Tilde{\phi}\|_{C^{{m^*},\alpha}}<\infty$. Then there exists a $C^{{m^*},\alpha}$-solution $\phi$ to \eqref{G-1} which is a small perturbation of $\tilde{\phi}$:
\begin{equation}\label{G-8-again}
    \|\phi-\Tilde{\phi}\|_{C^{{m^*},\alpha}}\lesssim_{C_0} A^{2-{m^*}}\|\Tilde{\phi}\|_{C^{{m^*},\alpha}},
\end{equation}
and which makes the following cross terms small:
\begin{equation}\label{G-9-again}
    \|X_i\phi\cdot X_j\psi-X_i\Tilde{\phi}\cdot X_jP_{(\le N_0)}\psi\|_{C^0}\le A^{1-{m^*}}\|\Tilde{\phi}\|_{C^{{m^*},\alpha}},\quad i,j=1,\cdots,k.
\end{equation}
\end{corollary}
\begin{remark}
\begin{enumerate}[leftmargin=*]
    \item To repeat, we will later on use $\alpha=\frac 23$ and ${m^*}=s^2+s+1$, where $G$ is of step $s$, so these are to be considered ``geometric constants''. Following our hierarchy of constants, we will be choosing $C_0$, $N_0$, and $A$, in that order.
    \item At first glance, Corollary \ref{perturbation-cor} may seem to be missing a necessary dependence on $N_0$: the solution $\Tilde{\phi}$ to \eqref{G-5} approximates the solution $\phi$ to \eqref{G-1}, with the approximation getting better as $N_0$ gets larger, yet the quantitative estimates \eqref{G-8-again} and \eqref{G-9-again} lack any explicit dependence on $N_0$. However, the dependence is implicit: \eqref{G-8-again} and \eqref{G-9-again} depend on $A$ which in turn depends on $N_0$, and as $N_0$ gets larger, the right-hand sides of \eqref{G-8-again} and \eqref{G-9-again} get smaller. The part when we needed the dependence of $A$ on $N_0$ in the proof of Proposition \ref{Perturbation} was in \eqref{A_hierarchy-2}.
\end{enumerate}
\end{remark}


\section{Regular extensions of orthonormal systems}\label{sec:ON-ext}
\setcounter{equation}{0}

Another tool that we will need is a certain result on extending orthonormal systems. In order to apply Corollary \ref{perturbation-cor}, we first need to construct a nontrivial solution $\tilde{\phi}$ to the low-frequency equation \eqref{G-5}, or more generally the stronger equation
\begin{equation}\label{lowfreq-strongperp}
    X_i\tilde{\phi}\cdot X_j P_{(\le N_0)}\psi =0,\quad i,j=1,\cdots,k
\end{equation}
as promised in Remark \ref{perturb-rem}. One can see using the Leibniz rule that it is enough to solve the system
\begin{equation}\label{pre-strongperp}
\begin{cases}
\tilde{\phi}\cdot X_i P_{(\le N_0)}\psi =0,& i=1,\cdots,k,\\
\tilde{\phi}\cdot X_iX_j P_{(\le N_0)}\psi =0,& i,j=1,\cdots,k.
\end{cases}
\end{equation}
However, the vectors $\{X_iX_j P_{(\le N_0)}\psi\}_{i,j=1,\cdots,k}$ may not be linearly independent, as $\{X_iX_j-X_jX_i\in V_2:1\le i<j\le k\}$ may be linearly dependent, so the above system \eqref{pre-strongperp} may be overdetermined. Instead, we will solve the equivalent system
\begin{equation}\label{strongperp}
    \begin{cases}
\tilde{\phi}\cdot X_i P_{(\le N_0)}\psi =0,& i=1,\cdots,k,\\
\tilde{\phi}\cdot X_iX_j P_{(\le N_0)}\psi =0,& i,j=1,\cdots,k,~i\le j,\\
\tilde{\phi}\cdot X_{2,i} P_{(\le N_0)}\psi =0,& i=1,\cdots,k_2,
\end{cases}
\end{equation}
where we recall that $\{X_{2,i}\}_{i=1}^{k_2}$ is a basis of $V_2$ (the equivalence of \eqref{pre-strongperp} and \eqref{strongperp} follows from $[V_1,V_1]=V_2$). The vectors $\{X_i P_{(\le N_0)}\psi\}_{i=1,\cdots,k}\cup\{X_iX_j P_{(\le N_0)}\psi\}_{1\le i\le j\le k}\cup \{X_{2,i} P_{(\le N_0)}\psi\}_{i=1,\cdots,k_2}$ can be made linearly independent if we require the stronger freeness property
\begin{equation}\label{freeness-depth2}
\left|\bigwedge_{i=1}^k X_i\psi\wedge \bigwedge_{1\le i\le j\le k}X_iX_j\psi\wedge \bigwedge_{i=1}^{k_2}X_{2,i}\psi\right|\gtrsim_{C_0}\prod_{i=1}^k |X_i\psi|\cdot \prod_{1\le i\le j\le k}|X_iX_j\psi|\cdot \prod_{i=1}^{k_2}|X_{2,i}\psi|
\end{equation}
along with some regularity, say $\|\nabla^2\psi\|_{C_A^1}\lesssim_{C_0} A^{-1}$, for then we can apply \eqref{lp-5} to approximate the derivatives of $P_{(\le N_0)}\psi$ by the corresponding derivatives of $\psi$.

Now suppose $v_1,\cdots,v_{k(k+1)/2+k_2}$ is the result of applying the Gram-Schmidt process to the vectors $\{X_i P_{(\le N_0)}\psi\}_{i=1,\cdots,k}\cup\{X_iX_j P_{(\le N_0)}\psi\}_{1\le i\le j\le k}\cup \{X_{2,i} P_{(\le N_0)}\psi\}_{i=1,\cdots,k_2}$. The freeness property \eqref{freeness-depth2} guarantees some regularity of the $v_i$, and solving \eqref{strongperp} is equivalent to solving
\begin{equation}\label{perp-motiv}
    \tilde{\phi}\cdot v_i=0,\quad i=1,\cdots, \frac{k(k+1)}{2}+k_2.
\end{equation}

Thus, the following general question arises:
\begin{question}\label{general-lift}
Given a space $X$ on which a function space $(\mathcal{F}(X),\|\cdot\|_\mathcal{F})$ is defined, and given maps $v_1,\cdots,v_m:X\to\bbs^{D-1}$, $D\ge m+1$, which form a pointwise orthonormal system and which have the uniform regularity bound $\|v_i\|_\mathcal{F}\lesssim 1$, when can we extend the system to include a new map $v_{m+1}:X\to\bbs^{D-1}$, such that $v_1,\cdots,v_m,v_{m+1}$ forms a pointwise orthonormal system and $\|v_{m+1}\|_\mathcal{F}\lesssim_{X,\mathcal{F}(X),m,D}1$?
\end{question}

We will provide partial positive answers to this question that are applicable to our construction of the embedding. Theorem \ref{Cj-lifting-lattice} is a previous positive answer from \cite{tao2021embedding}, while Theorems \ref{Lip-lifting}, \ref{Cj-lifting}, and \ref{partialpositiveanswer} and Corollary \ref{multiple-Cj-lifting} are the new positive answers of this paper.

Under conditions that validate the above question, we would be able to simply take $\tilde{\phi}$ to be $v_{m+1}$ to solve \eqref{perp-motiv}, for then $\tilde{\phi}$ would have regularity similar to that of the $v_i$; i.e., it would have bounded $C^{{m^*},\alpha}$ norm (there is some loss of constant factors when applying \eqref{lp-1}, but this is offset by the fact that there is a change of scale: we have control on the $C_A^{{m^*},\alpha}$ norm of $\psi$, and we need only control the $C^{{m^*},\alpha}$ norm of $\tilde{\phi}$).

Actually, in order to also obtain a freeness property for $\psi+\phi$, we will take $\tilde{\phi}$ to be a linear combination of a larger extension $v_{m+1},\cdots,v_{m+m'}$ of $v_1,\cdots,v_m$ with variable coefficients; the coefficients will guarantee the freeness property in this case (see \eqref{IsomCompose} and \eqref{decomp}). Obtaining a larger orthonormal extension will be possible simply by adding new vectors one by one.

In \cite[Section 8]{tao2021embedding} Question \ref{general-lift} has been answered in the affirmative for the case $X=\bbh^3$ and
\[
\|\phi\|_\mathcal{F}=\|\phi\|_{C^0}+R\|\nabla \phi\|_{C^j}
\]
for any $j\ge 0$ and $R\ge 1$. The proof in \cite{tao2021embedding} used the fact that the Heisenberg group $\bbh^3$ admits a CW complex structure that is periodic with respect to its standard discrete cocompact lattice. The methods of \cite{tao2021embedding} can be generalized in a straightforward manner to prove the following theorem.

\begin{theorem}[{\cite[Corollary 8.4]{tao2021embedding}}]\label{Cj-lifting-lattice}
Let $G$ be a Carnot group that admits a cocompact lattice $\Gamma$ and a CW structure whose cells can be obtained from left $\Gamma$-translation from a finite list of cells. Let $1\le m\le D-n-1$ (where $n$ is the topological dimension of $G$), $j\ge 1$, and let $\{R_i\}_{i=1}^j$ be a log-concave sequence of positive reals, i.e., $R_iR_{i'}\ge R_{i+i'}$ whenever $i+i'\le j$. Let $v_1,\cdots,v_m:G\to \bbs^{D-1}$ be functions that form an orthonormal system at each point, with the uniform regularity bound
\[
\sum_{k=1}^j R_k\|\nabla^k v_i\|_{C^0}\le 1,\quad i=1,\cdots, m.
\]
Then there exists another function $v_{m+1}:G\to \bbs^{D-1}$ such that $v_1,\cdots,v_m$ along with $v_{m+1}$ form an orthonormal system at each point, and
\begin{equation}\label{controlled-norms-lattice}
\sum_{k=1}^j R_k\|\nabla^k v_{m+1}\|_{C^0}\lesssim_{G,D,j} 1.
\end{equation}
\end{theorem}
In other words, we are given a bundle $B$ over $G$, where for each $p\in G$ the fiber of $B$ over $p$ is the collection of unit vectors $v\in \bbs^{D-1}$ which are perpendicular to $v_1(p),\cdots,v_m(p)$ (so each fibre is homeomorphic to $\bbs^{D-m-1}$), and we need to show that there is a section of this bundle which has the same level of control on the regularity of the bundle itself. Tao \cite[Section 8]{tao2021embedding} achieved this for the Heisenberg group $\bbh^3$ in the spirit of quantitative topology, by imposing a ``uniform'' CW-structure on $\bbh^3$ as above and then inductively constructing the section starting from low-dimensional skeleta. In the inductive step in \cite{tao2021embedding}, one has to use the fact that the homotopy groups $\pi_i(\bbs^n)$ vanish for $i<n$, which necessitates the ``dimension gap'' $m\le D-3-1$.

The main difficulty in proving Theorem \ref{Cj-lifting-lattice} is to construct a section $\tilde{v}_{m+1}$ which is uniformly continuous, with the modulus of uniform continuity depending only on $G$, $D$ and $R_1$. We can then obtain a section $v_{m+1}$ with the stronger regularity property \eqref{controlled-norms-lattice} simply by mollifying $\tilde{v}_{m+1}$ and then applying the Gram-Schmidt orthogonalization process; the regularity \eqref{controlled-norms-lattice} will simply be a consequence of the algebra property for norms of the form $\|\cdot\|_{C^0}+\sum_{k=1}^j R_k\|\nabla^k \cdot\|_{C^0}$.

General Carnot groups may not admit a cocompact lattice (as the structure constants for any basis may be irrational), and it is not clear whether $G$ admits a ``uniform'' CW structure that is amenable to the above proof method. We will avoid the need for a CW structure by only using the fact that $G$ is a doubling metric space. More precisely, we will first prove that the most challenging part of Theorem \ref{Cj-lifting-lattice}, i.e., the case $j=1$, can be done in the setting of doubling metric spaces. We state this result separately in anticipation of future work.

\begin{theorem}\label{Lip-lifting}
Let $(X,d)$ be a $K$-doubling metric space ($K\ge 2$), and let $m\le D-224 K^4\log K$. If $v_1,\cdots,v_m:X\to\bbs^{D-1}$ are 1-Lipschitz functions that form an orthonormal system at each point, then there exists a $150K^5m(m+1)$-Lipschitz function $v_{m+1}:X\to \bbs^{D-1}$ such that $v_1,\cdots,v_m$ along with $v_{m+1}$ form an orthonormal system at each point.
\end{theorem}

It will later become clear that Theorem \ref{Lip-lifting} provides a partial positive answer to Question \ref{general-lift} for many other function spaces as well; see Theorem \ref{partialpositiveanswer}.

\begin{proof}
Take a maximal $\delta$-net $\mathcal{N}_\delta$ of $X$, where $\delta=1/(8Km)$. By \eqref{doubling-net},
\[
|\mathcal{N}_\delta \cap B_{2\delta}(p)|\le K^2\quad \mathrm{and}\quad|\mathcal{N}_\delta \cap B_{4\delta}(p)|\le K^3.
\]
Let $\Omega$ be a probability space on which independent random variables $v_{m+1}'(p)\in \bbs^{D-1}$ are defined so that for each $\omega\in \Omega$ and $p\in X$, $v_{m+1}'(p)(\omega)$ forms an orthonormal set along with $v_1(p)(\omega),$ $\cdots,$ $v_m(p)(\omega)$. Let $\epsilon=1/(4K^2)$, and for each $p\in \mathcal{N}_\delta$, let us define the event
\[
A_p=\{\omega\in \Omega:\mathrm{there~exists~} q\in \mathcal{N}_\delta \cap (B_{2\delta}(p)\setminus \{p\})\mathrm{~such~that~}|v_{m+1}'(p)(\omega)\cdot v_{m+1}'(q)(\omega)|>\epsilon\}.
\]
Note that for any $p,q\in \mathcal{N}_\delta$ distinct, we may compute
\[
\mathrm{Pr}(|v_{m+1}'(p)\cdot v_{m+1}'(q)|>\epsilon)= \mathbb{E}_p \mathrm{Pr}_q(|v_{m+1}'(p)\cdot v_{m+1}'(q)|>\epsilon)\le \exp\left(-\frac{\epsilon^2}{2}(D-m)\right),
\]
where the last inequality follows from a standard computation on the area of caps on the sphere (see \cite[Chapter 2]{10.5555/21465} for instance). Therefore, we may estimate the probability of each $A_p$ using a union bound:
\begin{align*}
\mathrm{Pr}(A_p)&\le \sum_{q\in \mathcal{N}_\delta \cap (B_{2\delta}(p)\setminus \{p\})}\mathrm{Pr}(|v_{m+1}'(p)\cdot v_{m+1}'(q)|>\epsilon)\\
&\le |\mathcal{N}_\delta \cap (B_{2\delta}(p)\setminus \{p\})|\cdot \exp\left(-\frac{\epsilon^2}{2}(D-m)\right)\le K^2\cdot \exp\left(-\frac{\epsilon^2}{2}(D-m)\right).
\end{align*}
Also note that for each $p\in \mathcal{N}_\delta$, $A_p$ is mutually independent with the collection of events $\{A_q:q\in \mathcal{N}_\delta \setminus B_{4\delta}(p)\}$, which are all the $A_q$ except possibly $|\mathcal{N}_\delta\cap B_{4\delta}(p)|\le K^3$ of them. By the Lov\'{a}sz local lemma, we see that if
\begin{equation}\label{lovasz}
e\cdot K^3 \cdot K^2 \cdot \exp\left(-\frac{\epsilon^2}{2}(D-m)\right)< 1,    
\end{equation}
then for any finite subcollection $S\subset \mathcal{N}_\delta$ we have $\mathrm{Pr}(\bigcap_{p\in S}A_p^\mathsf{c})>0$. But by our choice of parameters $D-m\ge 224K^4\log K$ and $\epsilon=\frac{1}{4K^2}$, the condition \eqref{lovasz} is indeed satisfied, since the left-hand side is bounded by
\begin{align*}
    &e\cdot K^5\cdot \exp\left(-\frac{\epsilon^2}{2}(D-m)\right)\le e\cdot K^5\cdot \exp\left(-7\log K\right)<1\quad(\mathrm{since~}7\log K>5\log K+1\mbox{ as }K\ge 2).
\end{align*}
Hence, for any finite subcollection $S\subset \mathcal{N}_\delta$ we have $\mathrm{Pr}(\bigcap_{p\in S}A_p^\mathsf{c})>0$, and in particular $\bigcap_{p\in S}A_p^\mathsf{c}\neq \emptyset$. We can thus find an assignment $\{v_{m+1,S}'(p)\}_{p\in S}$ such that for any distinct $p,q\in S$ with $d(p,q)<2\delta$ we have $|v_{m+1,S}'(p)\cdot v_{m+1,S}'(q)|\le \epsilon$. By taking an arbitrary enumeration of $\mathcal{N}_\delta$, taking a monotone increasing sequence of $S$'s that cover $\mathcal{N}_\delta$ and passing to a limit along a nonprincipal ultrafilter, we conclude the existence of an assignment $\{v_{m+1}'(p)\}_{p\in \mathcal{N}_\delta}$ such that for any distinct $p,q\in \mathcal{N}_\delta$ with $d(p,q)<2\delta$ we have $|v_{m+1}'(p)\cdot v_{m+1}'(q)|\le \epsilon$.

We now ``interpolate'' the discrete vector field $\{v_{m+1}'(p)\}_{p\in \mathcal{N}_\delta}$ to produce a vector field\linebreak $\{\tilde{v}_{m+1}(p)\}_{p\in X} $ defined on the entirety of $X$, which nearly has the desired properties. We first construct a ``quadratic'' partition of unity $\{\phi_q\}_{q\in \mathcal{N}_\delta}$, i.e., functions $\phi_q:X\to [0,1]$ defined for each $q\in \mathcal{N}_\delta$ such that
\begin{itemize}
    \item $\mathrm{supp~} \phi_q\subset B_{2\delta}(q)$, $q\in \mathcal{N}_\delta$,
    \item $\sum_{q\in \mathcal{N}_\delta}\phi_q^2=1$ on $X$,
    \item $\|\phi_q\|_{\mathrm{Lip}}\le 2K^2\delta^{-1}$, $q\in \mathcal{N}_\delta$.
\end{itemize}
Indeed, we start by defining for each $q\in \mathcal{N}_\delta$ the function
\[
\tilde{\phi}_q(p)=
\begin{cases}
1,& d(p,q)\le \delta,\\
2-\frac{d(p,q)}{\delta},&\delta<d(p,q)\le 2\delta,\\
0,& d(p,q)>2\delta,
\end{cases}
\]
and then define
\[
\phi_q\coloneqq \frac{\tilde{\phi}_q}{\sqrt{\sum_{r\in \mathcal{N}_\delta}\tilde{\phi}_r^2}}.
\]
This obviously satisfies the first two properties, and it remains to compute $\|\phi_q\|_{\mathrm{Lip}}$. We clearly have $\|\tilde{\phi}_q\|_{C^0}\le 1$, $\|\tilde{\phi}_q\|_{\mathrm{Lip}}\le \delta^{-1}$.

We first observe that
\[
\sum_{r\in \mathcal{N}_\delta}\tilde{\phi}_r^2\ge 1\quad\mathrm{and}\quad 
\left\|\sum_{r\in \mathcal{N}_\delta}\tilde{\phi}_r^2\right\|_{\mathrm{Lip}}\le (4K^2-2)\delta^{-1}.
\]
Indeed, the first follows from the fact that $\mathcal{N}_\delta$ is a maximal $\delta$-net. For the second property, fix any $p,p'\in X$ with $p\neq p'$. If $B_{2\delta}(p)\cap B_{2\delta}(p')\cap \mathcal{N}_\delta=\emptyset$ then we must have $d(p,p')\ge \delta$, because otherwise $B_{2\delta}(p)\cap B_{2\delta}(p')\cap \mathcal{N}_\delta\supseteq B_\delta(p)\cap \mathcal{N}_\delta\neq\emptyset$, by maximality of $\mathcal{N}_\delta$. We have 
\[
1\le \sum_{r\in \mathcal{N}_\delta}\tilde{\phi}_r(p)^2=\sum_{r\in B_{2\delta}(p)\cap \mathcal{N}_\delta}\tilde{\phi}_r(p)^2\le |B_{2\delta}(p)\cap \mathcal{N}_\delta|\le K^2,
\]
and similarly
\[
1\le \sum_{r\in \mathcal{N}_\delta}\tilde{\phi}_r(p')^2\le K^2.
\]
Therefore, in this case,
\[
\frac{\left|\sum_{r\in \mathcal{N}_\delta}\tilde{\phi}_r(p)^2-\sum_{r\in \mathcal{N}_\delta}\tilde{\phi}_r(p')^2\right|}{d(p,p')}\le (K^2-1)\delta^{-1}\le (4K^2-2)\delta^{-1}.
\]
On the other hand, if $B_{2\delta}(p)\cap B_{2\delta}(p')\cap \mathcal{N}_\delta\neq \emptyset$ then clearly
\[
|B_{2\delta}(p)\cap B_{2\delta}(p')\cap \mathcal{N}_\delta|\le |B_{2\delta}(p)\cap \mathcal{N}_\delta|+| B_{2\delta}(p')\cap \mathcal{N}_\delta|-1\le 2K^2-1.
\]
We have
\[
\sum_{r\in \mathcal{N}_\delta}\tilde{\phi}_r(p)^2=\sum_{r\in B_{2\delta}(p)\cap \mathcal{N}_\delta}\tilde{\phi}_r(p)^2=\sum_{r\in (B_{2\delta}(p)\cup B_{2\delta}(p'))\cap \mathcal{N}_\delta}\tilde{\phi}_r(p)^2
\]
and similarly
\[
\sum_{r\in \mathcal{N}_\delta}\tilde{\phi}_r(p')^2=\sum_{r\in (B_{2\delta}(p)\cup B_{2\delta}(p'))\cap \mathcal{N}_\delta}\tilde{\phi}_r(p')^2,
\]
so, noting that $|\tilde{\phi}_r(p)^2-\tilde{\phi}_r(p')^2|\le 2|\tilde{\phi}_r(p)^2-\tilde{\phi}_r(p')|\le 2\|\tilde{\phi}_r\|_{\mathrm{Lip}} d(p,p')\le 2\delta^{-1}d(p,p')$,
\begin{align*}
    \left|\sum_{r\in \mathcal{N}_\delta}\tilde{\phi}_r(p)^2-\sum_{r\in \mathcal{N}_\delta}\tilde{\phi}_r(p')^2\right|&\le \sum_{r\in (B_{2\delta}(p)\cup B_{2\delta}(p'))\cap \mathcal{N}_\delta} |\tilde{\phi}_r(p)^2-\tilde{\phi}_r(p')^2|\\
    &\le |(B_{2\delta}(p)\cup B_{2\delta}(p'))\cap \mathcal{N}_\delta|\cdot 2 \delta^{-1}d(p,p')\\
    &\le (2K^2-1)\cdot 2 \delta^{-1}d(p,p').
\end{align*}
This finishes the verification that $\left\|\sum_{r\in \mathcal{N}_\delta}\tilde{\phi}_r^2\right\|_{\mathrm{Lip}}\le (4K^2-2)\delta^{-1}$.

By \eqref{reciprocal} and \eqref{squareroot}, we have
\[
\left\|\frac{1}{\sqrt{\sum_{r\in \mathcal{N}_\delta}\tilde{\phi}_r^2}}\right\|_{\mathrm{Lip}}\le\left\|\sqrt{\sum_{r\in \mathcal{N}_\delta}\tilde{\phi}_r^2}\right\|_{\mathrm{Lip}}\le  \frac 12 \left\|\sum_{r\in \mathcal{N}_\delta}\tilde{\phi}_r^2\right\|_{\mathrm{Lip}}\le (2K^2-1)\delta^{-1}.
\]
Thus, by the definition of $\phi_q$ and \eqref{first-alg}, we have
\begin{align*}
    \|\phi_q\|_{\mathrm{Lip}}&\le \left\|\frac{1}{\sqrt{\sum_{r\in \mathcal{N}_\delta}\tilde{\phi}_r^2}}\right\|_{C^0}\left\|\tilde{\phi}_q\right\|_{\mathrm{Lip}}+\left\|\frac{1}{\sqrt{\sum_{r\in \mathcal{N}_\delta}\tilde{\phi}_r^2}}\right\|_{\mathrm{Lip}}\left\|\tilde{\phi}_q\right\|_{C^0}\\
    &\le 1\cdot \delta^{-1}+(2K^2-1)\delta^{-1}\cdot 1\le 2K^2\delta^{-1}.
\end{align*}

We now interpolate the vectors $\{v_{m+1}'(q)\}_{q\in \mathcal{N}_\delta}$ using the quadratic partition of unity $\{\phi_q\}_{q\in \mathcal{N}_\delta}$:
\begin{equation}\label{interpolate}
\tilde{v}_{m+1}(p)\coloneqq\sum_{q\in \mathcal{N}_\delta}\phi_q(p)v_{m+1}'(q),\quad p\in X.
\end{equation}
The idea is that this interpolates nearby $v'_{m+1}$, which are mutually almost orthogonal, so $\tilde{v}_{m+1}$ should nearly be a unit vector. Moreover, since the 1-Lipschitz functions $v_1,\cdots,v_m$ vary slowly over distance $\delta$, this $\tilde{v}_{m+1}$ should also be nearly orthogonal to $v_1,\cdots,v_m$. This $\tilde{v}_{m+1}$ will oscillate much quicker than $v_1,\cdots,v_m$, but its Lipschitz constant will be controlled by $K$ and $\delta$.

More precisely, we claim that $\tilde{v}_{m+1}$ satisfies the following quantitative estimates:
\begin{enumerate}[leftmargin=*]
    \item $\frac 34\le |\tilde{v}_{m+1}(p)|^2\le \frac 54$ for each $p\in X$,
    \item $|\tilde{v}_{m+1}(p)\cdot v_i(p)|\le \frac 1{4m}$ for each $p\in X$ and $i=1,\cdots,m$,
    \item $\|\tilde{v}_{m+1}\|_{\mathrm{Lip}}\le 2K^2(2K^2-1)\delta^{-1}$.
\end{enumerate}
Indeed, by the support property of $\phi_q$, we see that the summation in \eqref{interpolate} is locally finite:
\[
\tilde{v}_{m+1}(p)=\sum_{q\in \mathcal{N}_\delta\cap B_{2\delta}(p)}\phi_q(p)v_{m+1}'(q).
\]
We can verify property (1) using the near-orthogonality and the fact that the sum of squares of the $\phi_q$ is 1:
\begin{align*}
    \left||\tilde{v}_{m+1}(p)|^2-1\right|&=\left|\sum_{\substack{q,q'\in \mathcal{N}_\delta\cap B_{2\delta}(p) \\ q\neq q'}}\phi_q(p)\phi_{q'}(p)v_{m+1}'(q)\cdot v_{m+1}'(q')\right|\\
    &\le \sum_{\substack{q,q'\in \mathcal{N}_\delta\cap B_{2\delta}(p) \\ q\neq q'}}\phi_q(p)\phi_{q'}(p)\left|v_{m+1}'(q)\cdot v_{m+1}'(q')\right|\\
    &\le \sum_{\substack{q,q'\in \mathcal{N}_\delta\cap B_{2\delta}(p) \\ q\neq q'}}\frac{\phi_q(p)^2+\phi_{q'}(p)^2}{2}\epsilon \le |\mathcal{N}_\delta\cap B_{2\delta}(p)|\epsilon\le K^2\epsilon=\frac 14.
\end{align*}
To verify property (2), we see that for each $p\in X$ and $i=1,\cdots, m$,
\[
\tilde{v}_{m+1}(p)\cdot v_i(p) =\sum_{q\in \mathcal{N}_\delta\cap B_{2\delta}(p)}\phi_q(p)v_{m+1}'(q)\cdot v_i(p).
\]
We observe that for each $q\in \mathcal{N}_\delta\cap B_{2\delta}(p)$, we can estimate
\begin{align*}
|v_{m+1}'(q)\cdot v_i(p)|&\le |v_{m+1}'(q)\cdot v_i(q)|+|v_{m+1}'(q)\cdot (v_i(p)-v_i(q))|\\
&\le 0+ | v_i(p)-v_i(q)|\le 2\delta,
\end{align*}
where we have used the orthogonality of $v_{m+1}'(q)$ against $v_i(q)$ and the fact that $\| v_i\|_{\mathrm{Lip}}\le 1$. By these facts and Cauchy--Schwarz, we have the bound
\begin{align*}
    |\tilde{v}_{m+1}(p)\cdot v_i(p)| \le \sum_{q\in \mathcal{N}_\delta\cap B_{2\delta}(p)}\phi_q(p)\cdot 2\delta\le |\mathcal{N}_\delta\cap B_{2\delta}(p)|^{1/2}\cdot 2\delta\le K\cdot 2\delta=\frac 1{4m}.
\end{align*}

To verify (3), we first fix $p,p'\in X$ with $p\neq p'$. If $B_{2\delta}(p)\cap B_{2\delta}(p')\cap \mathcal{N}_\delta=\emptyset$ then we must have $d(p,p')\ge \delta$, because otherwise $B_{2\delta}(p)\cap B_{2\delta}(p')\cap \mathcal{N}_\delta\supseteq B_\delta(p)\cap \mathcal{N}_\delta\neq\emptyset$. By (1), we have $\frac{\sqrt{3}}{2}\le |\tilde{v}_{m+1}(p)|,|\tilde{v}_{m+1}(p')|\le \frac{\sqrt{5}}{2}$, so
\[
\frac{|\tilde{v}_{m+1}(p)-\tilde{v}_{m+1}(p')|}{d(p,p')}\le \frac{\sqrt{5}-\sqrt{3}}{2}\delta^{-1}\le 2K^2(2K^2-1)\delta^{-1}.
\]
On the other hand, if $B_{2\delta}(p)\cap B_{2\delta}(p')\cap \mathcal{N}_\delta\neq \emptyset$ then clearly
\[|B_{2\delta}(p)\cap B_{2\delta}(p')\cap \mathcal{N}_\delta|\le |B_{2\delta}(p)\cap \mathcal{N}_\delta|+| B_{2\delta}(p')\cap \mathcal{N}_\delta|-1\le 2K^2-1.
\]
We have
\[
\tilde{v}_{m+1}(p)=\sum_{r\in B_{2\delta}(p)\cap \mathcal{N}_\delta}\phi_r(p)v_{m+1}(r)=\sum_{r\in (B_{2\delta}(p)\cup B_{2\delta}(p'))\cap \mathcal{N}_\delta}\phi_r(p)v_{m+1}(r)
\]
and similarly
\[
\tilde{v}_{m+1}(p')=\sum_{r\in (B_{2\delta}(p)\cup B_{2\delta}(p'))\cap \mathcal{N}_\delta}\phi_r(p')v_{m+1}(r),
\]
so
\begin{align*}
    \left|\tilde{v}_{m+1}(p)-\tilde{v}_{m+1}(p')\right|&\le \sum_{r\in (B_{2\delta}(p)\cup B_{2\delta}(p'))\cap \mathcal{N}_\delta}|\phi_r(p)-\phi_r(p')|\cdot |v_{m+1}(r)|\\
    &\le \sum_{r\in (B_{2\delta}(p)\cup B_{2\delta}(p'))\cap \mathcal{N}_\delta}\|\phi_r\|_{\mathrm{Lip}}d(p,p')\\
    &\le |(B_{2\delta}(p)\cup B_{2\delta}(p'))\cap \mathcal{N}_\delta|\cdot 2K^2\delta^{-1}d(p,p')\\
    &\le 2K^2(2K^2-1)\delta^{-1}d(p,p').
\end{align*}
This finishes the verification of properties (1) to (3).

Properties (1)-(3) above finally allow us to use the Gram-Schmidt orthogonalization process to obtain a true $v_{m+1}$ with the desired properties:
\[
v_{m+1}(p):=\frac{\tilde{v}_{m+1}(p)-\sum_{i=1}^m \tilde{v}_{m+1}(p)\cdot v_i(p)}{|\tilde{v}_{m+1}(p)-\sum_{i=1}^m \tilde{v}_{m+1}(p)\cdot v_i(p)|},\quad p\in X.
\]
This is well-defined because
\[
|\tilde{v}_{m+1}(p)-\sum_{i=1}^m \tilde{v}_{m+1}(p)\cdot v_i(p)|\ge |\tilde{v}_{m+1}(p)|-\sum_{i=1}^m |\tilde{v}_{m+1}(p)\cdot v_i(p)|\ge \frac{\sqrt{3}}{2}-\frac 14>0,
\]
and $v_{m+1}(p)$ clearly forms an orthonormal system along with $v_1(p),\cdots,v_m(p)$. We also note that
\[
|\tilde{v}_{m+1}(p)-\sum_{i=1}^m \tilde{v}_{m+1}(p)\cdot v_i(p)|\le |\tilde{v}_{m+1}(p)|+\sum_{i=1}^m |\tilde{v}_{m+1}(p)\cdot v_i(p)|\le \frac{\sqrt{3}}{2}+\frac 14.
\]
To check the Lipschitz regularity of $v_{m+1}$ we first compute (recalling $\delta^{-1}=8Km$)
\begin{align*}
    \left\|\tilde{v}_{m+1}-\sum_{i=1}^m \tilde{v}_{m+1}\cdot v_i\right\|_{\mathrm{Lip}}&\le \left\|\tilde{v}_{m+1}\right\|_{\mathrm{Lip}}+\sum_{i=1}^m (\left\|\tilde{v}_{m+1}\right\|_{C^0}\left\| v_i\right\|_{\mathrm{Lip}}+\left\|\tilde{v}_{m+1}\right\|_{\mathrm{Lip}}\left\| v_i\right\|_{C^0})\\
    &\le 2K^2(2K^2-1)\delta^{-1}+m(\frac{\sqrt{5}}{2}\cdot 1+2K^2(2K^2-1)\delta^{-1}\cdot 1)\\
    &= 2K^2\delta^{-1}\Big(2K^2-1+\frac{\sqrt{5}}{32K^3}+2K^2m-m\Big)\\
    &\le 4K^4\delta^{-1}(m+1)=32K^5m(m+1).
\end{align*}
By \eqref{unit-alg} we finally have
\begin{align*}
    \|v_{m+1}\|_{\mathrm{Lip}}&\le \Big(\big(\frac{\sqrt{3}}{2}-\frac 14\big)^{-1}+\big(\frac{\sqrt{3}}{2}-\frac 14\big)^{-2}\big(\frac{\sqrt{3}}{2}+\frac 14\big)\Big) \left\|\tilde{v}_{m+1}-\sum_{i=1}^m \tilde{v}_{m+1}\cdot v_i\right\|_{\mathrm{Lip}}\\
    &\le \Big(\big(\frac{\sqrt{3}}{2}-\frac 14\big)^{-1}+\big(\frac{\sqrt{3}}{2}-\frac 14\big)^{-2}\big(\frac{\sqrt{3}}{2}+\frac 14\big)\Big)\cdot 32K^5m(m+1)\\
    &\le 150K^5m(m+1).
\end{align*}
\end{proof}

\begin{remark}
In Theorem \ref{Cj-lifting-lattice}, the required dimension gap $D-m-1$ is precisely the topological dimension of the given domain $G$. In contrast, in Theorem \ref{Lip-lifting} the dimension gap is $224 K^4\log K$, which is exponential in the ``metric dimension'' $\log K$ of the given domain $X$. Note that a dimension gap of $\Omega(\log K)$ is necessary: take for example $X=S^{2n}$, $D=2n+1$, $m=1$, and $v_1:S^{2n}\hookrightarrow\bbr^{2n+1}$ the standard inclusion. Then the log of the doubling constant for $X$ is a universal constant multiple of $n$, and the dimension gap is $D-m-1=2n-1$. However, the high-dimensional hairy ball theorem tells us that no continuous orthonormal extension is possible, let alone a Lipschitz orthonormal extension.

Perhaps one could reduce the dimension gap in Theorem \ref{Lip-lifting}, say by randomizing the proof using the machinery of random nets and partitions and their padding properties as in \cite{naor2010assouad}, though reducing it to somewhere near $\Omega(\log K)$ seems to require more effort (or perhaps it is impossible, but we do not have a counterexample yet).
\end{remark}

We now state and prove a version of Theorem \ref{Cj-lifting-lattice} but for general Carnot groups, using the idea of proof of Theorem \ref{Lip-lifting}.

\begin{theorem}\label{Cj-lifting}
Let $1\le m\le D-2^{4n_h+7}n_h$, $j\ge 1$, and let $\{R_i\}_{i=1}^j$ be a log-concave sequence of positive reals. Let $v_1,\cdots,v_m:G\to \bbs^{D-1}$ form an orthonormal system at each point, with the uniform regularity bound
\[
\sum_{k=1}^j R_k\|\nabla^k v_i\|_{C^0}\le 1,\quad i=1,\cdots, m.
\]
Then there exists $v_{m+1}:G\to \bbs^{D-1}$ such that $v_1,\cdots,v_m$ along with $v_{m+1}$ form an orthonormal system at each point, and
\[
\sum_{k=1}^j R_k\|\nabla^k v_{m+1}\|_{C^0}\lesssim_{G,m,j} 1.
\]
\end{theorem}
\begin{proof}

By a simple rescaling argument using the scaling map $\delta_{R_1}$, we may assume $R_1=1$. We repeat the proof of Theorem \ref{Lip-lifting} with a slight variation.

Take a maximal $\delta$-net $\mathcal{N}_\delta$ of $G$, where $\delta=\frac{1}{6\cdot 2^{n_h}m}$. Let $\epsilon=\frac{1}{4^{n_h+1}}$. We repeat the first part of the proof of Theorem \ref{Lip-lifting}, but instead use the estimates
\[
|\mathcal{N}_\delta \cap B_{1.5\delta}(p)|\le 4^{n_h},\quad |\mathcal{N}_\delta \cap B_{3\delta}(p)|\le 7^{n_h}
\]
which follow from \eqref{locbd}.
We again define the probability space $\Omega$ and random variables $\{v_{m+1}'(p)\in \bbs^{D-1}:p\in \mathcal{N}_\delta\}$, and for each $p\in \mathcal{N}_\delta$, we define the slightly different event
\[
A_p=\{\omega\in \Omega:\mathrm{there~exists~}q\in \mathcal{N}_\delta \cap (B_{1.5\delta}(p)\setminus \{p\})\mathrm{~such~that}~|v_{m+1}'(p)(\omega)\cdot v_{m+1}'(q)(\omega)|>\epsilon\}.
\]
By the same computations as before we have that for any $p,q\in \mathcal{N}_\delta$ distinct,
\[
\mathrm{Pr}(|v_{m+1}'(p)\cdot v_{m+1}'(q)|>\epsilon)\le \exp\left(-\frac{\epsilon^2}{2}(D-m)\right),
\]
so we again estimate the probability of each $A_p$ using a union bound:
\begin{align*}
\mathrm{Pr}(A_p)&\le |\mathcal{N}_\delta \cap (B_{1.5\delta}(p)\setminus \{p\})|\cdot \exp\left(-\frac{\epsilon^2}{2}(D-m)\right)\\
&\le 4^{n_h}\cdot \exp\left(-\frac{\epsilon^2}{2}(D-m)\right).
\end{align*}
Observing that for each $p\in \mathcal{N}_\delta$, $A_p$ is mutually independent with the collection of events $\{A_q:q\in \mathcal{N}_\delta \setminus B_{3\delta}(p)\}$, which are all the $A_q$ except possibly $|\mathcal{N}_\delta\cap B_{3\delta}(p)|\le 7^{n_h}$ of them, we see that our choice of parameters $D-m\ge 2^{4n_h+7}n_h$ and $\epsilon=\frac{1}{4^{n_h+1}}$ allows us to apply the Lov\'{a}sz local lemma because
\begin{align*}
e\cdot 7^{n_h}\cdot 4^{n_h}\cdot \exp\left(-\frac{\epsilon^2}{2}(D-m)\right)&=e\cdot 28^{n_h}\cdot \exp\left(-\frac{\epsilon^2}{2}(D-m)\right)\\
    &<e^{4n_h}\cdot \exp\left(-\frac{1}{2^{4n_h+5}}(2^{4n_h+7}n_h)\right)=1\\
    &\quad (\mathrm{since~} 1+\log(28)n_h<4n_h\mbox{ as }n_h\ge 2).
\end{align*}
By the same limiting argument, we conclude the existence of an assignment $\{v_{m+1}'(p)\}_{p\in \mathcal{N}_\delta}$ such that for any distinct $p,q\in \mathcal{N}_\delta$ with $d(p,q)<1.5\delta$ we have $|v_{m+1}'(p)\cdot v_{m+1}'(q)|\le \epsilon$.

As in the previous proof, we now ``interpolate'' the discrete vector field $\{v_{m+1}'(p)\}_{p\in \mathcal{N}_\delta}$ to produce a vector field $\{\tilde{v}_{m+1}(p)\}_{p\in G} $ defined on the entirety of $G$, which nearly has the desired properties. It is not difficult to define a ``quadratic'' partition of unity $\{\phi_q\}_{q\in \mathcal{N}_\delta}$, i.e., functions $\phi_q:G\to [0,1]$ defined for each $q\in \mathcal{N}_\delta$ such that
\begin{itemize}
    \item $\mathrm{supp~} \phi_q\subset B_{1.5\delta}(q)$, $q\in \mathcal{N}_\delta$,
    \item $\sum_{q\in \mathcal{N}_\delta}\phi_q^2=1$ on $G$,
    \item $\|\phi_q\|_{C^0}+\sum_{k=1}^j R_k\|\nabla^k\phi_q\|_{C^0}\le \|\phi_q\|_{C^{j+1}}\lesssim_{G,m} 1$, uniformly over $q\in \mathcal{N}_\delta$.
\end{itemize}
We now interpolate the $\{v_{m+1}'(p)\}_{p\in \mathcal{N}_\delta}$ using $\{\phi_q\}_{q\in \mathcal{N}_\delta}$:
\begin{equation}\label{interpolate2}
\tilde{v}_{m+1}(p)=\sum_{q\in \mathcal{N}_\delta}\phi_q(p)v_{m+1}'(q),\quad p\in G.
\end{equation}
We claim that $\tilde{v}_{m+1}$ nearly satisfies the required properties, namely it satisfies
\begin{enumerate}[leftmargin=*]
    \item $\frac{3}{4}\le |\tilde{v}_{m+1}(p)|^2\le \frac 54$ for each $p\in G$,
    \item $|\tilde{v}_{m+1}(p)\cdot v_i(p)|\le \frac 1{4m}$ for each $p\in G$ and $i=1,\cdots,m$,
    \item $\|\tilde{v}_{m+1}\|_{C^0}+\sum_{k=1}^j R_k\|\nabla^k\tilde{v}_{m+1}\|_{C^0}\lesssim_{G,j,m} 1$.
\end{enumerate}
Indeed, by the support property of $\phi_q$, we see that the summation in \eqref{interpolate2} is locally finite:
\[
\tilde{v}_{m+1}(p)=\sum_{q\in \mathcal{N}_\delta\cap B_{1.5\delta}(p)}\phi_q(p)v_{m+1}'(q).
\]
From this, property (3) above immediately follows. We can verify property (1) using the near-orthogonality and the fact that the sum of squares of the $\phi_q$ is 1:
\begin{align*}
    \left||\tilde{v}_{m+1}(p)|^2-1\right|&=\left|\sum_{\substack{q,q'\in \mathcal{N}_\delta\cap B_{1.5\delta}(p) \\ q\neq q'}}\phi_q(p)\phi_{q'}(p)v_{m+1}'(q)\cdot v_{m+1}'(q')\right|\\
    &\le \sum_{\substack{q,q'\in \mathcal{N}_\delta\cap B_{1.5\delta}(p) \\ q\neq q'}}\phi_q(p)\phi_{q'}(p)\left|v_{m+1}'(q)\cdot v_{m+1}'(q')\right|\\
    &\le \sum_{\substack{q,q'\in \mathcal{N}_\delta\cap B_{1.5\delta}(p) \\ q\neq q'}}\frac{\phi_q(p)^2+\phi_{q'}(p)^2}{2}\epsilon \le |\mathcal{N}_\delta\cap B_{1.5\delta}(p)|\epsilon\le 4^{n_h}\epsilon=\frac 14.
\end{align*}
Finally, we verify property (2). We have, for each $p\in G$ and $i=1,\cdots, m$,
\[
\tilde{v}_{m+1}(p)\cdot v_i(p) =\sum_{q\in \mathcal{N}_\delta\cap B_{1.5\delta}(p)}\phi_q(p)v_{m+1}'(q)\cdot v_i(p).
\]
We observe that for each $q\in \mathcal{N}_\delta\cap B_{1.5\delta}(p)$, we can estimate
\begin{align*}
|v_{m+1}'(q)\cdot v_i(p)|&\le |v_{m+1}'(q)\cdot v_i(q)|+|v_{m+1}'(q)\cdot (v_i(p)-v_i(q))|\\
&\le 0+ | v_i(p)-v_i(q)|\le 1.5\delta,
\end{align*}
where we have used the orthogonality of $v_{m+1}'(q)$ against $v_i(q)$ and the fact that $\|\nabla v_i\|_{C^0}\le 1$. From these facts and Cauchy--Schwarz, we have the bound
\begin{align*}
    |\tilde{v}_{m+1}(p)\cdot v_i(p)| \le \sum_{q\in \mathcal{N}_\delta\cap B_{1.5\delta}(p)}\phi_q(p)\cdot 1.5\delta\le |\mathcal{N}_\delta\cap B_{1.5\delta}(p)|^{1/2}\cdot 1.5\delta\le 2^{n_h}\cdot 1.5\delta=\frac 1{4m}.
\end{align*}
This finishes the verification of properties (1) to (3). Now we can use Gram-Schmidt orthogonalization to obtain a true $v_{m+1}$:
\[
v_{m+1}(p):=\frac{\tilde{v}_{m+1}(p)-\sum_{i=1}^m \tilde{v}_{m+1}(p)\cdot v_i(p)}{|\tilde{v}_{m+1}(p)-\sum_{i=1}^m \tilde{v}_{m+1}(p)\cdot v_i(p)|},\quad p\in G.
\]
This is well-defined because
\[
|\tilde{v}_{m+1}(p)-\sum_{i=1}^m \tilde{v}_{m+1}(p)\cdot v_i(p)|\ge |\tilde{v}_{m+1}(p)|-\sum_{i=1}^m |\tilde{v}_{m+1}(p)\cdot v_i(p)|\ge \frac{\sqrt{3}}{2}-\frac 14>0,
\]
and $v_{m+1}(p)$ clearly forms an orthonormal system along with $v_1(p),\cdots,v_m(p)$. The required regularity of $v_{m+1}$ follows from the algebra property for norms defined by log-concave sequences.
\end{proof}

From the proof of Theorem \ref{Cj-lifting}, we can easily see that Theorem \ref{Lip-lifting} provides a partial positive answer to Question \ref{general-lift} for many other function spaces as well, described in the following theorem.
\begin{theorem}\label{partialpositiveanswer}
Let $(X,d)$ be a $K$-doubling metric space ($K\ge 2$), and let $\mathcal{F}_1\subset \mathrm{Lip}(X, \mathbb{R})\coloneqq \{f:X\to\mathbb{R}|f \mbox{ is Lipschitz}\}$, $\mathcal{F}\subset \mathrm{Lip}(X, \mathbb{R}^D)\coloneqq \{f:X\to\mathbb{R}^D|f \mbox{ is Lipschitz}\}$ be spaces of functions on $X$ such that
\begin{enumerate}[leftmargin=*]
    \item The function spaces $(\mathcal{F}_1,\|\cdot\|_{\mathcal{F}_1})$, $(\mathcal{F},\|\cdot\|_\mathcal{F})$ are normed linear spaces.
    \item $\|f\|_{\mathrm{Lip}(X,\mathbb{R})} \lesssim \|f\|_{\mathcal{F}_1}$ for all $f \in \mathcal{F}$ and $\|v\|_{\mathrm{Lip}(X,\mathbb{R}^D)} \lesssim \|v\|_{\mathcal{F}}$ for all $v \in \mathcal{F}$.
    \item (closure under algebraic operations)
    For $f\in \mathcal{F}_1,v,w\in \mathcal{F}$ we have $fv\in \mathcal{F}$ and $v\cdot w\in \mathcal{F}_1$ with
    \[
    \|fv\|_\mathcal{F}\lesssim \|f\|_{\mathcal{F}_1}\|v\|_\mathcal{F}, \quad \|v\cdot w\|_{\mathcal{F}_1}\lesssim \|v\|_\mathcal{F}\|w\|_\mathcal{F}.
    \]
    Also, if $f\in \mathcal{F}_1$, $f(x)\ge c>0$ for all $x\in X$, then $1/f \in \mathcal{F}_1$ with
    \[
    \|1/f\|_{\mathcal{F}_1}\lesssim_c \|f\|_{\mathcal{F}_1}.
    \]
    If $v\in \mathcal{F}$, $|v(x)|\ge c>0$ for all $x\in X$, then $|v|\in \mathcal{F}_1$ and
    \[
    \||v|\|_{\mathcal{F}_1}\lesssim_c \|v\|_{\mathcal{F}}.
    \]
    \item (density in the space of Lipschitz functions) For any $\delta>0$ and $w \in \mathrm{Lip}(X, \mathbb{R}^D)$ with $\|w\|_{\mathrm{Lip}}\le 1$ there exists $v\in \mathcal{F}(X)$ such that
    \[
    \|v-w\|_{L^\infty}<\delta,\quad \|v\|_{\mathcal{F}}\lesssim_\delta \|w\|_{L^\infty}+\|w\|_{\mathrm{Lip}}.
    \]
\end{enumerate}

Let $1\le m\le D-224 K^4\log K$. If $v_i:X\to\bbs^{D-1}$, $\|v_i\|_{\mathcal{F}}\le 1$, $i=1,\cdots,m,$ form an orthonormal system at each point $p\in X$, then there exists $v_{m+1}:X\to \bbs^{D-1}$ with $\|v_{m+1}\|_\mathcal{F}\lesssim_{K,m,\mathcal{F}} 1$ such that $v_1,\cdots,v_m$ along with $v_{m+1}$ form an orthonormal system at each point $p \in X$.
\end{theorem}
In words, Theorem \ref{partialpositiveanswer} says that as long as the function space $\mathcal{F}$ behaves ``better than'' the space of Lipschitz functions in the sense of (2), is closed under the operations used in the Gram-Schmidt process in the sense of (3), and can approximate Lipschitz functions arbitrarily well (4) (for Carnot groups we accomplished this by convolving with a smooth mollifier), then we have a positive answer to Question \ref{general-lift}.

By applying Theorem \ref{Cj-lifting} several times, we obtain the following.
\begin{corollary}\label{multiple-Cj-lifting}
Let $1\le m\le D-2^{4n_h+7}n_h-m'+1$, $j\ge 1$, and let $\{R_i\}_{i=1}^j$ be a log-concave sequence of positive reals. Let $v_1,\cdots,v_m:G\to \bbs^{D-1}$ form an orthonormal system at each point, with the uniform regularity bound
\[
\sum_{k=1}^j R_k\|\nabla^k v_i\|_{C^0}\le 1,\quad i=1,\cdots, m.
\]
Then there exist $v_{m+1},\cdots,v_{m+m'}:G\to \bbs^{D-1}$ such that $v_1,\cdots,v_m$ along with $v_{m+1},\cdots,v_{m+m'}$ form an orthonormal system at each point, and
\[
\sum_{k=1}^j R_k\|\nabla^k v_i\|_{C^0}\lesssim_{G,m,m',j} 1,\quad i=m+1,\cdots, m+m'.
\]
\end{corollary}
\begin{remark}
The need for a normal field in embedding problems dates back to \cite{nash1954c1}, where the existence of a section was demonstrated using a homotopy argument from \cite{steenrod1999topology} based on the fact that the base space is contractible. However, such an argument in this situation will fail to control the regularity of the section at points of the base far from the contraction point. In \cite[Section 8]{tao2021embedding} this was achieved for the Heisenberg group $\bbh^3$ by imposing a ``uniform'' CW-structure on $\bbh^3$ and then inductively defining the section starting from low-dimensional skeleta. In the inductive step in \cite{tao2021embedding}, one has to use the fact that the homotopy groups $\pi_i(\bbs^n)$ vanish for $i<n$, which necessitates the ``dimension gap'' $D-m-1\ge 3$. Here we have chosen the section over a 0-skeleton such that the section is locally roughly an orthonormal set and obtained a section by directly interpolating. This allows us to avoid the need for a CW structure by only using the doubling property of the base space, but we have thus increased the dimension gap exponentially.
\end{remark}


\section{Main Iteration Lemma}\label{sec:iterationlemma}
\setcounter{equation}{0}

The starting point of the iterative construction is a function that oscillates at a fixed scale while satisfying a suitable freeness property. (This will also be an ingredient in the inductive step of the iterative construction, when we pass from a larger scale $A^{m+1}$ down to a smaller scale $A^m$.) We begin by constructing this single oscillating function on the Carnot group $G$. (See also Proposition 5.2 of \cite{tao2021embedding} for an analogous statement.)

\begin{proposition}\label{FirstStep}
There exists a smooth map $\phi^0:G\to \bbr^{14^{n_h}}$ with the following properties.
\begin{enumerate}[leftmargin=*]
    \item (smoothness) For all $j\ge 1$,
    \[
    \|\phi^0\|_{C^j}\lesssim_{G,j} 1.
    \]
    \item (locally free embedding) For all $p\in G$, we have
    \begin{equation}\label{firstfree}
    \left|\bigwedge_{\substack{W=X_{r_1,j_1}\cdots X_{r_m,j_m}  \\  1\le m\le s, (r_1,j_1)\preceq \cdots \preceq (r_{m},j_{m})}}W\phi^0(p)\right|\ge 1.
    \end{equation}
\end{enumerate}
\end{proposition}
\begin{proof}
Inspired by the Veronese-type embedding used in \cite{tao2021embedding}, we first begin with the function 
\[
\varphi^0:G\to \bigoplus_{r=1}^s \otimes^r \bbr^{n}=\bbr^{\sum_{r=1}^s\binom{n}{r}},\quad
\varphi^0(\exp(x))=\bigoplus_{r=1}^s \frac{1}{r!}\otimes^r x,\quad x\in \mathfrak{g},
\]
where in the image we identify $\mathfrak{g}$ with $\bbr^n$ via $\sum_{r=1}^s\sum_{i=1}^{k_r}x_{r,i}X_{r,i}\leftrightarrow \sum_{r=1}^s\sum_{i=1}^{k_r}x_{r,i}f_{r,i}$. In these coordinates, we would have
\[
\varphi^0\left(\exp\left(\sum_{r=1}^s\sum_{i=1}^{k_r}x_{r,i}X_{r,i}\right)\right)=\sum_{r=1}^s \frac{1}{r!}\otimes^r \left(\sum_{r=1}^s\sum_{i=1}^{k_r}x_{r,i}f_{r,i}\right),\quad x_{r,i}\in \bbr.
\]
Recall that
\[
X_{r,i}=\frac{\partial}{\partial x_{r,i}}+\sum_{r'>r}^s\sum_{j=1}^{k_{r'}}(\mbox{polynomial in }\{x_{r'',i'}\}_{r''<r'}\mbox{ of weighted degree }r-r')\frac{\partial}{\partial x_{r',j}},
\]
and recall that $X_{r,i}$ acts on polynomials by reducing the weighted degree by $r$. Consequently, for each $(r_1,j_1)\preceq \cdots \preceq (r_{m},j_{m})$, the differential operator $X_{r_1,j_1} \cdots X_{r_{m},j_{m}}$ reduces weighted degrees by $\sum_{i=1}^m r_i$, and so
\begin{align*}
X_{r_1,j_1} \cdots X_{r_{m},j_{m}}\varphi^0=&f_{r_1,j_1}\otimes \cdots \otimes f_{r_{m},j_{m}}\\
&\qquad+\sum_{\substack{(r'_1,j'_1)\preceq \cdots\preceq (r'_{m'},j'_{m'})  \\ r'_1+\cdots+r'_{m'}> r_1+\cdots+r_m \\ m'\le s}}(\mbox{polynomial of degree }\Sigma r'_i-\Sigma r_i)f_{r'_1,j'_1}\otimes \cdots \otimes f_{r'_{m'},j'_{m'}}.
\end{align*}
Thus, when taking the wedge product of all these differentials, we can rearrange the differentials in order of their degree and cancel the higher degree terms. This leads to
\[
\left|\bigwedge_{\substack{W=X_{r_1,j_1}\cdots X_{r_m,j_m}  \\  1\le m\le s, (r_1,j_1)\preceq \cdots \preceq (r_{m},j_{m})}}W\varphi^0(p)\right|=\left|\bigwedge_{\substack{W=X_{r_1,j_1}\cdots X_{r_m,j_m}  \\  1\le m\le s, (r_1,j_1)\preceq \cdots \preceq (r_{m},j_{m})}} f_{r_1,j_1}\otimes \cdots \otimes f_{r_{m},j_{m}}\right|= 1.
\]

Hence, we can create a mapping with the required freeness property. It remains to modify this construction so that we have bounded $C^j$ norms as well. This can be done by viewing $\varphi^0$ as a mapping that works locally around the origin and then pasting together mollifications of $\varphi^0$, using the doubling property of $G$ to ensure that the pasting process only increases the $C^j$ norm by a bounded factor. (This construction is inspired by \cite{assouad1983plongements}.)

Take a smooth function $\eta:G\to [0,1]$ which is identically 1 on the unit ball $B_1$ and which vanishes on $B_{1.5}^c$. Then the function $\varphi^1:G\to \bbr^{\sum_{r=1}^s\binom{n}{r}}$ defined by $\varphi^1=\eta \varphi^0$ has bounded $C^j$ norm for all $j\ge 1$, and satisfies $\left|\bigwedge_{W= X_{r_1,j_1},\cdots,X_{r_m,j_m}}W\varphi^1\right|= 1$ on $B_1$. Now take a maximal 1-net $\mathcal{N}_1$  of $G$. We claim that we can take the decomposition
\[
\mathcal{N}_1=\bigsqcup_{a=1}^{7^{n_h}}\mathcal{N}^a_3,
\]
where $n_h$ is the Hausdorff dimension of $G$, and each $\mathcal{N}^a_3$ is a 3-net of $G$. (This is a standard coloring argument.) Indeed, each point $g\in \mathcal{N}_1$ has at most $7^{n_h}-1$ points in $\mathcal{N}_1$ in its 3-neighborhood by \eqref{locbd}; having inductively assigned a finite number of points of $\mathcal{N}_1$ in one of the $7^{n_h}$ sets \linebreak$\{\mathcal{N}^a_3\}_{a=1}^{7^{n_h}}$, any other point of $\mathcal{N}_1$ can be assigned to one of them consistently.

We now define $\phi^0:G\to\bbr^{7^{n_h}\sum_{r=1}^s\binom{n}{r}}$ as
\[
\phi^0(p)\coloneqq \bigoplus_{a=1}^{7^{n_h}}\sum_{g\in \mathcal{N}^a_3}\varphi^1(g^{-1}p),\quad p\in G.
\]
Then this satisfies the given properties because for each $a=1,\cdots,7^{n_h}$, the function
\[
\sum_{g\in \mathcal{N}^a_3}\varphi^1(g^{-1}p)
\]
is a sum of smooth compactly supported functions whose supports $gB_{1.5}$ are disjoint, thus has bounded $C^j$ norm, and satisfies the wedge product bound
\[
\left|\bigwedge_{\substack{W=X_{r_1,j_1}\cdots X_{r_m,j_m}  \\  1\le m\le s, (r_1,j_1)\preceq \cdots \preceq (r_{m},j_{m})}}W\sum_{g\in \mathcal{N}^a_3}\varphi^1(g^{-1}p)\right|\ge 1,\quad p\in \mathcal{N}^a_3B_1.
\]
As $\{\mathcal{N}^a_3B_1\}_{a=1}^{7^{n_h}}$ covers $G$, we see that $\varphi^0$ satisfies \eqref{firstfree}. As $\sum_{r=1}^s\binom{n}{r}\le 2^n\le 2^{n_h}$ we are done by composing $\phi^0$ with an embedding $\bbr^{7^{n_h}\sum_{r=1}^s\binom{n}{r}}\to \bbr^{14^{n_h}}$.
\end{proof}

\begin{remark}
The method of proof of Proposition \ref{FirstStep} in \cite{tao2021embedding} for the case $G=\bbh^3$, or more generally for the case when $G$ admits a cocompact lattice $\Gamma$, is the following. Since the nilmanifold $G/\Gamma$ is a smooth compact $n$-dimensional manifold, by the strong Whitney immersion theorem \cite{whitney1944singularities} there exists a smooth immersion $G/\Gamma\to \bbr^{2n-1}$. By precomposing with the projection $G\to G/\Gamma$, one obtains a map $\varphi^1:G\to \bbr^{2n-1}$ that satisfies the weaker freeness property
\[
\left|\bigwedge_{r=1}^s\bigwedge_{i=1}^{k_r}X_{r,i}\varphi^1\right|\gtrsim 1
\]
while having bounded $C^j$ norms due to the compactness of $G/\Gamma$. We can then obtain the stronger freeness property \eqref{firstfree} by composing $\varphi^1$ with a Veronese-type embedding, say
\[
\phi^1:G\to\bigoplus_{r=1}^s \otimes^r \bbr^{2n-1}=\bbr^{\sum_{r=1}^s\binom{2n-1}{r}},\quad \phi^1\coloneqq\bigoplus_{r=1}^s \frac{1}{r!}(\varphi^1)^{\otimes r}.
\]
(One can indeed prove the stronger freeness property by using a simple change of coordinates argument.) There is an exponential saving in the target dimension: the target dimension, in this case, is polynomial in the topological dimension of $G$, whereas the target dimension in Proposition \ref{FirstStep} is exponential in the Hausdorff dimension of $G$. Perhaps one could improve the target dimension in Proposition \ref{FirstStep} to be polynomial in the Hausdorff dimension of $G$, say by using random nets and partitions as in \cite{naor2010assouad} and being more careful about how we paste the different embeddings.
\end{remark}

Now we state the inductive step. Having constructed a map $\psi:G\to\bbr^D$ which ``represents'' the geometry of $G$ at scale $A^{m+1}$ and above, we need to construct a correction $\phi:G\to\bbr^D$ which oscillates at scale $A^m$ such that $\psi+\phi$ represents the geometry of $G$ at scale $A^m$. When we say that a mapping represents the geometry of $G$ at scale $A^m$, we mean that its $C_{A^m}^{m^*,2/3}$ norm is controlled and that it has good freeness properties. To make a viable induction argument, we need to show that the quantitative controls on the $C_{A^m}^{m^*,2/3}$ norm and the freeness properties are preserved when we pass from $\psi$ to $\psi+\phi$. By rescaling, we may simply assume $m=0$. The precise statement for the inductive step is as follows.

\begin{proposition}[main iteration step]\label{KeyIteration2}
Let $M$ be a real number with
\[
M\ge C_0^{-1},
\]
and let ${m^*}\ge \max\{3,s^2\}$. Suppose a map $\psi:G\to \bbr^{128\cdot 23^{n_h}}$ obeys the following estimates:
\begin{enumerate}[leftmargin=*]
    \item (nondegenerate first derivatives) For any $p\in G$, we have
    \begin{equation}\label{It-1-again}
        C_0^{-1}M\le |X_i\psi(p)|\le C_0 M,\quad i=1,\cdots,k,
    \end{equation}
    and 
    \begin{equation}\label{It-2-again}
        \left|\bigwedge_{m=1}^i X_{j_m}\psi(p)\right|\ge C_0^{-i^2-i+2} \prod_{m=1}^i \left|X_{j_m}\psi(p)\right|
    \end{equation}
    for $2\le i\le k$ and $1\le j_1<\cdots<j_i\le k$.
    \item (locally free embedding) For any $p\in G$, we have
    \begin{equation}\label{It-3-again}
        \left|\bigwedge_{\substack{W=X_{r_1,j_1}\cdots X_{r_m,j_m}  \\  1\le m\le s, (r_1,j_1)\preceq \cdots \preceq (r_{m},j_{m})}}W\psi(p)\right|\ge C_0^{-2k^2-7k+2}A^{-\sum_{j=2}^s (j-1)\binom{n+j-1}{j}}\prod_{j=1}^k \left|X_j\psi(p)\right|.
    \end{equation}
    \item (H\"older regularity at scale $A$) We have
    \begin{equation}\label{It-4-again}
        \|\nabla^2 \psi\|_{C_A^{{m^*}-2,{2/3}}}\le C_0A^{-1}.
    \end{equation}
\end{enumerate}
Then there exists a map $\phi:G\to\bbr^{128\cdot 23^{n_h}}$ such that $\phi$ obeys the following estimates:
\begin{enumerate}[leftmargin=*]
    \item (regularity at scale 1) We have
    \begin{equation}\label{newholder}
        \|\phi\|_{C^{{m^*},{2/3}}}\lesssim_G 1.
    \end{equation}
    \item (orthogonality) We have
    \begin{equation}\label{neworthog}
        B(\phi,\psi)=0.
    \end{equation}
    \item (nondegenerate first derivatives) We have
    \begin{equation}\label{newnondeg}
        |X_i\phi|\gtrsim_G 1,\quad i=1,\cdots,k.
    \end{equation}
\end{enumerate}
and the sum $\psi+\phi$ obeys the following regularity estimates:
\begin{enumerate}[leftmargin=*]
    \item (nondegenerate first derivatives) For any $p\in G$, we have
    \begin{equation}\label{It-1-again-lower-scale}
        C_0^{-1}\sqrt{M^2+1}\le |X_i(\psi+\phi)(p)|\le C_0 \sqrt{M^2+1},\quad i=1,\cdots,k,
    \end{equation}
    and 
    \begin{equation}\label{It-2-again-lower-scale}
        \left|\bigwedge_{m=1}^i X_{j_m}(\psi+\phi)(p)\right|\ge C_0^{-i^2-i+2} \prod_{m=1}^i \left|X_{j_m}(\psi+\phi)(p)\right|
    \end{equation}
    for $2\le i\le k$ and $1\le j_1<\cdots<j_i\le k$.
    \item (locally free embedding) For any $p\in G$, we have
    \begin{equation}\label{It-3-again-lower-scale}
        \left|\bigwedge_{\substack{W=X_{r_1,j_1}\cdots X_{r_m,j_m}  \\  1\le m\le s, (r_1,j_1)\preceq \cdots \preceq (r_{m},j_{m})}}W(\psi+\phi)(p)\right|\ge C_0^{-2k^2-7k+3}A^{-\sum_{j=2}^s (j-1)\binom{n+j-1}{j}}\prod_{m=1}^i \left|X_{j_m}(\psi+\phi)(p)\right|.
    \end{equation}
    \item (H\"older regularity at scale $1$) We have
    \begin{equation}\label{It-4-again-lower-scale}
        \|\nabla^2 (\psi+\phi)\|_{C^{{m^*}-2,{2/3}}}\le C_0.
    \end{equation}
\end{enumerate}
\end{proposition}

\begin{remark}
The dimension $128\cdot 23^{n_h}$ will result from Proposition \ref{FirstStep} along with Corollary \ref{multiple-Cj-lifting}; see Lemma \ref{GoodIsom} below.
\end{remark}

\begin{remark}
Proposition \ref{KeyIteration2} shows that if $\psi$ is a map with the regularity and freeness properties \eqref{It-1-again}-\eqref{It-4-again} at scale $A$, then we can find a correction $\phi$ so that $\psi+\phi$ is a map with the same regularity and freeness properties \eqref{It-1-again-lower-scale}-\eqref{It-4-again-lower-scale} but at scale $1$. In particular, the constants in the freeness properties are the same. We will make this possible by creating a ``hierarchy'' of freeness properties. Namely, the freeness property \eqref{It-2-again-lower-scale} for $i$-fold wedge products of horizontal derivatives of $\psi+\phi$ will be based on the freeness property \eqref{It-2-again} for $(i-1)$-fold wedge products of horizontal derivatives of $\psi$. Also, the freeness property \eqref{It-3-again-lower-scale} for the wedge product of up to $s$-order derivatives of $\psi+\phi$ will be based on the freeness property \eqref{It-2-again} for $k$-fold wedge products of horizontal derivatives of $\psi$. Thus, we do not lose constants when passing from $\psi$ to $\psi+\phi$, which will allow us to close the iteration.
\end{remark}

Proposition \ref{KeyIteration2} will be a consequence of the following lemma, which is a generalization of Proposition 5.1 of \cite{tao2021embedding} for the case $G=\bbh^3$.

\begin{lemma}[main iteration lemma]\label{KeyIteration}
Let $M$ be a real number with
\[
M\ge C_0^{-1},
\]
and let ${m^*}\ge \max\{3,s^2\}$. Suppose a map $\psi:G\to \bbr^{128\cdot 23^{n_h}}$ obeys the following estimates:
\begin{enumerate}[leftmargin=*]
    \item (nondegenerate first derivatives) For any $p\in G$, we have
    \begin{equation}\label{It-1}
        C_0^{-1}M\le |X_i\psi(p)|\le C_0 M,\quad i=1,\cdots,k,
    \end{equation}
    and 
    \begin{equation}\label{It-2}
        \left|\bigwedge_{m=1}^i X_{j_m}\psi(p)\right|\ge C_0^{-i^2-2i+2} M^{i},
    \end{equation}
    for $2\le i\le k$ and $1\le j_1<\cdots<j_i\le k$.
    \item (locally free embedding) For any $p\in G$, we have
    \begin{equation}\label{It-3}
        \left|\bigwedge_{\substack{W=X_{r_1,j_1}\cdots X_{r_m,j_m}  \\  1\le m\le s, (r_1,j_1)\preceq \cdots \preceq (r_{m},j_{m})}}W\psi(p)\right|\ge C_0^{-2k^2-8k+2}A^{-\sum_{j=2}^s (j-1)\binom{n+j-1}{j}}M^{k}.
    \end{equation}
    \item (H\"older regularity at scale $A$) We have
    \begin{equation}\label{It-4}
        \|\nabla^2 \psi\|_{C_A^{{m^*}-2,\alpha}}\le C_0A^{-1}.
    \end{equation}
\end{enumerate}
Then there exists a map $\phi:G\to\bbr^{128\cdot 23^{n_h}}$ obeying the following estimates.
\begin{enumerate}[leftmargin=*]
    \item (nondegenerate first derivatives) For any $p\in G$, we have
    \begin{equation}\label{It-5}
        |X_i\phi(p)|\gtrsim_G 1
    \end{equation}
    and
    \begin{equation}\label{It-6}
        \left|\bigwedge_{m=1}^i X_{j_m}(\psi+\phi)(p)\right|^2-\left|\bigwedge_{m=1}^i X_{j_m}\psi(p)\right|^2\ge C_0^{-2i^2+5}M^{2(2i-1)},
    \end{equation}
    for $2\le i\le k$ and $1\le j_1<\cdots<j_i\le k$.
    \item (locally free embedding) For any $p\in G$, we have
    \begin{equation}\label{It-7}
        \left|\bigwedge_{\substack{W=X_{r_1,j_1}\cdots X_{r_m,j_m}  \\  1\le m\le s, (r_1,j_1)\preceq \cdots \preceq (r_{m},j_{m})}}W(\psi+\phi)(p)\right|\ge C_0^{-2k^2-5k+3}M^{k}.
    \end{equation}
    \item (H\"older regularity at scale 1) We have
    \begin{equation}\label{It-8}
        \|\phi\|_{C^{{m^*},\alpha}}\lesssim_G 1.
    \end{equation}
    \item (orthogonality) We have
    \begin{equation}\label{It-9}
        B(\phi,\psi)=0.
    \end{equation}
\end{enumerate}
\end{lemma}

We now show why Proposition \ref{KeyIteration2} follows from Lemma \ref{KeyIteration}.

\begin{proof}[proof of Proposition \ref{KeyIteration2} assuming Lemma \ref{KeyIteration}]
Let $\psi$ be as in Proposition \ref{KeyIteration2}.
One can easily verify the hypotheses of Lemma \ref{KeyIteration} for $\psi$, as \eqref{It-2} and \eqref{It-3} each follow from \eqref{It-2-again} and \eqref{It-3-again} combined with \eqref{It-1-again}. Thus, by Lemma \ref{KeyIteration}, there exists a function $\phi:G\to\bbr^{128\cdot 23^{n_h}}$ that satisfies \eqref{newholder}, \eqref{neworthog}, and \eqref{newnondeg} (as these are exactly  \eqref{It-8}, \eqref{It-9}, and \eqref{It-5}, respectively) and the following (which are just restatements of \eqref{It-7} and \eqref{It-8}):
\begin{enumerate}[leftmargin=*]
    \item (nondegenerate first derivatives)
    For $p\in G$, $2\le i\le k$ and $1\le j_1<\cdots<j_i\le k$,
    \begin{equation}\label{It-6-again}
        \left|\bigwedge_{q=1}^i X_{j_q}(\psi+\phi)(p)\right|^2-\left|\bigwedge_{q=1}^i X_{j_q}\psi(p)\right|^2\ge C_0^{-2i^2+5}M^{2(i-1)}.
    \end{equation}
    \item (locally free embedding) For any $p\in G$, we have
    \begin{equation}\label{It-7-again}
        \left|\bigwedge_{\substack{W=X_{r_1,j_1}\cdots X_{r_m,j_m}  \\  1\le m\le s, (r_1,j_1)\preceq \cdots \preceq (r_{m},j_{m})}}W(\psi+\phi)(p)\right|\ge C_0^{-2k^2-5k+3}M^{k}.
    \end{equation}
\end{enumerate}
We now verify that the map $\psi+\phi$ satisfies the properties \eqref{It-1-again-lower-scale} to \eqref{It-4-again-lower-scale}.

\vspace{0.1in}
\noindent \underline{Verification of \eqref{It-1-again-lower-scale}}:
    By \eqref{neworthog}, we have $|X_i(\psi+\phi)|^2=|X_i\psi|^2+|X_i\phi|^2$. But by \eqref{newnondeg} and \eqref{newholder}, we have $|X_i\phi|\sim_G 1$, so we have 
    \begin{equation}\label{C_0_hierarchy-1}
        C_0^{-1}\le |X_i\phi|\le C_0.
    \end{equation}
    (This uses our hierarchy of constants in subsection \ref{sec:hierarchy}, by choosing $C_0$ depending on $G$.)
    Combining these facts with \eqref{It-1-again} we obtain \eqref{It-1-again-lower-scale}.

\vspace{0.1in}
\noindent \underline{Verification of \eqref{It-2-again-lower-scale}}:
    This follows from \eqref{It-6-again}:
    \begin{align}\label{C_0_hierarchy-2}
    \begin{aligned}
        \left|\bigwedge_{q=1}^i X_{j_q}(\psi+\phi)\right|^2&\ge \left|\bigwedge_{q=1}^i X_{j_q}\psi\right|^2+ C_0^{-2i^2+5}M^{2(i-1)}\\
        &\ge C_0^{-2i^2-2i+4}\prod_{q=1}^i\left| X_{j_q}\psi\right|^2+ C_0^{-2i^2+5}M^{2(i-1)}\\
        &\ge C_0^{-2i^2-2i+4}\prod_{q=1}^i\left(\left| X_{j_q}\psi\right|^2+|X_{j_q}\phi|^2\right)\\
        &= C_0^{-2i^2-2i+4}\prod_{q=1}^i| X_{j_q}(\psi+\phi)|^2
    \end{aligned}
    \end{align}
    where in the third inequality we used $|X_{j_q}\psi|\le C_0M$ and $|X_{j_q} \phi|\lesssim_G 1$.

\vspace{0.1in}
\noindent \underline{Verification of \eqref{It-3-again-lower-scale}}:
    This follows from \eqref{It-7-again}:
    \begin{align}\label{C_0_hierarchy-3}
    \begin{aligned}
    \left|\bigwedge_{\substack{W=X_{r_1,j_1}\cdots X_{r_i,j_i}  \\ 1\le i\le s, (r_1,j_1)\preceq \cdots \preceq (r_{i},j_{i})}}W(\psi+\phi)(p)\right|    &\ge C_0^{-2k^2-5k+3}M^{k}\\
    &\ge C_0^{-2k^2-7k+3}\prod_{i=1}^k|X_i(\psi+\phi)(p)|,
    \end{aligned}
    \end{align}
    where in the last inequality we used that $|X_i(\psi+\phi)|^2=|X_i\psi|^2+|X_i\phi|^2\le C_0^2M^2+O_G(1)\le C_0^4M^2$ (recall that $M\ge C_0^{-1}$), and our hierarchy of constants where we choose $C_0$ depending on $G$.

\vspace{0.1in}
\noindent \underline{Verification of \eqref{It-4-again-lower-scale}}:
    This follows from \eqref{It-4-again} and \eqref{newholder}:
    \begin{equation}\label{C_0_hierarchy-4}
        \|\nabla^2 (\psi+\phi)\|_{C^{{m^*}-2,{2/3}}}\le \|\nabla^2 \psi\|_{C_A^{{m^*}-2,{2/3}}}+\|\phi_0\|_{C^{{m^*},{2/3}}}\le C_0,
    \end{equation}
    and our hierarchy of constants where we choose $C_0$ depending on $G$.
\end{proof}

For the rest of the section, we will prove Lemma \ref{KeyIteration}.

Suppose that $\psi$ is as in Lemma \ref{KeyIteration}. We will first construct a solution $\Tilde{\phi}$ to the low-frequency equation \eqref{G-5} as
\begin{equation}\label{IsomCompose}
    \Tilde{\phi}(p)=U(p)\left(\phi^0(p)\right),
\end{equation}
where $\phi^0:G\to \bbr^{14^{n_h}}$  is as in Proposition \ref{FirstStep}, and $U(p):\bbr^{14^{n_h}}\to \bbr^{128\cdot 23^{n_h}}$ is a linear isometry with the following properties, which is constructed using Corollary \ref{multiple-Cj-lifting}. (See also Lemma 9.1 of \cite{tao2021embedding} for an analogous statement for $\bbh^3$.)

\begin{lemma}\label{GoodIsom}
There exists $U:G\to\hom(\bbr^{14^{n_h}},\bbr^{128\cdot 23^{n_h}})$ such that
\begin{enumerate}[leftmargin=*]
    \item For each $p\in G$, $U(p)\in \hom(\bbr^{14^{n_h}},\bbr^{128\cdot 23^{n_h}})$ is an isometry.
    \item For each $p\in G$, $s\in \bbr^{14^{n_h}}$, we have
    \begin{equation}\label{IsomPerp}
        (U(p)(s))\cdot X_iP_{(\le N_0)}\psi(p)=0,~ (U(p)(s))\cdot X_iX_jP_{(\le N_0)}\psi(p)=0,\quad 1\le i, j\le k.
    \end{equation}
    \item We have the smoothness
    \[
    \|\nabla U\|_{C^{{m^*}}}\lesssim_{N_0}\frac{1}{A}.
    \]
\end{enumerate}
\end{lemma}
\begin{proof}
Let $W_1,\cdots, W_{\frac{k(k+3)}{2}+k_2}$ denote the rescaled differential operators
\[
((M^{-1}X_i)_{i=1}^k,(AX_iX_j)_{1\le i\le j\le k},(AX_{2,i})_{i=1}^{k_2}).
\]
Let $w_i(p)\coloneqq W_iP_{(\le N_0)}\psi(p)$ for $1\le i\le \frac{k(k+3)}{2}+k_2$ for all $p\in G$. Then, by \eqref{It-1}, \eqref{It-4}, and Theorem \ref{LP} (3) (more specifically, \eqref{lp-5} with $j=1$, $l=2$, and recalling $m^*\ge 3$),
\begin{equation}\label{N_0_hierarchy-1}
    w_i=W_i\psi+O_G\left(\frac{C_0}{AN_0 M}\right)=O(C_0),\quad i=1,\cdots, k,
\end{equation}
and by \eqref{It-4}  and Theorem \ref{LP} (3) (more specifically, \eqref{lp-5} with $j=2$, $l=3$),
\begin{equation}\label{N_0_hierarchy-2}
    w_i=W_i\psi+O_G\left(\frac{C_0}{AN_0}\right)=O(C_0),\quad i=k+1,\cdots,\frac{k(k+3)}{2}+k_2.
\end{equation}
(Note that we have used our hierarchy of constants in \eqref{N_0_hierarchy-1} and \eqref{N_0_hierarchy-2}, choosing $N_0$ depending on $G$ and $C_0$, while we have only used $A\ge 1$ so far.)

On the other hand, \eqref{It-3} tells us that
\[
    \left|\bigwedge_{i=1,\cdots,\frac{k(k+3)}{2}+k_2}W_i\psi(p)\right|\gtrsim_{C_0} 1,
\]
so by applying the triangle inequality and Cauchy--Schwarz inequality, we conclude that
\[
    \left|\bigwedge_{i=1,\cdots,\frac{k(k+3)}{2}+k_2}w_i(p)\right|\gtrsim_{C_0} 1.
\]
By applying Cauchy--Schwarz again, we conclude that
\begin{equation}\label{wp-lb}
    \left|\bigwedge_{i=1,\cdots,j}w_i(p)\right|\sim_{C_0}  1,\quad j=1,\cdots, \frac{k(k+3)}{2}+k_2.
\end{equation}

Again, using \eqref{It-4} along with Theorem \ref{LP} (3), we can see that 
\begin{equation}\label{N_0_hierarchy-3}
    \|w_i\|_{C_0}+A\|\nabla w_i\|_{C^{{m^*}}} \lesssim_{N_0}1, \quad j=1,\cdots, \frac{k(k+3)}{2}+k_2,
\end{equation}
because for $i=1,\cdots, k,$ we use \eqref{It-4}, and \eqref{lp-1} with $l=2$, $j=m^*$, to see
\begin{align}\label{N_0_hierarchy-4}
    \begin{aligned}
        \|w_i\|_{C_0}+A\|\nabla w_i\|_{C^{{m^*}}}&\stackrel{\mathclap{\eqref{N_0_hierarchy-1}}}{\lesssim} C_0+\frac AM \|\nabla^2 P_{(\le N_0)}\psi\|_{C^{m^*}}\le C_0+\frac {AN_0^{m^*}}M \|\nabla^2 P_{(\le N_0)}\psi\|_{C^{m^*}_{1/N_0}}\\
        & \stackrel{\mathclap{\eqref{lp-1}}}{\lesssim}_{G,m^*} C_0+ \frac {AN_0^{m^*}}M \|\nabla^2\psi\|_{C^0}\stackrel{\eqref{It-4}}{\le} C_0+ \frac {AN_0^{m^*}}M \cdot C_0A^{-1}\lesssim_{N_0}1,
    \end{aligned}
\end{align}
and for $i=k+1,\cdots,\frac{k(k+3)}{2}+k_2,$ we use \eqref{It-4}, and \eqref{lp-1} with $l=3$, $j=m^*$, to see
\begin{align}\label{N_0_hierarchy-5}
    \begin{aligned}
        \|w_i\|_{C_0}+A\|\nabla w_i\|_{C^{{m^*}}}&\stackrel{\mathclap{\eqref{N_0_hierarchy-2}}}{\lesssim} C_0+A^2 \|\nabla^3 P_{(\le N_0)}\psi\|_{C^{m^*}}\le C_0+A^2 N_0^{m^*} \|\nabla^3 P_{(\le N_0)}\psi\|_{C^{m^*}_{1/N_0}}\\
        & \stackrel{\mathclap{\eqref{lp-1}}}{\lesssim}_{G,m^*} C_0+ A^2N_0^{m^*} \|\nabla^3\psi\|_{C^0}\stackrel{\eqref{It-4}}{\le} C_0+ A^2N_0^{m^*} \cdot C_0A^{-2}\lesssim_{N_0}1.
    \end{aligned}
\end{align}
(We have used our hierarchy of constants in \eqref{N_0_hierarchy-3}, \eqref{N_0_hierarchy-4}, and \eqref{N_0_hierarchy-5}, by choosing $N_0$ after $G$ and $C_0$.)
Now, because the norm of \eqref{N_0_hierarchy-3} is of the form \eqref{logconcave}, we can apply the product rule to this norm: for instance,
\[
\left\|\bigwedge_{i=1,\cdots,j}w_i\right\|_{C_0}+A\left\|\nabla\bigwedge_{i=1,\cdots,j} w_i\right\|_{C^{{m^*}}}\lesssim_{N_0}1, \quad j=1,\cdots, \frac{k(k+3)}{2}+k_2.
\]
Now let us consider the orthonormal system $v_1,\cdots, v_{\frac{k(k+3)}{2}+k_2}$ formed by applying the Gram-Schmidt process to the vectors $w_i$, i.e., we inductively define
\[
v_i\coloneqq \frac{|\bigwedge_{j<i}w_j|}{|\bigwedge_{j\le i}w_j|}\left(w_i-\sum_{j<i}(w_i\cdot v_j)v_j\right).
\]
By \eqref{wp-lb} this is well-defined, and by a repeated application of the aforementioned product rule, one can deduce the smoothness
\[
\|v_i\|_{C_0}+A\|\nabla v_i\|_{C^{{m^*}}}\lesssim_{N_0}1,\quad i=1,\cdots, \frac{k(k+3)}{2}+k_2.
\]

We now apply Corollary \ref{multiple-Cj-lifting} with $m=\frac{k(k+3)}{2}$, $m'=14^{n_h}$, $j={m^*}+1$, $R_i=A$ to the above $v_i$. This is possible because
\[
\frac{k(k+3)}{2}+k_2+14^{n_h}+2^{4n_h+7}n_h\le \frac 12 n_h^2+\frac 32 n_h+14^{n_h}+2^{4n_h+7}n_h\le 128\cdot 23^{n_h}\quad(\mathrm{since~} n_h\ge 4).
\]
Thus, we have maps $v_{\frac{k(k+3)}{2}+k_2+1},\cdots,v_{\frac{k(k+3)}{2}+k_2+14^{n_h}}:G\to \bbr^{128\cdot 23^{n_h}}$ such that
\[
\|v_i\|_{C_0}+A\|\nabla v_i\|_{C^{{m^*}}}\lesssim_{N_0}1,\quad i=1,\cdots, \frac{k(k+3)}{2}+k_2+14^{n_h}
\]
and such that $v_1(p),\cdots, v_{\frac{k(k+3)}{2}+k_2+14^{n_h}}(p)$ are orthonormal for all $p\in G$.

Now define $U(p):\mathbb{R}^{14^{n_h}}\to\mathbb{R}^{128\cdot 23^{n_h}}$, $p\in G$, to be the map
\[
U(p)(s)=\sum_{i=1}^{14^{n_h}}s_iv_{\frac{k(k+3)}{2}+k_2+i}(p).
\]
This clearly has the properties (1) and (3) asserted above, and we can also deduce property (2) once we note that $X_iX_j-X_jX_i\in \mathrm{span}\{X_{2,1},\cdots,X_{2,k_2}\}$.
\end{proof}

Continuing with the proof of Lemma \ref{KeyIteration}, let $\Tilde{\phi}$ be as in \eqref{IsomCompose}. By Lemma \ref{GoodIsom} and Proposition \ref{FirstStep}(1), and using our hierarchy of choosing $A$ after $C_0$ and $N_0$, we have
\begin{equation}\label{A_hierarchy_1}
    \|\Tilde{\phi}\|_{C^{{m^*},{2/3}}}\lesssim \|\Tilde{\phi}\|_{C^{{m^*}+1}}\lesssim_G 1,
\end{equation}
where we take $\alpha=\frac 23$. Also, by \eqref{IsomPerp} and the Leibniz rule, it is clear that $\tilde{\phi}$ satisfies \eqref{lowfreq-strongperp} and a fortiori solves the low-frequency equation \eqref{G-5}. It is also clear that $\psi$ satisfies the hypothesis of Corollary \ref{perturbation-cor}. By applying Corollary \ref{perturbation-cor}, there exists a $C^{{m^*},{2/3}}$-function $\phi:G\to\bbr^{128\cdot 23^{n_h}}$ such that
\begin{align}
B(\phi,\psi)&=0,\nonumber\\
    \|\phi-\Tilde{\phi}\|_{C^{{m^*},{2/3}}}&\lesssim_{C_0} A^{2-{m^*}},\label{G-8-again-again}
\end{align}
and, as $\tilde{\phi}$ satisfies \eqref{lowfreq-strongperp},
\begin{equation}\label{NearOrtho}
    \|X_i\phi\cdot X_j\psi\|_{C^0}=\|X_i\phi\cdot X_j\psi-X_i\Tilde{\phi}\cdot X_jP_{(\le N_0)}\psi\|_{C^0}\le A^{1-{m^*}},\quad 1\le i,j\le k.
\end{equation}
It remains to verify conditions \eqref{It-5}-\eqref{It-9}. Conditions \eqref{It-8} and \eqref{It-9} are immediate from the construction. For later use, we note that
\begin{equation}\label{decomp}
X_i\phi=X_i\tilde{\phi}+X_i(\phi-\tilde{\phi})=U(X_i\phi^0)+(X_iU)\left(\phi^0\right)+O_{C_0}\left(A^{2-{m^*}}\right)=U(X_i\phi^0)+O_{N_0}\left(A^{-1}\right)
\end{equation}
for $i=1,\cdots,k$. From this and Proposition \ref{FirstStep}, we immediately have (again using our hierarchy of constants that $A$ is chosen after $N_0$)
\begin{equation}\label{A_hierarchy-3}
    |X_i\phi|\sim_G 1,\quad \left|\bigwedge_{i=1}^k X_i\phi\right|\sim_G 1,
\end{equation}
so in particular \eqref{It-5} immediately follows. It remains to verify \eqref{It-6} and \eqref{It-7}.

\vspace{0.1in}
\noindent \underline{Verification of \eqref{It-6}}:
    Recall from \eqref{NearOrtho} that $|X_i\phi\cdot X_j\psi|\le \frac{1}{A}$ for $i,j=1,\cdots,k$.
    Now we observe that for $1\le j_1<\cdots<j_i\le k$ we have the expansion
    \[
    \bigwedge_{m=1}^i X_{j_m}(\psi+\phi)=\bigwedge_{m=1}^i X_{j_m}\psi+\sum_{n=1}^{2^{i}-1}\bigwedge_{m=1}^i X_{j_m} f_m^n,
    \]
    for some sequence $\{f_m^n\}_{n=1,\cdots,2^i-1,~m=1,\cdots, i}$ of functions, each being either $\phi$ or $\psi$. Note that for each $n$ there must exist some $m$ such that $f_m^n=\phi$. This expansion implies
    \begin{align*}
        &\left|\bigwedge_{m=1}^i X_{j_m}(\psi+\phi)\right|^2-\left|\bigwedge_{m=1}^i X_{j_m}\psi\right|^2=\left|\sum_{n=1}^{2^{i}-1}\bigwedge_{m=1}^i X_{j_m} f_m^n\right|^2+2\sum_{n=1}^{2^{i}-1}\bigg< \bigwedge_{m=1}^i X_{j_m}\psi,\bigwedge_{m=1}^i X_{j_m} f_m^n\bigg>.
    \end{align*}
    For each fixed $n$, the polarized Cauchy--Binet formula shows that $\bigg< \bigwedge_{m=1}^i X_{j_m}\psi,\bigwedge_{m=1}^i X_{j_m} f_m^n\bigg>$ can be represented as the determinant of an $i\times i$ matrix, whose entries are each of magnitude $O_{C_0}(M^2)$ and one of whose columns (the $m$-th column, where $m$ is such that $f_m^n=\phi$) consists of entries of magnitude $O(A^{-1})$ (because of \eqref{NearOrtho}). Thus, the determinant of this $i\times i$ matrix is of magnitude $O_{C_0}(A^{-1}M^{2(i-1)})$, or equivalently $\bigg< \bigwedge_{m=1}^i X_{j_m}\psi,\bigwedge_{m=1}^i X_{j_m} f_m^n\bigg>=O_{C_0}(A^{-1}M^{2(i-1)})$. Summing over all $n$, we obtain
    \[
    2\sum_{n=1}^{2^{i}-1}\bigg< \bigwedge_{m=1}^i X_{j_m}\psi,\bigwedge_{m=1}^i X_{j_m} f_m^n\bigg>=O_{C_0}(A^{-1}M^{2(i-1)}),
    \]
    and so (again using our hierarchy of choosing $A$ after $C_0$) it is enough to show that 
    \begin{equation}\label{A_hierarchy-5}
        \left|\sum_{n=1}^{2^{i}-1}\bigwedge_{m=1}^i X_{j_m} f_m^n\right|^2\ge C_0^{-2i^2+5.5}M^{2(i-1)}.
    \end{equation}
    But by Cauchy--Schwarz,
    \begin{align*}
        \left|\sum_{n=1}^{2^{i}-1}\bigwedge_{m=1}^i X_{j_m} f_m^n\right|^2 &\asymp \left|\sum_{n=1}^{2^{i}-1}\bigwedge_{m=1}^i X_{j_m} f_m^n\right|^2\left|\bigwedge_{m=2}^i X_{j_m}\phi \right|^2\ge \left|\left(\sum_{n=1}^{2^{i}-1}\bigwedge_{m=1}^i X_{j_m} f_m^n\right)\wedge \bigwedge_{m=2}^i X_{j_m}\phi \right|^2\\
        &= \left|\left(X_{j_1}\phi\wedge \bigwedge_{m=2}^i X_{j_m} \psi \right)\wedge \bigwedge_{m=2}^i X_{j_m}\phi \right|^2= \left|\bigwedge_{m=2}^i X_{j_m} \psi \wedge \bigwedge_{m=1}^i X_{j_m}\phi \right|^2.
    \end{align*}
    By the Cauchy--Binet formula, this is the determinant of a certain $(2i-1)\times (2i-1)$ matrix, which, by \eqref{NearOrtho}, is close to being block-diagonal: the upper-left $(i-1)\times (i-1)$ block consists of entries of size $O_{C_0}(M^2)$, the lower-right $i\times i$ block consists of entries of size $O_{C_0}(1)$, while the off-block-diagonal entries are of size $O(A^{-1})$ (by \eqref{NearOrtho}). Therefore, we may estimate the total determinant with the determinant of the block diagonal approximation, with error $O_{C_0}(A^{-2}M^{2(i-2)})$:
    \[
    \left|\bigwedge_{m=2}^i X_{j_m} \psi \wedge \bigwedge_{m=1}^i X_{j_m}\phi \right|^2=\left|\bigwedge_{m=2}^i X_{j_m} \psi \right|^2\left| \bigwedge_{m=1}^i X_{j_m}\phi \right|^2+O_{C_0}(A^{-2}M^{2(i-2)}).
    \]
    But by \eqref{It-2}, we have
    \begin{equation}\label{C_0_A_hierarchy-1}
        \left|\bigwedge_{m=2}^i X_{j_m} \psi \right|^2 \ge C_0^{-2i^2+6}M^{2(i-1)}.
    \end{equation}
    This completes the verification of \eqref{It-6} (again using the hierarchy that $C_0$ is chosen after $G$ and $A$ is chosen after $C_0$).
    
\vspace{0.1in}
\noindent \underline{Verification of \eqref{It-7}}:
    We first observe that for differential operators $W$ of degree at least 2, $W\phi$ dominates $W\psi$, and so we may approximate $\bigwedge_W W(\psi+\phi)$ by $\bigwedge_W W\phi$, where $W$ ranges over such operators.
    
    More precisely, for $W=X_{r_1,j_1}\cdots X_{r_m,j_m}$, where either $m=1$ and $r_1\ge 2$ or $2\le m\le s$, we have
    \begin{align*}
        W(\psi+\phi)&=W\tilde{\phi}+W(\phi-\tilde{\phi})+W\psi\\
        &=W\Big(U(\phi^0)\Big)+O_{C_0}(A^{2-{m^*}})+O_{C_0}(A^{-1})\\
        &=W\Big(U(\phi^0)\Big)+O_{C_0}(A^{-1}).
    \end{align*}
    (The second equation holds because our choice of ${m^*}\ge s^2$ allows us to use our bounds \eqref{G-8-again-again} and \eqref{It-4} on $\|\phi-\tilde{\phi}\|_{C^{{m^*}}}$ and $\|\nabla^2\psi\|_{C_A^{{m^*}-2}}$). Many applications of the Leibniz rule tell us that $W\Big(U(\phi^0)\Big)-U(W\phi^0)$ is a linear combination of derivatives of $U$ times $\phi^0$ or derivatives of $\phi^0$, so from $\|\nabla U\|_{C^{{m^*}}}\lesssim_{N_0}A^{-1}$, we have
    \[
    W\Big(U(\phi^0)\Big)-U(W\phi^0)=O_{N_0}(A^{-1})
    \]
    and consequently
    \[
    W(\psi+\phi)=U(W\phi^0)+O_{N_0}(A^{-1}).
    \]
    Therefore, we have
    \begin{equation}\label{higher-deriv-approx}
        \bigwedge_{\substack{W=X_{r_1,j_1}\cdots,X_{r_m,j_m} \\ m=1 \mbox{ and }r_1\ge 2, \mbox{ or } 2\le m\le s}}W(\psi+\phi)=\omega+O_{N_0}(A^{-1}),
    \end{equation}
    where
    \[
    \omega\coloneqq \bigwedge_{\substack{W=X_{r_1,j_1}\cdots,X_{r_m,j_m} \\ m=1 \mbox{ and }r_1\ge 2, \mbox{ or } 2\le m\le s}}U(W\phi^0).
    \]
    Since $U$ is an isometry, and $\phi^0$ has the freeness property \eqref{firstfree}, we have
    \begin{equation}\label{omega_is_free}
    |\omega|= \left|\bigwedge_{\substack{W=X_{r_1,j_1}\cdots,X_{r_m,j_m} \\ m=1 \mbox{ and }r_1\ge 2, \mbox{ or } 2\le m\le s}}W\phi^0\right|\asymp_G 1.
    \end{equation}
    Therefore, by \eqref{It-1}, \eqref{C_0_hierarchy-1}, \eqref{higher-deriv-approx} and our hierarchy that $C_0$ is chosen after $G$ and that $A$ is chosen after $C_0$ and $N_0$, to verify \eqref{It-7} it is enough to verify
    \begin{equation}\label{C_0_A_hierarchy-2}
        \left|\bigwedge_{i=1}^kX_i(\psi+\phi)\wedge \omega \right|\gtrsim  C_0^{-2k^2-5k+3.3}M^{k}.
    \end{equation}
    By Cauchy--Schwarz, and
    \[
    \left|\bigwedge_{i=1}^k X_iP_{(\le N_0)}\psi\right|\le \prod_{i=1}^k\left|X_iP_{(\le N_0)}\psi\right|\stackrel{\mathclap{\eqref{lp-1}}}{\lesssim}_{G} \left\|\nabla\psi\right\|_{C^0}^k\le C_0^{k}M^{k},
    \]
    (we used \eqref{lp-1} with $l=1$), we see (using our hierarchy that $C_0$ is chosen after $G$) that it is enough to verify
    \begin{equation}\label{C_0-hierarchy-5}
        \bigg< \bigwedge_{i=1}^k X_iP_{(\le N_0)}\psi\wedge \omega,\bigwedge_{i=1}^k X_i(\psi+\phi)\wedge \omega \bigg>\gtrsim C_0^{-2k^2-4k+3.7}M^{2k}.
    \end{equation}
    Since all the components of $\omega$ are orthogonal to the vectors $X_iP_{(\le N_0)}\psi$, we can use Cauchy--Binet twice to see that the left-hand side is equal to
    \[
    \bigg< \bigwedge_{i=1}^k X_iP_{(\le N_0)}\psi,\bigwedge_{i=1}^k X_i(\psi+\phi) \bigg>|\omega|^2.
    \]
    Using that for $1\le i,j\le k$ we have
    \begin{align}\label{N_0_hierarchy-7}
    \begin{aligned}
        X_iP_{(\le N_0)}\psi\cdot X_j(\psi+\phi)&=X_i\psi\cdot X_j\psi+X_i\psi\cdot X_j\phi\\
        &\stackrel{\mathclap{\eqref{NearOrtho}}}{=}~X_i\psi\cdot X_j\psi-X_iP_{(>N_0)}\psi\cdot X_j(\psi+\phi)+O(A^{-1}),
        \\ &\stackrel{\mathclap{\eqref{It-1},\eqref{C_0_hierarchy-1}, \eqref{lp-5}}}{=}\quad\quad\quad X_i\psi\cdot X_j\psi+ O_G\left(\frac{C_0}{AN_0}\right)\cdot (O(C_0M)+O(C_0))+O(A^{-1}),
        \\ &=X_i\psi\cdot X_j\psi+ O(C_0A^{-1}M),
    \end{aligned}
    \end{align}
    (where in the penultimate equation we used \eqref{lp-5} with $j=1$, $l=2$, and in the last equation we used our hierarchy that $N_0$ is chosen after $C_0$), we see using Cauchy--Binet twice that
    \[
    \bigg< \bigwedge_{i=1}^kX_iP_{(\le N_0)}\psi,\bigwedge_{i=1}^k X_i(\psi+\phi) \bigg> =\left|\bigwedge_{i=1}^k X_i\psi\right|^2+O_{C_0}(A^{-1}M^{2k-1}).
    \]
    But, from \eqref{It-2}, we have
    \begin{equation}\label{C_0_A_hierarchy-3}
        \left|\bigwedge_{i=1}^k X_i\psi\right|^2\ge C_0^{-2k^2-4k+4}M^{2k}.
    \end{equation}
    The claim \eqref{C_0-hierarchy-5} follows from the above, \eqref{omega_is_free}, and our hierarchy of choosing $C_0$ after $G$ and $A$ after $C_0$.

This concludes the proof of Lemma \ref{KeyIteration}.


\section{Construction of the embedding}\label{sec:applyiteration}
\setcounter{equation}{0}

One obtains the following proposition by repeating Proposition \ref{KeyIteration2} a finite number of times (see also Claim 5.4 of \cite{tao2021embedding} for an analogous statement for $\bbh^3$).

\begin{proposition}[finite iteration]\label{FiniteIteration}
Let $0<\varepsilon\le 1/A$, and let $M_1\le M_2$ be integers. One can find maps $\phi_m:G\to\bbr^{128\cdot 23^{n_h}}$ for $M_1\le m\le M_2$ obeying the following bounds, where $\phi_{(\ge m)}:G\to\bbr^{128\cdot 23^{n_h}}$ is defined by
\[
\phi_{(\ge m)}\coloneqq \sum_{m\le m'\le M_2}A^{-\varepsilon(m'-m)}\phi_{m'}.
\]
\begin{enumerate}[leftmargin=*]
    \item (smoothness at scale $A^m$) For all $M_1\le m\le M_2$, we have
    \begin{align}
        \|\phi_m\|_{C_{A^m}^{s^2+s+1}}&\le C_0 A^m \label{Finite-1}\\
        \|\nabla^2 \phi_{(\ge m)}\|_{C^{s^2+s-1,{2/3}}_{A^{m}}}&\le C_0 A^{-m}.\label{Finite-2}
    \end{align}
    \item (orthogonality) One has, for all $M_1\le m\le M_2$,
    \begin{equation}\label{Finite-3}
        \sum_{m'>m}A^{-\varepsilon(m'-m)}B(\phi_m,\phi_{m'})=0.
    \end{equation}
    \item (nondegeneracy) For all $p\in G$ and $M_1\le m\le M_2$, we have the estimates
    \begin{equation}\label{Finite-4}
        |X_i\phi_m(p)|\ge C_0^{-1},\quad i=1,\cdots,k,
    \end{equation}
    \begin{equation}\label{Finite-5}
        \left|\bigwedge_{q=1}^i X_{j_q}\phi_{(\ge m)}(p)\right|\ge C_0^{-i^2-i+2}\prod_{q=1}^i |X_{j_q}\phi_{(\ge m)}(p)|,\quad 2\le i\le k,~1\le j_1<\cdots<j_i\le k,
    \end{equation}
    \begin{equation}\label{Finite-6}
        \left|\bigwedge_{\substack{W=X_{r_1,j_1}\cdots X_{r_i,j_i}  \\  1\le i\le s, (r_1,j_1)\preceq \cdots \preceq (r_{i},j_{i})}}W\phi_{(\ge m)}(p)\right|\ge C_0^{-2k^2-7k+3}A^{-m\sum_{j=2}^s (j-1)\binom{n+j-1}{j}}\prod_{i=1}^k|X_i\phi_{(\ge m)}(p)|.
    \end{equation}
\end{enumerate}
\end{proposition}
\begin{proof}
We prove this by induction on $M_2-M_1$. Note that the statement is invariant under rescaling, so we may assume $M_1=0$. The case $M_2=0$ follows directly from Proposition \ref{FirstStep}, so we may assume $M_2>0$. By the inductive hypothesis applied to $M_1'=1$ and $M_2'=M_2$, we may find functions $\phi_m$, $1\le m\le M_2$, so that 
\begin{enumerate}[leftmargin=*]
    \item (smoothness at scale $A^n$) For all $1\le m\le M_2$, we have
    \begin{align}
        \|\phi_m\|_{C_{A^m}^{s^2+s+1}}&\le C_0 A^m\label{Finite-7}\\
        \|\nabla^2 \phi_{(\ge m)}\|_{C^{s^2+s-1,{2/3}}_{A^m}}&\le C_0 A^{-m}\label{Finite-8}
    \end{align}
    \item (orthogonality) One has, for all $1\le m\le M_2$,
    \begin{equation}\label{Finite-9}
        \sum_{m'>m}A^{-\varepsilon(m'-m)}B(\phi_m,\phi_{m'})=0.
    \end{equation}
    \item (nondegeneracy) For all $p\in G$ and $1\le m\le M_2$, we have the estimates
    \begin{equation}\label{Finite-10}
        |X_i\phi_m(p)|\ge C_0^{-1},\quad i=1,\cdots,k,
    \end{equation}
    \begin{equation}\label{Finite-11}
        \left|\bigwedge_{q=1}^i X_{j_q}\phi_{(\ge m)}(p)\right|\ge C_0^{-i^2-i+2}\prod_{q=1}^i |X_{j_q}\phi_{(\ge m)}(p)|,\quad 2\le i\le k,~1\le j_1<\cdots<j_i\le k,
    \end{equation}
    \begin{equation}\label{Finite-12}
        \left|\bigwedge_{\substack{W=X_{r_1,j_1}\cdots X_{r_i,j_i}  \\  1\le i\le s, (r_1,j_1)\preceq \cdots \preceq (r_{i},j_{i})}}W\phi_{(\ge m)}(p)\right|\ge C_0^{-2k^2-7k+3}A^{-m\sum_{j=2}^s (j-1)\binom{n+j-1}{j}}\prod_{i=1}^k|X_i\phi_{(\ge m)}(p)|.
    \end{equation}
\end{enumerate}
We now verify the hypotheses of Proposition \ref{KeyIteration2} for
\[
\psi\coloneqq A^{-\varepsilon}\phi_{(\ge 1)}=\sum_{1\le m\le M_2}A^{-\varepsilon m}\phi_m,\quad \mathrm{with~} M=\left(\sum_{1\le m\le M_2}A^{-2\varepsilon m}\right)^{1/2},~{m^*}=s^2+s+1, ~\alpha=\frac 23.
\]
We have $M\ge A^{-\varepsilon}\ge A^{-1/A}\ge \frac 12$, so we have $M\ge C_0^{-1}$.

\vspace{0.1in}
\noindent \underline{Verification of \eqref{It-1-again}}:
    From \eqref{Finite-9} and \eqref{Finite-10} we have, for $i=1,\cdots,k$,
    \[
    |X_i\psi|^2=\sum_{1\le m\le M_2}A^{-2\varepsilon m}|X_i\phi_m|^2\ge \sum_{1\le m\le M_2}A^{-2\varepsilon m}C_0^{-2}=C_0^{-2}M^2,
    \]
    and from \eqref{Finite-7} we have, for $i=1,\cdots,k$,
    \[
    |X_i\psi|^2=\sum_{1\le m\le M_2}A^{-2\varepsilon m}|X_i\phi_m|^2\le \sum_{1\le m\le M_2}A^{-2\varepsilon m}C_0^{2}=C_0^{2}M^2.
    \]
\vspace{0.1in}
\noindent \underline{Verification of \eqref{It-2-again}}:
    From \eqref{Finite-11} we have, for $2\le i\le k$ and $1\le j_1<\cdots<j_i\le k$,
    \begin{align*}
        \left|\bigwedge_{q=1}^i X_{j_q}\psi\right|&=A^{-i\varepsilon}\left|\bigwedge_{q=1}^i X_{j_q}\phi_{(\ge 1)}\right|\ge A^{-i\varepsilon}C_0^{-i^2-i+2}\prod_{q=1}^i\left| X_{j_q}\phi_{(\ge 1)}\right|=C_0^{-i^2-i+2}\prod_{q=1}^i\left| X_{j_q}\psi\right|.
    \end{align*}
\vspace{0.1in}
\noindent \underline{Verification of \eqref{It-3-again}}:
    From \eqref{Finite-12} we have
    \begin{align*}
        \left|\bigwedge_{\substack{W=X_{r_1,j_1}\cdots X_{r_m,j_m}  \\  1\le m\le s, (r_1,j_1)\preceq \cdots \preceq (r_{m},j_{m})}}W\psi(p)\right|&=A^{-\sum_{m=1}^s \binom{n+m-1}{m}\varepsilon}\left|\bigwedge_{\substack{W=X_{r_1,j_1}\cdots X_{r_m,j_m}  \\  1\le m\le s, (r_1,j_1)\preceq \cdots \preceq (r_{m},j_{m})}}W\phi_{(\ge 1)}\right|\\
        &\ge A^{-\sum_{m=1}^s \binom{n+m-1}{m}\varepsilon}C_0^{-2k^2-7k+3}A^{-m\sum_{j=2}^s (j-1)\binom{n+j-1}{j}}\prod_{i=1}^k|X_i\phi_{(\ge m)}(p)|\\
        &= A^{-(\sum_{m=1}^s \binom{n+m-1}{m}-k)\varepsilon}C_0^{-2k^2-7k+3}A^{-m\sum_{j=2}^s (j-1)\binom{n+j-1}{j}}\prod_{i=1}^k|X_i\psi(p)|\\
        &\ge C_0^{-2k^2-7k+2}A^{-m\sum_{j=2}^s (j-1)\binom{n+j-1}{j}}\prod_{i=1}^k|X_i\psi(p)|.
    \end{align*}
\vspace{0.1in}
\noindent \underline{Verification of \eqref{It-4-again}}:
    From \eqref{Finite-8} we have
    \[
    \|\nabla^2 \psi\|_{C_A^{s^2+s-1,{2/3}}}=A^{-\varepsilon}\|\nabla^2 \phi_{(\ge 1)}\|_{C^{s^2+s-1,{2/3}}_{A}}\le A^{-\varepsilon}C_0 A^{-1}\le  C_0A^{-1}.
    \]
Hence, $\psi$ and $M$ satisfy the assumptions of Proposition \ref{KeyIteration2}, and so there exists a function $\phi_0:G\to\bbr^{128\cdot 23^{n_h}}$ that satisfies the following:
\begin{enumerate}[leftmargin=*]
    \item (regularity at scale 1) We have
    \begin{equation}\label{newholder-prime}
        \|\phi_0\|_{C^{s^2+s+1,{2/3}}}\lesssim_G 1.
    \end{equation}
    \item (orthogonality) We have
    \begin{equation}\label{neworthog-prime}
        B(\phi_0,\psi)=0.
    \end{equation}
    \item (nondegenerate first derivatives) For any $p\in G$, we have
    \begin{equation}\label{newnondeg-prime}
        |X_i\phi_0(p)|\gtrsim_G 1,\quad i=1,\cdots,k,
    \end{equation}
    \begin{equation}\label{It-1-again-lower-scale-prime}
        C_0^{-1}\sqrt{M^2+1}\le |X_i(\psi+\phi_0)(p)|\le C_0 \sqrt{M^2+1},\quad i=1,\cdots,k,
    \end{equation}
    \begin{equation}\label{It-2-again-lower-scale-prime}
        \left|\bigwedge_{m=1}^i X_{j_m}(\psi+\phi_0)(p)\right|\ge C_0^{-i^2-i+2} \prod_{m=1}^i \left|X_{j_m}(\psi+\phi_0)(p)\right|
    \end{equation}
    for $2\le i\le k$ and distinct $1\le j_1<\cdots<j_i\le k$.
    \item (locally free embedding) For any $p\in G$, we have
    \begin{equation}\label{It-3-again-lower-scale-prime}
        \left|\bigwedge_{\substack{W=X_{r_1,j_1}\cdots X_{r_m,j_m}  \\  1\le m\le s, (r_1,j_1)\preceq \cdots \preceq (r_{m},j_{m})}}W(\psi+\phi_0)(p)\right|\ge C_0^{-2k^2-7k+3}A^{-\sum_{j=2}^s (j-1)\binom{n+j-1}{j}}\prod_{m=1}^i \left|X_{j_m}(\psi+\phi_0)(p)\right|.
    \end{equation}
    \item (H\"older regularity at scale $1$) We have
    \begin{equation}\label{It-4-again-lower-scale-prime}
        \|\nabla^2 (\psi+\phi_0)\|_{C^{s^2+s-1,{2/3}}}\le C_0.
    \end{equation}
\end{enumerate}
Note that $\phi_{(\ge 0)}=\psi+\phi_0$. To verify that the larger family of maps $\{\phi_m\}_{0\le m\le M_2}$ satisfies the properties \eqref{Finite-1} to \eqref{Finite-6}, we need only verify these properties for $m=0$. But, for $m=0$, \eqref{Finite-1} follows directly from \eqref{newholder-prime} and \eqref{Finite-4} follows from \eqref{newnondeg-prime}, while \eqref{Finite-2}, \eqref{Finite-3}, \eqref{Finite-5}, and \eqref{Finite-6} are precisely conditions \eqref{It-4-again-lower-scale-prime}, \eqref{neworthog-prime}, \eqref{It-2-again-lower-scale-prime}, and \eqref{It-3-again-lower-scale-prime}, respectively.
\end{proof}

By taking the limit $M_1\to-\infty$, $M_2\to\infty$, we now obtain a full set of lacunary maps. (See also Theorem 4.1 of \cite{tao2021embedding} for an analogous statement for $\bbh^3$.)

\begin{theorem}[Maps oscillating at lacunary scales]\label{lacunary}
Let $0<\varepsilon\le 1/A$. Then one can find a map $\phi_m:G\to \bbr^{128\cdot 23^{n_h}}$ for each integer $m$ obeying the following bounds:
\begin{itemize}[leftmargin=*]
    \item (smoothness at scale $A^n$) For all integers $m$, one has
    \begin{equation*}
        \|\phi_m\|_{C^{s^2+s}_{A^m}}\lesssim_{C_0} A^m.
    \end{equation*}
    In particular, we have
    \begin{equation}\label{D-5}
        X_{r,i}\phi_m(p)=O_{C_0}(A^{-m(r-1)}),\quad \mbox{for all }r,i.
    \end{equation}
    \item (orthogonality) For all integers $m$, one has
    \begin{equation*}
        \sum_{m'>m} A^{-\varepsilon(m'-m)}B(\phi_m,\phi_{m'})=0
    \end{equation*}
    identically on $G$. (By \eqref{D-5} this sum is absolutely convergent.)
    
    \item (nondegeneracy and immersion) For all integers $m$ and all $p\in G$, one has
    \begin{equation*}
        |X_i\phi_m(p)|\gtrsim_{C_0}1
    \end{equation*}
    and
    \begin{equation}\label{D-8}
        \left|\bigwedge_{\substack{W=X_{r_1,j_1}\cdots X_{r_i,j_i}  \\  1\le i\le s, (r_1,j_1)\preceq \cdots \preceq (r_{i},j_{i})}}W\phi_{(\ge m)}(p)\right|\gtrsim_{C_0} \prod_{\substack{W=X_{r_1,j_1}\cdots X_{r_i,j_i}  \\  1\le i\le s, (r_1,j_1)\preceq \cdots \preceq (r_{i},j_{i})}}\left|W\phi_{(\ge m)}(p)\right|,
    \end{equation}
    where
    \[
    \phi_{(\ge m)}(p)=\sum_{m'\ge m} A^{-\varepsilon(m'-m)}\phi_{m'}.
    \]
\end{itemize}

\end{theorem}

\begin{proof}
For each $M\in\bbn$, we can apply Proposition \ref{FiniteIteration} to find $\phi^M_m:G\to\bbr^{128\cdot 23^{n_h}}$ for $-M\le m\le M$ obeying the following bounds, where $\phi^M_{(\ge m)}:G\to\bbr^{128\cdot 23^{n_h}}$ is defined by
\[
\phi^M_{(\ge m)}\coloneqq \sum_{m\le m'\le M_2}A^{-\varepsilon(m'-m)}\phi^M_{m'}.
\]
\begin{enumerate}[leftmargin=*]
    \item (smoothness at scale $A^m$) For all $-M\le m\le M$, we have
    \begin{equation}\label{D-20}
        \|\phi^M_m\|_{C_{A^m}^{s^2+s+1}}\le C_0 A^m.
    \end{equation}
    \item (orthogonality) One has, for all $-M\le m\le M$,
    \begin{equation*}
        \sum_{m'>m}A^{-\varepsilon(m'-m)}B(\phi^M_m,\phi^M_{m'})=0.
    \end{equation*}
    \item (nondegeneracy) For all $p\in G$ and $-M\le m\le M$, we have the estimates
    \begin{equation*}
        |X\phi^M_m(p)|\ge C_0^{-1}
        \end{equation*}
and
        \begin{equation}\label{D-23}
            \left|\bigwedge_{\substack{W=X_{r_1,j_1}\cdots X_{r_i,j_i}  \\  1\le i\le s, (r_1,j_1)\preceq \cdots \preceq (r_{i},j_{i})}}W\phi^M_{(\ge m)}(p)\right|\ge C_0^{-2k^2-7k+3}A^{-m\sum_{j=2}^s (j-1)\binom{n+j-1}{j}}\prod_{i=1}^k|X_i\phi^M_{(\ge m)}(p)|.    
        \end{equation}

\end{enumerate}
From \eqref{D-20} and \eqref{D-23} we see that
\[
\left|\bigwedge_{\substack{W=X_{r_1,j_1}\cdots X_{r_i,j_i}  \\  1\le i\le s, (r_1,j_1)\preceq \cdots \preceq (r_{i},j_{i})}}W\phi^M_{(\ge m)}(p)\right|\ge C_0^{-2k^2-7k+3-\sum_{j=2}^s (j-1)\binom{n+j-1}{j}}\prod_{\substack{W=X_{r_1,j_1}\cdots X_{r_i,j_i}  \\  1\le i\le s, (r_1,j_1)\preceq \cdots \preceq (r_{i},j_{i})}}|W\phi^M_{(\ge m)}(p)|.    
\]
For each $m\in \bbz$, the sequence $\{\phi_m^M\}_{M\ge |m|}$ is bounded in the $C_{A^m}^{s^2+s+1}$ norm, so by the Arzel\`a-Ascoli theorem, one can find a subsequence of $\{M_k\}$ such that $\{\phi_m^{M_k}\}_{M_k\ge |m|}$ locally converges in the $C^{s^2+s}$ topology, say to $\phi_m$, for every $m\in \bbz$. It now readily follows that these $\phi_m$ satisfy the above properties.
\end{proof}

Once we have this lacunary family guaranteed by Theorem \ref{lacunary}, we can construct a function $\Phi_1:G\to \bbr^{128\cdot 23^{n_h}}$ as
\[
\Phi_1(p)\coloneqq \sum_{m=-\infty}^\infty A^{-\varepsilon m}(\phi_m(p)-\phi_m(0)).
\]
Then $\Phi_1$ is ``almost'' a bi-Lipschitz embedding of $(G,d_{G}^{1-\varepsilon})$ into $\bbr^{128\cdot 23^{n_h}}$:
\begin{proposition}\label{AlmostLipschitz}
The map $\Phi_1:G\to \bbr^{128\cdot 23^{n_h}}$ satisfies the following estimates.

\begin{enumerate}[leftmargin=*]
    \item (H\"older upper bound)
\begin{equation}\label{partial-ub}
|\Phi_1(p)-\Phi_1(p')|\lesssim_A \varepsilon^{-1/2} d_{G}(p,p')^{1-\varepsilon},\quad \mathrm{for~all~} p,p'\in G.    
\end{equation}
\item (partial H\"older lower bound) For $p,p'\in G$ so that $A^{n_0}A^{-1/(s+1)}\le d_G(p,p')\le 2A^{n_0}A^{-1/(s+1)}$ for some integer $n_0$, we have
\begin{equation}\label{partial-lb}
    |\Phi_1(p)-\Phi_1(p')|\gtrsim_A d_{G}(p,p')^{1-\varepsilon}.
\end{equation}
\end{enumerate}

\end{proposition}

\begin{proof}
\begin{enumerate}[leftmargin=*]
    \item Let $p,p'\in G$. By translating and rescaling, we may assume $p'=0$ and $A^{-1}\le d_G(p,0)\le 1$. We introduce the low-frequency component
    \[
    \Psi(q)\coloneqq \sum_{m=0}^\infty A^{-\varepsilon m}(\phi_m(q)-\phi_m(0)),\quad q\in G.
    \]
    Then 
    \begin{equation}\label{lowfreq-approx}
    \Phi_1(q)=\Psi(q)+O_{C_0}(A^{-1}),
    \end{equation}
    so it will be enough to show
    \[
    |\Psi(p)|\lesssim_{A} \varepsilon^{-1/2}.
    \]
    As $d_G(p,0)\le 1$, there exists a horizontal curve $\gamma$ in $G$ from $0$ to $p$ of length $\le 1$. Therefore, we have
    \[
    |\Psi(p)|=|\Psi(p)-\Psi(0)|\le 1\cdot \||\nabla \Psi|\|_{L^\infty(G)},
    \]
    But by the orthogonality statement of Theorem \ref{lacunary}, we have
    \[
    |X_i \Psi|=|\sum_{m=0}^\infty A^{-\varepsilon m}X_i\phi_m|=\left(\sum_{m=0}^\infty A^{-2\varepsilon m}|X_i\phi_m|^2\right)^{1/2}\asymp_{C_0}M,\quad i=1,\cdots, k,
    \]
    where
    \[
    M\coloneqq \left(\sum_{n=0}^\infty A^{-2\varepsilon n}\right)^{1/2}=\left(\frac{1}{1-A^{-2\varepsilon }}\right)^{1/2}\asymp \frac{1}{\sqrt{\varepsilon \log A}}\lesssim_A  \varepsilon^{-1/2},
    \]
    so we conclude
    \begin{equation}\label{A_hierarchy-6}
        |\Psi(p)|\le \|\nabla \Psi\|_{L^\infty \ell^2}\lesssim_{C_0} M\lesssim_{A} \varepsilon^{-1/2},
    \end{equation}
    as desired (we just used the hierarchy of choosing $A$ after $C_0$).
    
    \item 
    
    Let $p,p'\in G$ be so that $A^{n_0}A^{-1/(s+1)}\le d_G(p,p')\le 2A^{n_0}A^{-1/(s+1)}$ for some integer $n_0$. Again, by translating and rescaling, we may assume $p'=0$ and $n_0=0$, i.e. $A^{-1/(s+1)}\le d_G(p,0)\le 2A^{-1/(s+1)}$. Writing $p=\exp\left(\sum_{r=1}^s\sum_{i=1}^{k_r}x_{r,i}X_{r,i}\right)$, we have
    \[
    \sum_{r=1}^s\sum_{j=1}^{k_r}|x_{r,j}|^{1/r}\asymp_G d_G(p,0)\asymp A^{-1/(s+1)}.
    \]
    Equivalently, we have $|x_{r,j}|\lesssim_G A^{-r/(s+1)}$ for all $r,j$, and there exists some pair $(r,j)$ such that $|x_{r,j}|\asymp_G A^{-r/(s+1)}$.
    
    By Taylor expansion
    \[
    \Psi(p)=\Psi(0)+\sum_{m=1}^s \frac{1}{m!}\left(\sum_{r=1}^s\sum_{j=1}^{k_r}x_{r,j}X_{r,j}\right)^m\Psi(0)+O_{C_0}(A^{- 1})
    \]
    (note that this is where we used the $C^{s(s+1)}$-regularity of $\Psi$), and in light of \eqref{lowfreq-approx} and $\Psi(0)=0$, we have
    \[
    \Phi_1(p)=\sum_{m=1}^s \frac{1}{m!}\left(\sum_{r=1}^s\sum_{j=1}^{k_r}x_{r,j}X_{r,j}\right)^m\Psi(0)+O_{C_0}(A^{- 1}).
    \]
    This expression contains many terms of the form $X_{r_1,j_1}\cdots X_{r_m,j_m}\Psi$, where the $(r_1,j_1),$ $\cdots,$ $(r_m,j_m)$ are arbitrary with $m\le s$ and are often unordered. In order to use the freeness property \eqref{D-8}, it would be necessary to modify the above Taylor expansion formula so that the only differential operators acting on $\Psi$ are the ones of the form $X_{r_1,j_1}\cdots X_{r_m,j_m}$ where $(r_1,j_1)\preceq \cdots \preceq (r_{m},j_{m})$.
    
    This modification is possible once we note that for any permutation $\pi$ of $\{1,\cdots,m\}$ we can express $X_{r_1,j_1}\cdots X_{r_m,j_m}-X_{r_{\pi(1)},j_{\pi(1)}}\cdots X_{r_{\pi(m)},j_{\pi(m)}}$ as a linear combination of differential operators $X_{r'_1,j'_1}\cdots X_{r'_{m'},j'_{m'}}$ where $\sum_{i=1}^{m'}r'_i=\sum_{i=1}^m r_i$ and $(r'_1,j'_1)\preceq \cdots \preceq (r'_{m'},j'_{m'})$. This can be proven by a simple induction argument on $m$ using the fact that $[X_{r,j},X_{r',j'}]\in V_{r+r'}$ is a linear combination of $X_{r+r',i}$ for $i=1,\cdots,k_{r+r'}$ if $r+r'\le s$, and $[X_{r,j},X_{r',j'}]=0$ if $r+r'>s$. Applying this fact to the above Taylor expansion formula and keeping track of the degrees, we obtain the following modified Taylor expansion formula:
    \begin{align}\label{modified-Taylor}
    \begin{aligned}
    \Psi(p)=&\sum_{r=1}^s\sum_{j=1}^{k_r}(x_{r,j}+p_{r,j})X_{r,j}\Psi(0)\\
    &\qquad\qquad+ \sum_{m=2}^s\sum_{(r_1,j_1)\preceq \cdots \preceq (r_{m},j_{m})}p_{(r_1,j_1), \cdots , (r_{m},j_{m})}X_{r_1,j_1}\cdots X_{r_m,j_m}\Psi(0)+O_{C_0}(A^{- 1})
    \end{aligned}
    \end{align}
    where $p_{r,j}$ is a homogeneous polynomial of degree $r$ where each monomial is a product of at least two terms, each of the form $x_{r',j'}$ with $r'<r$ (recall that we define the homogeneous degree by assigning weight $r$ to $x_{r,j}$), and $p_{(r_1,j_1), \cdots , (r_{m},j_{m})}$ is a homogeneous polynomial of degree $\sum_{i=1}^m r_i$.
    
    By the freeness property \eqref{D-8}, each term in \eqref{modified-Taylor} (except for the error term) may serve as a lower bound for the entire sum, up to multiplicative constants. Thus, it suffices to show that there exists $r,j$ such that $x_{r,j}+p_{r,j}$ has non-negligible size.
    
        More precisely, the freeness property \eqref{D-8} tells us that
    \[
        \left|\bigwedge_{\substack{W=X_{r_1,j_1}\cdots X_{r_m,j_m}  \\  1\le m\le s, (r_1,j_1)\preceq \cdots \preceq (r_{m},j_{m})}}W\Psi(p)\right|\asymp_{C_0}  \prod_{\substack{W=X_{r_1,j_1}\cdots X_{r_m,j_m}  \\  1\le m\le s, (r_1,j_1)\preceq \cdots \preceq (r_{m},j_{m})}}\left|W\Psi(p)\right|,
    \]
    which immediately gives us, for each $(r_0,j_0)$, the following control on the main term of the Taylor expansion:
    \begin{align*}
    \left|\sum_{r=1}^s\sum_{j=1}^{k_r}(x_{r,j}+p_{r,j})X_{r,j}\Psi(0)+ \sum_{m=2}^s\sum_{(r_1,j_1)\preceq \cdots \preceq (r_{m},j_{m})}p_{(r_1,j_1), \cdots , (r_{m},j_{m})}X_{r_1,j_1}\cdots X_{r_m,j_m}\Psi(0)\right|\qquad\quad&\\
    \gtrsim_{C_0} \left|(x_{r_0,j_0}+p_{r_0,j_0})X_{r_0,j_0}\Psi(0)\right|.&
    \end{align*}
    So it remains to single out a pair $(r_0,j_0)$ such that the right-hand side is large enough. Indeed, recall that we have $|x_{r,j}|\lesssim_G A^{-r/(s+1)}$ for all $r,j$, and $|x_{r,j}|\asymp_G A^{-r/(s+1)}$ for some pair $(r,j)$. Therefore, there exists some $(r_0,j_0)$ such that
    \begin{equation}\label{A_hierarchy-7}
        |x_{r,j}|< A^{-r/(s+1)-0.5} \mbox{ for all } r<r_0,1\le j\le k_r,\quad |x_{r_0,j_0}|\ge A^{-r_0/(s+1)-0.5}
    \end{equation}
    (we have just used the hierarchy that $A$ is chosen after $G$). Then, as $p_{r_0,j_0}$ is homogeneous of degree $r_0$ and consist of monomials which are the product of at least two terms, we must have
    \begin{equation}\label{A_hierarchy-8}
        |p_{r_0,j_0}|\lesssim_G A^{-r_0/(s+1)-1}, \quad \mbox{or}\quad |p_{r_0,j_0}|\le \frac 12 A^{-r_0/(s+1)-0.5}
    \end{equation}
    (we have again just used the hierarchy that $A$ is chosen after $G$).
    Hence,
    \[
    |x_{r_0,j_0}+p_{r_0,j_0}|\ge |x_{r_0,j_0}|-|p_{r_0,j_0}|\ge \frac 12 A^{-r_0/(s+1)-0.5},
    \]
    and we see that $\left|(x_{r_0,j_0}+p_{r_0,j_0})X_{r_0,j_0}\Psi(0)\right|\gtrsim_{C_0}A^{-r_0/(s+1)-0.5}$. In conclusion, we have $|\Psi(p)|\gtrsim_{C_0}A^{-r_0/(s+1)-0.5}$, or 
    \begin{equation}\label{A_hierarchy-9}
        |\Psi(p)|\gtrsim_{A}1
    \end{equation}
    (we just used the hierarchy of choosing $A$ after $C_0$).
\end{enumerate}
\end{proof}

Now, with Proposition \ref{AlmostLipschitz} in hand, we are ready to prove the main theorem.

\begin{proof}
Recall that the map $\Phi_1:G\to \bbr^{128\cdot 23^{n_h}}$ satisfies the H\"older upper bound \eqref{partial-ub} and the partial H\"older lower bound \eqref{partial-lb}, and that $A$ is a dyadic number, say $A=2^a$, $a\in \mathbb{N}$. By precomposing $\Phi_1$ with the scaling maps $\delta_{2^{-m+1}}$ and rescaling by $2^{(m-1)(1-\varepsilon)}$, $m=1,\cdots,a$, we obtain mappings $\Phi_m=2^{(m-1)(1-\varepsilon)}\Phi_1\circ \delta_{2^{-m+1}}:G\to \bbr^{128\cdot 23^{n_h}}$ that satisfy the same H\"older upper bound \eqref{partial-ub}:
\begin{equation}\label{partial-ub-full}
|\Phi_m(p)-\Phi_m(p')|\lesssim_A \varepsilon^{-1/2} d_{G}(p,p')^{1-\varepsilon},\quad \mathrm{for~all~} p,p'\in G,  
\end{equation}
and a different partial H\"older lower bound \eqref{partial-lb}:
\begin{equation}\label{partial-lb-full}
    |\Phi_m(p)-\Phi_m(p')|\gtrsim_A d_{G}(p,p')^{1-\varepsilon},
\end{equation}
for $p,p'\in G$ so that $2^{m-1} A^{n_0}A^{-1/(s+1)}\le d_G(p,p')\le 2^m A^{n_0}A^{-1/(s+1)}$ for some integer $n_0$.

We now obtain the full embedding $\Phi:G\to \bbr^{128\cdot 23^{n_h}\cdot a}$ by directly concatenating the mappings $\Phi_1,\cdots,\Phi_a$:
\begin{equation}\label{fullmapping}
    \Phi(p)\coloneqq \Big(\Phi_m(p)\Big)_{m=1}^{a},\quad p\in G.
\end{equation}
It remains to observe
\[
d_G(p,p')^{1-\varepsilon}\lesssim_A |\phi(p)-\phi(p')| \lesssim_A \varepsilon^{-1/2} d_G(p,p')^{1-\varepsilon},\quad p,p'\in G.
\]
The upper bound follows from \eqref{partial-ub-full}, and the lower bound follows by observing that for any given pair $p,p'\in G$, \eqref{partial-lb-full} applies to at least one coordinate.
\end{proof}

\bibliographystyle{myalpha}
\bibliography{bib}

\end{document}